\documentclass[11pt,a4paper]{article} 
\usepackage[margin=1.095in]{geometry}

\usepackage{euscript,amsmath,amsfonts,amssymb,amscd,amsthm,enumerate,hyperref}
\usepackage{comment}
\usepackage{graphicx}
\usepackage{color}

\newtheorem{theorem}{Theorem}[section]
\newtheorem{lemma}[theorem]{Lemma}
\newtheorem{corollary}[theorem]{Corollary}
\newtheorem{proposition}[theorem]{Proposition}
\theoremstyle{definition}
\newtheorem{definition}[theorem]{Definition}
\theoremstyle{remark}

\newcommand{\HO}{\mathcal F}
\newcommand{\QHO}{\widetilde \HO}

\newcommand{\D}{\mathcal{D}}
\newcommand{\DI}{\mathcal{DI}}

\newcommand{\GTH}{\widehat{\mathbb{GT}}}
\newcommand{\St}{\mathcal R}
\newcommand{\E}{\mathbb E}
\renewcommand{\i}{{\mathbf i}}
\newcommand{\Cr}{\mathbf{Cr}}
\renewcommand{\P}{\mathcal P}
\newcommand{\Y}{\mathbb Y}
\renewcommand{\S}{\mathbb S}
\newcommand{\eps}{\varepsilon}
\renewcommand{\Pr}{\mathbb P}
\newcommand{\GTP}{\mathbb{GT}^+}
\newcommand{\Mac}{\mathfrak{Mac}}
\renewcommand{\r}{\mathbf{r}}

\newcommand{\Jac}{\mathcal {J}}
\newcommand{\jr}{\mathfrak{j}}

\begin{document}                        


\title{General $\beta$ Jacobi corners process and the Gaussian Free Field.}

\author{Alexei Borodin\thanks{Department of Mathematics, Massachusetts Institute of Technology,
USA, and
 Institute for Information Transmission Problems of Russian Academy of Sciences,  Russia}
\footnote{e-mail: borodin@math.mit.edu} \quad Vadim Gorin$^*$\footnote{e-mail: vadicgor@gmail.com}}


\date{}






\maketitle

\begin{abstract}
We prove that the two--dimensional Gaussian Free Field describes the asymptotics of global
fluctuations of a multilevel extension of the general $\beta$ Jacobi random matrix ensembles.
 Our approach is based on the connection of the Jacobi
ensembles to a degeneration of the Macdonald processes that parallels the degeneration of the
Macdonald polynomials to the Heckman--Opdam hypergeometric functions (of type A). We also
discuss the $\beta\to\infty$ limit.
\end{abstract}



\tableofcontents



\section{Introduction}

\subsection{Preface}

The goal of this article is two--fold. First, we want to show how the Macdonald measures and
Macdonald processes \cite{BigMac}, \cite{BCGS}, \cite{BG}, \cite{BP2} (that generalize the Schur measures and
Schur processes of Okounkov and Reshetikhin \cite{Ok}, \cite{OR}) are related to classical general
$\beta$ ensembles of random matrices. The connection is obtained through a limit transition which
is parallel to the degeneration of the Macdonald polynomials to the Heckman--Opdam hypergeometric
functions. This connection brings certain new tools to the domain of random matrices, which we
further exploit.

Second, we extend the known Central Limit Theorems for the global fluctuations of (classical)
general $\beta$ random matrix ensembles to \emph{corners processes}, one example of which is the
joint distribution of spectra of the GUE--random matrix and its principal submatrices. We show
that the global fluctuations of certain general $\beta$ corners processes can be described via the
two--dimensional Gaussian Free Field. This suggests a new viewpoint on several known Central Limit
Theorems of random matrix theory: there is a unique two--dimensional universal limit object (the
Gaussian Free Field) and in different models one sees its various one--dimensional slices.

Let us proceed to a more detailed description of our model and results.
\subsection{The model}

The general $\beta$ ensemble of rank $N$ is the distribution on the set of $N$--tuples of reals
(``particles'' or ``eigenvalues'') $x_1<x_2<\dots<x_N$ with density (with respect to the Lebesgue
measure) proportional to
\begin{equation}
\label{eq_general_beta_intro}
 \prod_{1\le i<j\le N} (x_j-x_i)^\beta  \prod_{i=1}^N w(x_i),
\end{equation}
where $w(x)$ is the \emph{weight} of the ensemble. Perhaps, the most well-known case is $\beta=2$,
$w(x)=\exp(-x^2/2)$ (often called Gaussian or Hermite ensemble), which corresponds to the
eigenvalue distribution of the complex Hermitian matrix from the \emph{Gaussian Unitary Ensemble}
(GUE). Other well--studied weights are $w(x)=x^p e^{-x}$ on $\mathbb R_{>0}$, $p>-1$, known as the
Laguerre (or Wishart) ensemble, and $w(x)=x^p(1-x)^q$ on $(0,1)$, $p,q>-1$, known as the Jacobi (or
MANOVA) ensemble. These three weight functions correspond to \emph{classical} random matrix
ensembles because of their relation to classical orthogonal polynomials. In the present
article we focus on the Jacobi ensemble, but it is very plausible that all our results extend to
other ensembles by suitable limit transitions.


\smallskip

For $\beta=1,2,4$, and the three classical choices of $w(x)$ above, the distribution
\eqref{eq_general_beta_intro} appears as the distribution of eigenvalues of natural classes of
random matrices, see  Anderson--Guionnet--Zeitouni \cite{AGZ}, Forrester \cite{For}, Mehta
\cite{Me}. Here the parameter $\beta$ corresponds to the dimension of the base field over $\mathbb
R$, and one speaks about real, complex or quaternion matrices, respectively. There are also
different points of view on the density function \eqref{eq_general_beta_intro}, relating it to the
Coulomb log--gas, to the squared ground state wave function of the Calogero--Sutherland quantum
many--body system, or to random tridiagonal and unitary Hessenberg matrices. These viewpoints
naturally lead to considering $\beta>0$ as a continuous real parameter, see e.g.\ Forrester
\cite[Chapter 20 ``Beta Ensembles'']{ABF} and references therein.

\smallskip

 The above random matrix ensembles at $\beta=1,2,4$ come with an additional structure,
which is a natural coupling between the distributions \eqref{eq_general_beta_intro} with varying
number of particles $N$. In the case of the Gaussian Unitary Ensemble, take $M=[M_{ij}]_{i,j=1}^N$
to be a random Hermitian matrix with probability density proportional to $\exp\big(-{\rm
Trace}(M^2/2)\big)$. Let $x^k_1\le \dots\le x^k_k$, $k=1,\dots,N$, denote the set of (real)
eigenvalues of the top--left $k\times k$ corner $[M_{ij}]_{i,j=1}^k$. The joint probability
density of $(x^k_1,\dots,x^k_k)$ is given by \eqref{eq_general_beta_intro} with $\beta=2$,
$w(x)=\exp(-x^2/2)$. The eigenvalues satisfy the \emph{interlacing conditions} $x^j_i\le x^{j-1}_i
\le x^{j}_{i+1}$ for all meaningful values of $i$ and $j$.  (Although, the inequalities are not
strict in general, they are strict almost surely.) The joint distribution of the
$N(N+1)/2$--dimensional vector $\{x^j_i\}$, $i=1,\dots,j$, $j=1,\dots,N$ is known as the
\emph{GUE--corners} process (the term \emph{GUE--minors} process is also used), explicit formulas
for this distribution can be found in Gelfand--Naimark \cite{GN}, Baryshnikov \cite{Bar}, Neretin
\cite{Ner}, Johansson--Nordenstam \cite{JN}.

Similar constructions are available for the Hermite (Gaussian) ensemble with $\beta=1,4$. One can
notice that in the resulting formulas for the distribution of the corners process $\{x^j_i\}$,
$i=1,\dots,j$, $j=1,\dots,N$, the parameter $\beta$ enters in a simple way (see e.g.\
\cite[Proposition 1.1]{Ner}), which readily leads to the generalization of the definition of the
corners process to the case of general $\beta>0$, see Neretin \cite{Ner} and Okounkov--Olshanski
\cite[Section 4]{OO}.

From a different direction, if one considers the restriction of the corners process (both for
classical $\beta=1,2,4$ and for a general $\beta$) to two neighboring levels, i.e.\ the joint
distribution of vectors $(x^k_1,\dots,x^k_k)$ and $(x^{k-1}_1,\dots,x^{k-1}_{k-1})$, then one
finds
 formulas that are well-known in the theory of Selberg integrals. Namely, this distribution appears in the
Dixon--Anderson integration formula, see Dixon \cite{Dixon}, Anderson \cite{Anderson}, Forrester
\cite[chapter 4]{For}. More recently the same two--level distributions were studied by Forrester
and Rains \cite{FR} in relation with finding random recurrences for classical matrix ensembles
\eqref{eq_general_beta_intro} with varying $N$, and with  percolation models.

All of the above constructions presented for the Hermite ensemble admit a generalization to the
Jacobi weight. This leads to a multilevel general $\beta$ Jacobi ensemble, or
\emph{$\beta$--Jacobi corners process}, which is the main object of the present paper and whose
definition we now present. In Section \ref{section_matrix_models} we will further comment on the
relation of this multilevel ensemble to matrix models.

\smallskip

Fix two integer parameters $N>0$, $M>0$ and a real parameter $\alpha>0$. Also set
$\theta=\beta/2>0$. Let $\St^M(N)$ denote the set of families $x^1,x^2,\dots,x^N$, such that
for each $1\le n \le N$, $x^n$ is a sequence of real numbers of length $\min(n,M)$, satisfying:\
$$
 0< x^n_1< x^n_2 < \dots < x^n_{\min(n,M)} < 1.
$$
We also require that for each $1\le n\le N-1$ the sequences $x^n$ and $x^{n+1}$ \emph{interlace}
$x^n\prec x^{n+1}$, which means that
$$
 x^{n+1}_1 < x^n_1 < x^{n+1}_2 < x^n_2 < \dots.
$$
Note that for $n\ge M$ the length of sequence $x^n$ does not change and equals $M$, this is a
specific feature of the Laguerre and Jacobi ensembles which is not present in the Hermite
ensemble.

\begin{definition}
\label{Definition_multi_level_intro} The $\beta$--Jacobi corners process of rank $N$ with
parameters $M,\alpha,\theta$ as above is a probability distribution on the set $\St^M(N)$ with
density with respect to the Lebesgue measure proportional to
\begin{multline}
\label{eq_multilevel_dist_introN}
 \prod_{1\le i<j\le \min(N,M)} (x_j^N-x_i^N)
 \prod_{i=1}^{\min(N,M)} (x_i^N)^{\theta(\alpha+N-1)-1} \prod_{i=1}^{\min(N,M)} \left(1-x_i^{\min(N,M)}\right)^{\theta-1+\theta(N-M)_+ }
\\ \times
 \prod_{k=1}^{N-1}\left( \prod_{i=1}^{\min(k,M)} (x_i^k)^{-2\theta} \prod_{1\le i<j\le \min(k,M)} (x_j^k-x_i^k)^{2-2\theta} \prod_{a=1}^{\min(k,M)} \prod_{b=1}^{\min(k+1,M)} |x^k_a-x^{k+1}_b|^{\theta-1}
 \right),
\end{multline}
where $(N-M)_+=\max(N-M,0)$.
\end{definition}

 Similarly to the GUE--corners, the projection of $\beta$-Jacobi corners process to a single level
 $k$
 is given by the $\beta$--Jacobi
ensemble whose density in our notations is given by
\begin{equation}
\label{eq_Jacobi_intro}
 \prod_{1\le i<j \le \min(k,M)} (x^k_j-x^k_i)^{2\theta}
 \prod_{i=1}^{\min(k,M)} (x^k_i)^{\theta\alpha-1} (1-x^k_i)^{\theta(|M-k|+1)-1}.
\end{equation}
Further, the definition of $\beta$--Jacobi corners process is consistent for various $N\ge 1$,
meaning that the restriction of the corners process of rank $N$ to the first $n<N$ levels is the
corners process of rank $n$.

\subsection{The main result}



Our main result concerns the global (Gaussian) fluctuations of the $\beta$--Jacobi corners
process. Let us start, however, with a discussion of similar results for single--level
$\beta$--Jacobi ensembles that have been obtained before.\footnote{For a discussion of the
remarkable recent progress in understanding local fluctuations of general $\beta$ ensembles see
e.g.\ Valk\'{o}--Vir\'{a}g \cite{VV}, Ramirez--Rider--Vir\'{a}g
\cite{RRV}, Bourgade--Erdos--Yau \cite{BEY1}, \cite{BEY2}, Shcherbina \cite{Shch}, Bekerman--Figalli--Guionnet
\cite{BFG}
and references therein.}

Fix two reals $u_0, u_1\ge 0$ and let $x_1<\dots<x_N$ be sampled from the $\beta$-Jacobi ensemble
of rank $N$ with parameters $\beta u_0 N$, $\beta u_1 N$, i.e.\ \eqref{eq_general_beta_intro} with
$w(x)=x^{\beta u_0 N}(1-x)^{\beta u_1 N}$. Define the \emph{height function} ${\mathcal H}(x,N)$,
$x\in[0,1]$, $N=1,2,\dots$ as the number of eigenvalues $x_i$ which are less than $x$. The
computation of the leading term of the asymptotics of ${\mathcal H}(x,N)$ as $N\to\infty$ can be
viewed as the Law of Large Numbers.

\begin{proposition}[Dumitriu--Paquette \cite{DP}, Jiang \cite{Jiang}, Killip \cite{Killip}] \label{Proposition_LLN} For any values of $\beta>0$, $u_0, u_1\ge
0$, the normalized height function ${\mathcal H}(x,N)/N$ tends (in $sup$--norm, in probability) to
an explicit deterministic limit function $\widehat{\mathcal H}(x; u_0, u_1)$, which does not
depend on $\beta$.
\end{proposition}
In the case $\beta=1$ Proposition \ref{Proposition_LLN} was first established by Wachter
\cite{Wachter}, nowadays other proofs for classical $\beta=1,2,4$ exist, see e.g.\ Collins
\cite{Collins}. In the case of Hermite (Gaussian) ensemble and $\beta=1,2,4$ an analogue of
Proposition \ref{Proposition_LLN} dates back to the work of Wigner \cite{Wigner} and is known as
the Wigner semicircle law.

One feature of the limit profile $\widehat{\mathcal H}(x; u_0, u_1)$ is that its derivative (in
$x$) is non-zero only inside a certain interval $({\mathfrak l}(u_0, u_1), {\mathfrak r}(u_0,
u_1))\subset (0,1)$. In other words, outside this interval asymptotically as $N\to\infty$ we see
no eigenvalues, and $\widehat{\mathcal H}(x; u_0, u_1)$ is constant. Somewhat abusing notation
we call $[{\mathfrak l}(u_0,u_1), {\mathfrak r}(u_0, u_1)]$ \emph{the support} of
$\widehat{\mathcal H}(x; u_0, u_1)$.

\smallskip

Studying the next term in the asymptotic expansion of ${\mathcal H}(x,N)$ leads to various Central
Limit Theorems. One could either study ${\mathcal H}(x,N)-\E {\mathcal H}(x,N)$ as $N\to\infty$
for a fixed $x$ or a smoothed version
$$
 \int_{0}^1 f(x) \big({\mathcal H}(x,N)-\E {\mathcal H}(x,N)\big) dx,
$$
for a (smooth) function $f(x)$. It turns out that these two variations leads to different scalings
and different limits. We will  concentrate on the second (smoothed) version of the central limit
theorems, see Killip \cite{Killip} for some results in the non--smoothed case.

Central Limit Theorems for random matrix ensembles at $\beta=2$ go back to Szeg\"{o}'s theorems on
the asymptotics of Toeplitz determinants, see Szeg\"{o} \cite{Szego1}, Forrester \cite[Section
14.4.2]{For}, Krasovsky \cite{Kr}. Nowadays several approaches exist, see \cite{AGZ}, \cite{For},
\cite{ABF} and references therein. The first result for general $\beta$ was obtained by Johansson
\cite{J} who studied the distribution \eqref{eq_general_beta_intro} with \emph{analytic}
potential. Formally, the Jacobi case is out of the scope of the results of \cite{J} (although the
approach is likely to apply) and here the Central Limit Theorem was obtained very recently by
Dumitriu and Paquette \cite{DP} using the tridiagonal matrices approach to general $\beta$
ensembles, cf.\ Dumitriu--Edelman \cite{DE}, Killip--Nenciu \cite{KN}, Edelman--Sutton \cite{ES}.

\begin{proposition}[\cite{DP}] Take $k\ge 1$ and any $k$ continuously differentiable functions
$f_1$,  \dots, $f_k$ on $[0,1]$. In the settings of Proposition \ref{Proposition_LLN} with
$u_0+u_1>0$, the vector
$$
 \int_{0}^{1} f_i(x) \big({\mathcal H}(x,N)-\E {\mathcal H}(x,N)\big) dx, \quad i=1,\dots,k,
$$
converges as $N\to\infty$ to a Gaussian random vector.
\end{proposition}
Dumitriu and Paquette \cite{DP} also prove that the covariance matrix diagonalizes when $f(\cdot)$
ranges over suitably scaled Chebyshev polynomials of the first kind.

\bigskip
Let us now turn to the multilevel ensembles, i.e.\ to the $\beta$-Jacobi corners processes, and
state our result. Fix parameters $\hat M>0$, $\hat \alpha > 0$. Suppose that as our large
parameter $L\to\infty$, parameters $M$, $\alpha$ of Definition \ref{Definition_multi_level_intro}
grow linearly in $L$:
$$
  M \sim L \hat M,\quad \alpha \sim L \hat\alpha.
 $$
Let ${\mathcal H}(x,k)$, $x\in[0,1]$, $k\ge 1$, denote the height function of the $\beta$--Jacobi
corners process, i.e.\ ${\mathcal H}(x,k)$ counts the number of eigenvalues from $(x^{\lfloor k
\rfloor}_1,\dots,x^{\lfloor k \rfloor}_{\lfloor k \rfloor})$ that are less than $x$. Observe that
Proposition \ref{Proposition_LLN} readily implies the Law of Large Numbers for ${\mathcal
H}(x,L\hat N)$, $\hat N>0$, as $L\to\infty$. Let $[{\mathfrak l}(\hat N), {\mathfrak r}(\hat N)]$
denote the support of $\lim_{L\to\infty} {\mathcal H}(x,L\hat N)$ (both endpoints depend on $\hat
\alpha, \hat M$, but we omit this dependence from the notations). Further, let $D$ denote the
region inside $[0,1]\times {\mathbb R_{>0}}$ on the $(x,\hat N)$ plane defined by the inequalities
${\mathfrak l}(\hat N)\le x \le {\mathfrak r}(\hat N)$.

Now we are ready to state our main theorem, giving the asymptotics of the fluctuations of
$\beta$--Jacobi corners process in terms of the two--dimensional Gaussian Free Field (GFF, for
short). We briefly recall the definition and basic properties of the GFF in Section
\ref{Section_GFF}.

\begin{theorem}
\label{Theorem_main_intro}
 Suppose that as our large parameter $L\to\infty$, parameters $M$, $\alpha$ grow linearly in $L$:
 $$
  M \sim L \hat M,\quad \alpha \sim L \hat\alpha;\quad \hat M>0, \quad \hat \alpha>0.
 $$
 Then the centered random (with respect to the measure of Definition \ref{Definition_multi_level_intro}) height function
 $$
  \sqrt{\theta\pi}\big({\mathcal H}(x,L\hat N)-\E {\mathcal H}(x,L\hat N)\big)
 $$
 converges to the pullback of the Gaussian Free Field with Dirichlet boundary conditions on the upper halfplane $\mathbb H$ with respect to a map $\Omega: D\to \mathbb H$ (see Definition
 \ref{def_Omega} for the explicit formulas)
 in
 the following sense:
 For any set of polynomials $R_1,\dots,R_k\in \mathbb C[x]$ and positive numbers $\hat N_1,\dots,\hat N_k$,
 the joint distribution of
 $$
  \int_0^1 R_i(x) \big({\mathcal H}(x,L\hat N_i)-\E {\mathcal H}(x,L\hat N_i)\big)  dx,\quad i=1,\dots, k,
 $$
 converges to the joint distribution of the similar averages
 $$
  \int_0^1 {\mathcal F}(\Omega(x,\hat N_i)) R_i(x) dx,\quad i=1,\dots, k,
 $$
 of the pullback of the GFF.
\end{theorem}
There are several reasons why one might expect the appearance of the GFF in the study of the
general $\beta$ random matrix corners processes.

First, the GFF is believed to be a universal scaling limit for various models of \emph{random
surfaces} in $\mathbb R^3$. Now the appearance of the GFF is rigorously proved for several models
of random \emph{stepped} surfaces, see Kenyon \cite{Kenyon}, Borodin--Ferrari \cite{BF}, Petrov
\cite{Petrov}, Duits \cite{Duits}, Kuan \cite{Kuan}, Chhita--Johansson--Young \cite{CJY},
Borodin--Bufetov \cite{BB}. On the other hand, it is known that random matrix ensembles for
$\beta=2$ can be obtained as a certain limit of stepped surfaces, see Okounkov--Reshetikhin
\cite{OR-birth}, Johansson--Nordenstam \cite{JN}, Fleming--Forrester--Nordenstam \cite{FlForNor},
Gorin \cite{G}, \cite{G2}, Gorin--Panova \cite{GP}, hence one should expect the presence of the GFF in
random matrices.

Second, for random \emph{normal} matrices, whose eigenvalues are no longer real, but the
interaction potential is still logarithmic, the convergence of the fluctuations to the GFF was
established by Ameur--Hedenmalm--Makarov \cite{AHM1}, \cite{AHM2}, see also Rider--Vir\'{a}g
\cite{RV}.

Further, in \cite{B-CLT}, \cite{B-CLT2} one of the authors proved an analogue of
Theorem \ref{Theorem_main_intro} for $\beta=1,2$ Wigner random matrices. This result
can also be accessed through a suitable degeneration of our results, see Remark 2
after Proposition \ref{prop_Chebyshev_ortho} for more details. Let us note that in
\cite{B-CLT}, \cite{B-CLT2} the image of an analogue of our map $\Omega$ was the
whole upper half--plane $\mathbb H$. This is not the case in Theorem
\ref{Theorem_main_intro}, the image is smaller, see Section
\ref{section_identification} for more details.

Finally, Spohn \cite{Spohn} found the GFF in the asymptotics of the general $\beta$ circular Dyson
 Brownian Motion. In fact, generalizations of Dyson Brownian motion give ways to add a \emph{third} dimension to the
corners processes, cf.\ \cite{B-CLT2} and also \cite{GS}. It would be interesting to study the global fluctuations of
the resulting $3d$ object, see \cite{B-CLT2} and \cite{BF} for some progress in this direction.

\smallskip

Note that other classical ensembles can be obtained from $\beta$-Jacobi through a suitable limit
transition and, thus, one should expect that similar results hold for them as well.
\subsection{The method}

Let us outline our approach to the proof of Theorem \ref{Theorem_main_intro}.

Recall that \emph{Macdonald polynomials} $P_\lambda(x_1,\dots,x_N;q,t)$ and
$Q_\lambda(x_1,\dots,x_N;q,t)$ are certain symmetric polynomials in variables $x_1,\dots,x_N$
depending on two parameters $q,t>0$ and parameterized by \emph{Young diagrams} $\lambda$, see
e.g.\ Macdonald \cite[Chapter VI]{M}. Given two sets of parameters $a_1,\dots, a_N>0$,
$b_1,\dots,b_M>0$ satisfying $a_i b_j<1$, $1\le i\le N$, $1\le j\le M$, the \emph{Macdonald
measure} on the set of all Young diagrams $\Y$ is defined as a probability measure assigning to
a Young diagram $\lambda$ a weight proportional to
\begin{equation}
\label{eq_Macd_meas_intro}
 P_\lambda(a_1,\dots,a_N;q,t) Q_\lambda(b_1,\dots,b_M;q,t).
\end{equation}
It turns out that (for a specific choice of $a_1,\dots,a_N$, $b_1,\dots,b_M$) when $q,t\to 1$ in
such a way that $t=q^{\theta}$, the measure \eqref{eq_Macd_meas_intro} weakly converges to the
(single level) $\beta$--Jacobi distribution with $\beta=2\theta$, see Theorem
\ref{Theorem_convergence_to_Jacobi} for the exact statement. This fact was first noticed by
Forrester and Rains in \cite{FR}.

Further, the Macdonald measures admit multilevel generalizations called \emph{ascending Macdonald
processes}. The same limit transition as above yields the $\beta$--Jacobi  corners process of
Definition \ref{Definition_multi_level_intro}.

In Borodin--Corwin \cite{BigMac} an approach for studying Macdonald measures through the Macdonald
difference operators, which are diagonalized by the Macdonald polynomials, was suggested. This
approach survives in the above limit transition and allows us to compute the expectations of
certain observables (essentially, moments) of $\beta$--Jacobi ensembles as results of the
application of explicit difference operators to explicit functions, see Theorem
\ref{Theorem_HO_expectations} and Theorem \ref{Theorem_HO_expectations_p} for the details.
Moreover, as we will explain, a modification of this approach allows us to study the Macdonald
processes and, thus, the $\beta$--Jacobi corners process, see also
Borodin--Corwin--Gorin--Shakirov \cite{BCGS} for further generalizations.

The next step is to express the action of the obtained difference operators through contour
integrals, which are generally convenient for taking asymptotics. Here the approach of
\cite{BigMac} fails (the needed contours cease to exist for our choice of parameters $a_i$, $b_j$
of the Macdonal measures), and we have to proceed in a different way. This is explained in Section
\ref{Section_integral}.

The Central Limit Theorem itself is proved in Section \ref{Section_CLT_general} using a
combinatorial lemma, which is one of the important new ingredients of the present paper.
Informally, this lemma shows that when the joint moments of a family of random variables can be
written via nested contour integrals typical for the Macdonald processes, then the asymptotics of
these moments is given by Isserlis's theorem (also known as Wick's formula), which proves the
asymptotic Gaussianity. See Lemma \ref{lemma_limit_gaussianity} for more details.

\medskip

We also note that the convergence of Macdonald processes to $\beta$--Jacobi corners processes is a
manifestation of a more general limit transition that takes Macdonald polynomials to the so-called
Heckman-Opdam hypergeometric functions (see  Heckman--Opdam \cite{HO}, Opdam \cite{Op},
Heckman--Schlichtkrull \cite{HS} for the general information about these functions), and general
Macdonald processes to certain probability measures that we call \emph{Heckman-Opdam
processes}\footnote{In spite of a similar name, they are different from ``Heckman--Opdam Markov
processes'' of Schapira \cite{Sch}.}. We discuss this limit transition in more detail in the
Appendix.

\subsection{Matrix models for multilevel ensembles}
\label{section_matrix_models}

There are many ways to obtain $\beta$--Jacobi ensembles at $\beta=1,2,4$, with various special
exponents $p$, $q$, through random matrix models, see e.g.\ Forrester \cite[Chapter 3]{For},
Duenez \cite{Due}, Collins \cite{Collins}. In most of them there is a natural way of extending the
probability measure to multilevel settings. We hope that a non-trivial subset of these situations
would yield the Jacobi corners process of Definition \ref{Definition_multi_level_intro}, but we do
not know how to prove that.

For example, take two infinite matrices $X_{ij}$, $Y_{ij}$, $i,j=1,2,\dots$ with i.i.d.\ Gaussian
entries (either real, complex or quaternion). Fix three integers $A,B,C>0$, let $X^{AC}$ be the
$A\times C$ top--left corner of $X$, and let $Y^{BC}$ be $B\times C$ top--left corner of of $Y$.
Then the distribution of (distinct from $0$ and $1$) eigenvalues $x_1\le x_2\le \dots\le x_N$,
$N=\min(A,B,C)$, of
$$
 {\mathcal M}^{ABC}=(X^{AC})^*X^{AC} \Big((X^{AC})^*X^{AC}+(Y^{BC})^*Y^{BC}\Big)^{-1}
$$ is given by the $\beta$--Jacobi ensemble
($\beta=1,2,4$, respectively) with density (see e.g.\ \cite[Section 3.6]{For})
$$
 \prod_{1\le i<j\le N} (x_j-x_i)^\beta \prod_{i=1}^N (x_i)^{\frac{\beta}{2}(|A-C|+1)-1}
 (1-x_i)^{\frac{\beta}{2}(|B-C|+1)-1}.
$$
Comparing this formula with \eqref{eq_Jacobi_intro} it is reasonable to expect that the joint
distribution of the above eigenvalues for matrices ${\mathcal M}^{AnC}$, $n=1,2,\dots,N$, is given
by Definition \ref{Definition_multi_level_intro} with $\theta=\beta/2$, $M=C$, $\alpha=A-C+1$ (and
we assume that $A\ge C$ here). However, we were unable to locate this statement in the literature,
thus, we leave it as a conjecture here. A similar statement in the limiting case of the Hermite
ensemble is proven e.g.\ by Neretin \cite{Ner}, while for the Laguerre ensemble with $\beta=2$
this is discussed by Borodin--Peche \cite{BP}, Dieker--Warren \cite{DW}, Forrester--Nagao
\cite{FN}.

\smallskip

 Another matrix model for Definition \ref{Definition_multi_level_intro} at $\theta=\beta/2=1$ was
suggested by Adler--van Moerbeke--Wang \cite{AMW}. In the above settings (with complex matrices)
set
$$
 \hat X_{ij}=\begin{cases} X_{ij},& \text{if } 1\le i\le 2j+A,\\ 0,&\text{otherwise}.\end{cases}
$$
Let $\hat X^{AC}$ be $(2C+A)\times C$ top--left corner of $\hat X$ and $Y^{BC}$ be as above, and
set
$$ \hat {\mathcal
M}^{ABC}= (\hat X^{AC})^*\hat X^{AC} \Big((\hat X^{AC})^*X^{AC}+(Y^{BC})^*Y^{BC}\Big)^{-1}.
$$
Suppose $N\le M$; then the joint distribution of eigenvalues of $\hat {\mathcal M}^{ABn}$,
$n=1,\dots,N$,
 is  given by Definition \ref{Definition_multi_level_intro} with $\theta=\beta/2$, $M=B$,
$\alpha=A+1$, as seen from \cite[Theorem 1]{AMW} (one should take $\alpha_n=A+n-1$ in this
theorem).

\smallskip

Let us also remark on the connection with tridiagonal models for classical general $\beta$
ensembles of Dumitriu--Edelman \cite{DE}, Killip--Nenciu \cite{KN}, Edelman--Sutton \cite{ES}. One
may try to produce an alternative definition of multilevel $\beta$-Jacobi (Laguere, Hermite)
ensembles using tridiagonal models for the single level ensembles
 and taking the joint distribution of eigenvalues of suitable
submatrices. However, note that these models produce the ensemble of rank $N$ out of linear in $N$
number of independent random variables, while the dimension of the set of interlacing
configurations grows quadratically in $N$. Therefore, this construction would produce a
distribution concentrated on a lower dimensional subset and, thus, singular with respect to the
Lebesgue measure, which is not what we need. On the other hand, it is possible that the marginals
of the $\beta$--Jacobi corners processes on two neighboring levels can be obtained as eigenvalues
of a random tridiagonal matrix and its submatrix of size by one less, see Forrester--Rains
\cite{FR} for some results in this direction.

\smallskip

Let us finally mention that Edelman \cite{Ed} discovered an ingenious random matrix algorithm that
(conjecturally for $\beta\ne 1,2,4$) yields the $\beta$--Jacobi corners process of Definition
\ref{Definition_multi_level_intro}.

\subsection{Further results}
Let us list a few other results proved below.

First, in addition to Theorem \ref{Theorem_main_intro}, we express the limit covariance of the
$\beta$--Jacobi corners process in terms of Chebyshev polynomials (in the spirit of the results of
Dumitiu--Paquette \cite{DP}), see Proposition \ref{prop_Chebyshev_ortho} for the details.

Further, we use the same techniques as in the proof of Theorem \ref{Theorem_main_intro} to analyze
the behavior of the (multilevel) $\beta$--Jacobi ensemble as $\beta\to\infty$. It is known that
the eigenvalues of the Jacobi ensemble concentrate near roots of the corresponding Jacobi
orthogonal polynomials as $\beta\to\infty$ (see e.g.\ Szeg\"{o} \cite[Section 6.7]{Szego}, Kerov
\cite{K-LLN}), and in the Appendix we sketch a proof of the fact that the fluctuations (after
rescaling by $\sqrt{\beta}$) are asymptotically Gaussian, see Theorem \ref{Theorem_CLT_in_beta}
for the details. Similar results for the \emph{single--level} Hermite and Laguerre ensembles were
previously obtained in Dumitriu--Edelman \cite{DE-CLT}. We have so far been unable to produce
simple formulas for the limit covariance or to identify the limit Gaussian process with a known
object.

\section{Setup}

The aim of this section is to put general $\beta$ Jacobi random matrix ensembles in the context of
the Macdonald processes and their degenerations that we will refer to as the Heckman--Opdam
processes.

\subsection{Macdonald processes}

\label{Section_Macdonald}

Let $\GTP_N$ denote the set of all $N$ tuples of non-negative integers
$$
 \GTP_N=\{\lambda_1\ge\lambda_2\ge\dots\ge\lambda_N\ge 0\mid \lambda_i\in\mathbb Z\}.
$$
We say that $\lambda\in\GTP_N$ and $\mu\in\GTP_{N-1}$ \emph{interlace} and write
$\mu\prec\lambda$, if\footnote{The notation $\GTP$ comes from \emph{Gelfand--Tsetlin patterns},
which are interlacing sequence of elements of $\GTP_i$, $i=1,\dots,N$ and parameterize the same
named basis in the irreducible representations of the unitary group $U(N)$. In the representation
theory the above interlacing condition appears in the \emph{branching rule} for the restriction of
an irreducible representation to the subgroup.}
$$
\lambda_1\ge\mu_1\ge\lambda_2\ge\dots\ge\mu_{N-1}\ge\lambda_N.
$$
Sometimes (when it leads to no confusion) we also say that $\lambda\in\GTP_N$ and $\mu\in\GTP_{N}$
(note the change in index) interlace and write $\mu\prec\lambda$ if
$$
\lambda_1\ge\mu_1\ge\lambda_2\ge\dots\ge\lambda_N\ge\mu_{N}.
$$
Informally, in this case we complement $\lambda$ with a single zero coordinate.

Let $\Lambda_N$ denote the algebra of symmetric polynomials in $N$ variables $x_1,\dots,x_N$ with
complex coefficients. $\Lambda_N$ has a distinguished (linear) basis formed by Macdonald
polynomials $P_\lambda(\cdot;q,t)$, $\lambda\in\GTP_N$, see e.g.\ \cite[Chapter VI]{M}.  Here $q$
and $t$ are parameters which (for the purposes of the present paper) we
assume to be real numbers satysfying $0<q<1$, $0<t<1$. We also need the ``dual'' Macdonald
polynomials $Q_\lambda(\cdot;q,t)$. By definition
$$
 Q_\lambda(\cdot;q,t)=b_\lambda P_\lambda(\cdot;q,t) ,
$$
where $b_\lambda=b_\lambda(q,t)$ is a certain explicit constant, see \cite[Chapter VI, (6.19)]{M}.

We also need \emph{skew} Macdonald polynomials $P_{\lambda/\mu}$ ($\lambda,\mu\in\GTP_N$,
$\lambda_i\ge\mu_i$ for all $i$) and $Q_{\lambda/\mu}$, they can be defined through the identities
\begin{equation}
\label{eq_Skew_Macdonald}
 P_{\lambda}(x_1,\dots,x_N,y_1,\dots,y_N;q,t)=\sum_{\mu\in\GTP_N} P_{\lambda/\mu}(x_1,\dots,x_N; q,t)
 P_{\mu}(y_1,\dots,y_N;q,t),
\end{equation}
$$
 Q_{\lambda}(x_1,\dots,x_N,y_1,\dots,y_N;q,t)=\sum_{\mu\in\GTP_N} Q_{\lambda/\mu}(x_1,\dots,x_N; q,t)
 Q_{\mu}(y_1,\dots,y_N;q,t).
$$
Somewhat abusing the notations, in what follows we write $P_\lambda(x_1,\dots,x_M;q,t)$, with
$\lambda\in\GTP_N$, $N\le M$, for $P_{\hat \lambda}(x_1,\dots,x_M;q,t)$, where $\hat \lambda\in
\GTP_M$ is obtained from $\lambda$ adding $M-N$ zero coordinates; similarly for $Q_\lambda$,
$P_{\lambda/\mu}$, $Q_{\lambda/\mu}$.

Finally, we adopt the notation $(a;q)_n$, $n=0,1,\dots$, for the $q$--Pochhammer symbols:
$$
 (a;q)_n=\prod_{i=1}^n (1-a q^{i-1}),\quad \quad (a;q)_{\infty}=\prod_{i=1}^{\infty} (1-a q^{i-1}).
$$

Let us fix the following set of parameters: an integer $M>0$, positive
reals $a_1,a_2,\dots$ and positive reals $b_1,\dots,b_M$. We will assume that $a_i b_j<1$ for all $i$, $j$. The
following definition is a slight
generalization of \cite[Definition 2.2.7]{BigMac}.

\begin{definition} \label{Definition_Asc_Macdo} The \emph{infinite ascending Macdonald process} indexed by $M,\{a_i\},\{b_j\}$ is a (random) sequence $\lambda^1$,
$\lambda^2,\dots$ such that
\begin{enumerate}
\item For each $N\ge 1$, $\lambda^N\in\GTP_{\min(N,M)}$ and also
$\lambda^{N}\prec\lambda^{N+1}$.
\item For each $N\ge 1$ the (marginal) distribution of $\lambda^N$ is given by
\begin{equation}
\label{eq_macdonald_single}
 {\rm Prob } \{\lambda^N=\mu\} = \frac{1}{Z_N} P_\mu(a_1,\dots,a_N;q,t) Q_\mu(b_1,\dots,b_M;q,t),
\end{equation}
where
\begin{equation}
\label{eq_normalizing_Macdonald}
 Z_N=\sum_{\mu \in\GTP_{\min(N,M)}} P_\mu(a_1,\dots,a_N;q,t) Q_\mu(b_1,\dots,b_M;q,t) =
 \prod_{i=1}^N \prod_{j=1}^M \frac{(ta_i b_j;q)_{\infty}}{(a_i b_j;q)_{\infty}}.
\end{equation}
\item $\{\lambda^N\}_{N\ge 1}$ is a trajectory of a Markov chain with (backward) transition
probabilities
\begin{equation}
\label{eq_macdonald_transition}
 {\rm Prob } \{\lambda^{N-1}=\mu\mid \lambda^N=\nu\}= P_{\nu/\mu}(a_N;q,t)
 \frac{P_{\mu}(a_1,\dots,a_{N-1};q,t)}{P_\nu(a_1,\dots,a_N;q,t)}.
\end{equation}
\end{enumerate}
\end{definition}

\begin{proposition}
 If sequences $\{a_i\}_{i=1}^{\infty}$ and $\{b_j\}_{j=1}^M$
 of positive parameters are such that $a_i b_j<1$ for all $i,j$,
 then the infinite ascending Macdonald process indexed by
 $M\ge 1,\{a_i\},\{b_j\}$ is well defined.
\end{proposition}
\begin{proof} The nonegativity of \eqref{eq_macdonald_single},
\eqref{eq_macdonald_transition} follows from the combinatorial formula for the (skew) Macdonald
polynomials \cite[Chapter VI, Section 7]{M}. The identity \eqref{eq_normalizing_Macdonald} is the
Cauchy identity for Macdonald polynomials \cite[Chapter VI, (2.7)]{M}. Note that the absolute
convergence of the series $\sum_\mu P_\mu Q_\mu$ in \eqref{eq_normalizing_Macdonald} follows from
the fact that it is a rearrangement of the absolutely convergent power series in $\{a_i,b_j\}$ for
the product form of $Z_N$. The consistency of properties \eqref{eq_macdonald_transition} and
\eqref{eq_macdonald_single} follows from the definition of skew Macdonald polynomials, cf.\
\cite{BG}, \cite{BigMac}.
\end{proof}

Let $f\in\Lambda_N$ be any symmetric polynomial. For $\lambda\in\GTP_N$ we define
$$
 f(\lambda)=f(q^{\lambda_1}t^{N-1},q^{\lambda_2}t^{N-2},\dots q^{\lambda_N}).
$$
Further, for any subset $I\subset\{1,\dots,N\}$ define
$$
 A_I(z_1,\dots,z_N;t) = t^{\frac{|I|(|I|-1)}2}\prod_{i\in I} \prod_{j\not \in I} \frac{tz_i-z_j}{z_i-z_j}.
$$
Define the shift operator $T^q_i$ through
$$
 [T^q_i f](z_1,\dots,z_N) = f(z_1,\dots,z_{i-1}, qz_i,z_{i+1},\dots,z_N).
$$
For any $k\le N$ define the $k$th \emph{Macdonald difference operator} $\Mac^k_N$ through
$$
\Mac^k_N = \sum_{|I|=k} A_I(z_1,\dots,z_N;t) \prod_{i\in I} T^q_i.
$$

\begin{theorem}
\label{theorem_Macdonald_expectations}
 Fix any integers $m\ge 1$,  $N_1\ge N_2 \ge\dots\ge N_m\ge 1$ and $1\le k_i\le N_i$, $i=1,\dots,m$.
  Suppose that
$\lambda^1,\lambda^2,\dots$ is an infinite ascending Macdonald process indexed by $M\ge
1,\{a_i\}_{i=1}^{\infty},\{b_j\}_{j=1}^M$ as described above. Then
\begin{equation}
\label{eq_Macd_expect}
 \E \left(\prod_{i=1}^m e_{k_i}\big(\lambda^{N_i}\big) \right) = \dfrac{ \Mac^{k_m}_{N_m} \cdots \Mac^{k_2}_{N_2} \Mac^{k_1}_{N_1}
 \left[\prod\limits_{i=1}^{N_1} H(z_i)\right]}{ \prod_{i=1}^{N_1} H(z_i)} \rule[-5mm]{0.9pt}{17mm}_{\,
 z_i=a_i},
\end{equation}
where
$$
 H(z)=\prod_{i=1}^M \frac{(t z b_i;q)_\infty}{(zb_i;q)_{\infty}},
$$
 $e_k$ is the degree $k$ elementary symmetric polynomial, and $|_{z_i=a_i}$ in the right side of \eqref{eq_Macd_expect}
means that we plug in $z_i=a_i$, $i=1,2,\dots$ into the formula after applying the difference operators in
$z$--variables.
\end{theorem}
For $N_1=N_2=\dots=N_m$ this statement coincides with the observation of \cite[Section
2.2.3]{BigMac}. For general $N_j$'s a proof can be found in \cite{BCGS}, however, it is quite
simple and we reproduce its outline below.
\begin{proof}[Sketch of the proof of Theorem \ref{theorem_Macdonald_expectations}]
 Macdonald
 polynomials are the eigenfunctions of Macdonald difference operators
  (see \cite[Chapter VI, Section 4]{M}):
 \begin{equation}
 \label{eq_Macdonald_operator_eigen}
  \Mac^k_n(P_\lambda(z_1,\dots,z_n;q,t)) = e_k(q^{\lambda_1}t^{n-1},\dots,q^{\lambda_n})
  P_\lambda(z_1,\dots,z_n;q,t).
 \end{equation}
 So, first expand $\prod_{i=1}^{N_1} H(z_i)$ using the Cauchy identity \cite[Chapter VI, (2.7)]{M}
$$
 \prod_{i=1}^{N_1} H(z_i)=\sum_{\lambda\in\GTP_{\min(N_1,M)}} P_\lambda(z_1,\dots,z_{N_1};q,t)
 Q_\lambda(b_1,\dots,b_M;q,t)
$$
 and then
 apply
 $\prod_{i=1}^{r-1} \Mac^{k_i}_{N_i}$
  to the sum, where $r$ is the maximal number such that $N_1=N_2=\dots=N_{r-1}$.
   Using \eqref{eq_Macdonald_operator_eigen} we get
 \begin{multline}
 \label{eq_x20}
  \prod_{i=1}^{r-1} \Mac^{k_i}_{N_i} \left( \prod_{i=1}^{N_1} H(z_i) \right)=
  \sum_{\lambda^{N_1}\in\GTP_{\min(N_1,M)}}
  \left(\prod_{i=1}^{r-1} e_{k_i}\left(q^{\lambda^{N_1}_1}t^{{N_1}-1},\dots,q^{\lambda^{N_1}_{N_1}}\right)\right)
  \\ \times
  P_{\lambda^{N_1}}(z_1,\dots,z_{N_1};q,t) Q_{\lambda^{N_1}}(b_1,\dots,b_M;q,t)
 \end{multline}
 Now substitute in \eqref{eq_x20} the decomposition (which is a version of the
 definition \eqref{eq_Skew_Macdonald})
 $$
 P_{\lambda^{N_1}}(z_1,\dots,z_{N_1};q,t)= \sum_{\lambda^{N_r}\in\GTP_{\min(N_r,M)}}
 P_{\lambda^{N_k}}(z_1,\dots,z_{N_r};q,t) P_{\lambda^{N_1}/\lambda^{N_r}}(z_{N_r+1},\dots,z_{N_1};q,t)
 $$
 and apply (again using \eqref{eq_Macdonald_operator_eigen})
 $\prod_{i=r}^{h-1} \Mac^{k_i}_{N_i}$
 to the resulting sum, where $h$ is the maximal number such that $N_r=N_{r+1}=\dots=N_{h-1}$.
 Iterating this procedure we arrive at the
 desired statement.
 \end{proof}

\subsection{Elementary asymptotic relations}

In what follows we will use the following technical lemmas.

\begin{lemma}
\label{lemma_q_poch_conv}
 For any $a,b\in\mathbb C$ and complex--valued function $u(\cdot)$ defined in a neighborhood of $1$ and
  such that
$$
 \lim_{q\to 1} u(q)=u
$$
with $0<u<1$, we have
 $$
  \lim_{q\to 1} \frac{(q^a u(q);q)_{\infty}}{(q^b u(q);q)_{\infty}} = (1-u)^{b-a}.
 $$
\end{lemma}
\begin{proof}
 For $q$ approaching $1$ we have (using $\ln(1+x)\sim x$ for small $x$)\footnote{We use the notation $f(x)\sim g(x)$ as
$x\to a$ if $\lim_{x\to a}\frac{f(x)}{g(x)}=1$.}
\begin{multline}
\label{eq_x16} \frac{(q^a u(q);q)_{\infty}}{(q^b u(q);q)_{\infty}}=\exp\left( \sum_{m=0}^{\infty}
\ln \frac{1-q^{a+m} u(q)}{1-q^{b+m} u(q)}\right)\\=\exp\left( \sum_{m=0}^{\infty} \ln\left(1+
\frac{q^m u(q) (q^b-q^a)}{1-q^{b+m} u(q)}\right)\right)\sim \exp\left( \sum_{m=0}^{\infty}
\frac{q^m u(q) (q^b-q^a)}{1-q^{b+m} u(q)}\right)
\\=
\exp\left( \sum_{m=0}^{\infty} (q^m-q^{m+1})  \frac{u(q) (q^b-q^a)}{(1-q^{b+m} u(q))(1-q)}\right)
\end{multline}
Note that as $q\to 1$
$$
 \frac{u(q) (q^b-q^a)}{(1-q^{b+m} u(q))(1-q)}\sim \frac{u(a-b)}{1-q^m u},
$$
the last sum in \eqref{eq_x16} turns into a Riemannian sum for an integral, and we get (omitting a
standard uniformity of convergence estimate)
$$
 \exp\left(\int_0^1 \frac{u(a-b)}{1-ux} dx\right) =\exp(-(a-b)\ln(1-u))=(1-u)^{b-a}. \qedhere
$$
\end{proof}
 {\bf Remark.} The convergence in Lemma \ref{lemma_q_poch_conv} is uniform in $u$ bounded
away from $1$.


\begin{lemma}
\label{lemma_q_Gamma_conv}
 For any $x\in\mathbb C\setminus\{0,-1,-2,\dots\}$ we have
$$
 \lim_{q\to 1} \frac{(q;q)_\infty}{(q^x;q)_\infty} (1-q)^{1-x} = \Gamma(x).
$$
\end{lemma}
\begin{proof}
 E.g.\ \cite[Section 1.9]{KLS}, \cite[Section 10.3]{AAR}.
\end{proof}

\subsection{Heckman--Opdam processes and Jacobi distributions}

\label{Section_main_def}

Throughout this section we fix two parameters $M\in \mathbb {Z}_+$ and $\alpha >0$.

Let $\St^M$ denote the set of families $\r^1,\r^2,\dots$, such that for each $N\ge 1$, $\r^N$ is a
sequence of real numbers of length $\min(N,M)$, satisfying:\
$$
 0< \r^N_1< \r^N_2 < \dots < \r^N_{\min(N,M)} < 1.
$$
We also require that for each $N$ the sequences $\r^N$ and $\r^{N+1}$ \emph{interlace} $\r^N\prec
\r^{N+1}$, which means that
$$
 \r^{N+1}_1 < \r^N_1 < \r^{N+1}_2 < \r^N_2 < \dots.
$$

\begin{definition}
\label{Def_distrib} The probability distribution $\Pr^{\alpha,M,\theta}$ on $\St^M$ is the unique
distribution satisfying
 two conditions:
\begin{enumerate}
 \item For each $N\ge 1$ the distribution of $\r^N$ is given by the
 following density (with respect to the Lebesgue measure)
\begin{multline}
\label{eq_general_beta}
 \Pr^{\alpha,M,\theta}(\r^N\in [z,z+dz])=\\ {\rm const}\cdot
  \prod_{1\le i<j \le \min(N,M)} (z_j-z_i)^{2\theta}
 \prod_{i=1}^{\min(N,M)} z_i^{\theta\alpha-1} (1-z_i)^{\theta|M-N|+\theta-1} d z_i,
\end{multline}
 where $0< z_1<\dots< z_{\min(N,M)}< 1$ and ${\rm const}$ is an (explicit) normalizing constant.
\item Under $\Pr^{\alpha,M,\theta}$, $\{\r^N\}_{N\ge 1}$, is a trajectory of
a Markov chain with backward transition probabilities having the following density:
\begin{multline}
\label{eq_transitional_probabilities}
 \Pr^{\alpha,M,\theta}(\r^{N-1}\in [z,z+dz] \mid \r^{N}=y)=
 \frac{\Gamma(N\theta)}{\Gamma(\theta)^{N}}\\ \times \prod\limits_{j=1}^{N}{y_j}^{(N-1)\theta} \prod\limits_{1\le i<j<N}
 (z_j-z_i) \prod\limits_{1\le i<j\le N} (y_j-y_i)^{1-2\theta}
 \prod\limits_{i=1}^{N-1}\prod\limits_{j=1}^N
\left|y_{j}-z_i\right|^ {\theta-1}  \prod\limits_{i=1}^{N-1} \frac{dz_i}{z_i^{N\theta}},
\end{multline}
for $N\le M$ and
\begin{multline}
\label{eq_transitional_probabilities_big_N}
 \Pr^{\alpha,M,\theta}(\r^{N-1}\in [z,z+dz] \mid \r^{N}=y)= \frac{\Gamma(N\theta)}{\Gamma(\theta)^{M}\Gamma(N\theta-M\theta)}\prod\limits_{1\le i<j \le M}
 (z_j-z_i)(y_j-y_i)^{1-2\theta}
 \\ \times \prod\limits_{j=1}^{M}{y_j}^{(N-1)\theta}(1-y_j)^{\theta(M-N-1)+1} \prod\limits_{i, j=1}^M
\left|y_{j}-z_i\right|^ {\theta-1}  \prod\limits_{i=1}^{M}
(1-z_i)^{\theta(N-M)-1}\frac{dz_i}{z_i^{N\theta}},
\end{multline}
for $N>M$, where $z\prec y$ in both formulas.
\end{enumerate}
\end{definition}

{\bf Remark 1.} The distribution \eqref{eq_general_beta} is known as the \emph{general $\beta$
Jacobi ensemble.}

{\bf Remark 2.} Straightforward computation shows that the restriction of $\Pr^{\alpha,M,\theta}$
on the first $N\ge 1$ levels gives the $\beta$--Jacobi corners process of Definition
\ref{Definition_multi_level_intro}.

{\bf Remark 3.} The backward transitional probabilities \eqref{eq_transitional_probabilities} are
known in the theory of Selberg integrals. They appear in the integration formulas due to Dixon
\cite{Dixon} and Anderson \cite{Anderson}. More recently, two--level distribution of the above
kind was studied by Forrester and Rains \cite{FR}.

{\bf Remark 4.} Alternatively, one can write down \emph{forward} transitional probabilities of the
Markov chain  $\{\r^N\}_{N\ge 1}$; together with the distribution of $\r^1$ given by
\eqref{eq_general_beta} they uniquely define $\Pr^{\alpha,M,\theta}$. In particular, for $1\le
N\le M$ we have
\begin{multline*}
\Pr^{\alpha,M,\theta}(\r^{N}\in [y,y+dy] \mid \r^{N-1}=z)=
 {\rm const} \prod\limits_{1\le i<j<N}
 (z_j-z_i)^{1-2\theta} \prod\limits_{1\le i<j\le N} (y_j-y_i) \\ \times
 \prod\limits_{i=1}^{N-1}\prod\limits_{j=1}^N
\left|y_{j}-z_i\right|^ {\theta-1}  \prod\limits_{i=1}^{N-1} z_i^{-(N+\alpha)\theta+1}
(1-z_i)^{-\theta(N-M)+1} \prod\limits_{j=1}^{N}{y_j}^{(N-1+\alpha)\theta-1}
(1-y_j)^{\theta(N-M+1)-1}  dy_j.
\end{multline*}

\begin{proposition} The distribution $\Pr^{\alpha,M,\theta}$ is well-defined.
\end{proposition}
\begin{proof} We should check three properties here. First, we want to check  that the density
\eqref{eq_general_beta} is integrable and find the normalizing constant in this formula. This is
known as the Selberg Integral evaluation, see \cite{Selberg}, \cite{Anderson}, \cite[Chapter
4]{For}:
\begin{multline*}
 S_n(\alpha_0,\alpha_1,\gamma):=\int_{0}^1\cdots\int_0^1 \prod_{i=1}^n t_i^{\alpha_0-1} (1-t_i)^{\alpha_1-1}
 \prod_{1\le i<j\le n} |t_i-t_j|^{2\gamma} dt_1\dots dt_n
 \\=
 \prod_{j=0}^{n-1}
 \frac{\Gamma(\alpha_0+j\gamma)\Gamma(\alpha_1+j\gamma)\Gamma(1+(j+1)\gamma)}{\Gamma(\alpha_0+\alpha_1+(n+j-1)\gamma)\Gamma(1+\gamma)}.
\end{multline*}
We conclude that ${\rm const}$ in \eqref{eq_general_beta} is given by
$$
 {\rm const}=\dfrac{(\min(N,M))!}{ S_{\min(N,M)}(\theta\alpha,\theta|M-N|+\theta,\theta)}.
$$
Next, we want to check that \eqref{eq_transitional_probabilities} defines a probability
distribution, i.e.\ that the integral over $z$'s is $1$. For that we use a particular case of the
Dixon integration formula (see \cite{Dixon}, \cite[Exercise 4.2, q.\ 2]{For}) which reads
\begin{multline}
\label{eq_Dixon_formula}
 \int_T \prod_{1\le i<j\le n} (t_i-t_j) \prod_{i=1}^n \prod_{j=1}^{n+1}
 \frac{|t_i-a_j|^{\alpha_j-1}}{|b-t_i|^{\alpha_j}}
 dt_1\cdots dt_n
 \\
 = \dfrac{\prod\limits_{j=1}^{n+1}\Gamma(\alpha_j)}{\Gamma\left(\sum\limits_{j=1}^{n+1} \alpha_j\right)} \prod_{1\le i<j \le n+1}
 (a_i-a_j)^{\alpha_i+\alpha_j-1} \prod_{i=1}^{n+1} |b-a_i|^{\alpha_i-\sum_{j=1}^{n+1} \alpha_j} ,
\end{multline}
where the domain of integration $T$ is given by
$$
 a_1< t_1 <a_2<t_2\dots<t_n<a_{n+1}.
$$
Substituting , $n=N-1$, $b=0$, $\alpha_j=\theta$, $j=1,\dots,N$, in \eqref{eq_Dixon_formula} we
arrive at the required statement.

Finally, we want to prove the consistency of formulas \eqref{eq_general_beta} and
\eqref{eq_transitional_probabilities}, i.e.\ that for probabilities defined through those formulas
we have
\begin{multline*}
 \int_y \Pr^{\alpha,M,\theta}(\r^{N-1}\in [z,z+dz] \mid \r^{N}=y) \Pr^{\alpha,M,\theta}(\r^N\in
[y,y+dy])=\Pr^{\alpha,M,\theta}(\r^{N-1}\in [z,z+dz]).
\end{multline*}
Assuming $N\le M$, this is equivalent to
\begin{multline}
\label{eq_multilevel_consistency} \int \cdots \int \prod\limits_{j=1}^{N}{y_j}^{(N-1)\theta}
\prod\limits_{i<j}
 (z_j-z_i) \prod\limits_{i<j} (y_j-y_i)^{1-2\theta} \prod\limits_{i, j}
\left|y_{j}-z_i\right|^ {\theta-1}  \prod\limits_{i=1}^{N-1} z_i^{-N\theta}
\\\times
\prod_{1\le i<j \le N} (y_j-y_i)^{2\theta}
 \prod_{i=1}^{N} y_i^{\theta\alpha-1} (1-y_i)^{\theta|M-N|+\theta-1} d y_i,
\\
= {\rm const} \cdot \prod_{1\le i<j \le N-1} (z_j-z_i)^{2\theta}
 \prod_{i=1}^{N-1} z_i^{\theta\alpha-1} (1-z_i)^{\theta|M-N+1|+\theta-1} d z_i,
\end{multline}
where the integration goes over all $y$'s such that
$$
 0< y_1< z_1 < y_2< \dots < z_{N-1}<y_{N} < 1.
$$
In order to prove \eqref{eq_multilevel_consistency} we use another particular case of the Dixon
integration formula (this is $b\to\infty$ limit of \eqref{eq_Dixon_formula}), which was also
proved by Anderson \cite{Anderson}. This formula reads
\begin{multline}
\label{eq_Dixon_2} \int \cdots \int \prod\limits_{i<j} (y_j-y_i) \prod\limits_{i, j}
\left|y_{j}-a_i\right|^ {\alpha_i-1} dy_1 \cdots
dy_n=\frac{\prod\limits_{j=1}^{n+1}\Gamma(\alpha_j)} {\Gamma\left(\sum\limits_{j=1}^{n+1}
\alpha_j\right)} \prod_{1\le i<j \le n+1}
 (a_i-a_j)^{\alpha_i+\alpha_j-1},
\end{multline}
where the integration is over all $y_i$ such that
$$
 a_1< y_1 < a_2 < \dots < y_n < a_{n+1}.
$$
Choosing
$$
n=N,\quad a_1=0,\quad a_2=z_1,\quad a_3=z_2,\quad \dots,\quad a_{n}=z_{n-1},\quad a_{n+1}=1
$$
and appropriate $\alpha_i$ we arrive at \eqref{eq_multilevel_consistency}. For $N>M$ the argument
is similar.
\end{proof}

Our next aim is to show that $\Pr^{\alpha,M,\theta}$ is a  scaling limit of the ascending
Macdonald processes from Section \ref{Section_Macdonald}.

\begin{theorem}
\label{Theorem_convergence_to_Jacobi}
 Fix two positive reals $\theta$, $\alpha$ and a positive integer $M$. Consider two sequences $a_i=t^{i-1}$, $i=1,2,\dots$ and
$b_i=t^\alpha t^{i-1}$, $i=1,\dots,M$, and let $\lambda^1,\lambda^2,\dots$ be distributed
according to the infinite ascending  Macdonald process of Definition \ref{Definition_Asc_Macdo}.
For
 $\eps>0$ set
$$
 q=\exp(-\eps),\quad t=q^{\theta},
$$
and define
$$
 \r^i_j(\eps)=\exp(-\eps\lambda^i_j).
$$
Then as $\eps\to 0$ the finite-dimensional distributions of $\{\r^i_j(\eps)$, $i=1,2,\dots$,
$j=1,\dots,\min(M,i)\}$ weakly converge to those of $\Pr^{\alpha,M,\theta}$.
\end{theorem}
{\bf Remark.} The result of Theorem \ref{Theorem_convergence_to_Jacobi} is a manifestation of a
more general limit transition that takes Macdonald polynomials to the so-called Heckman-Opdam
hypergeometric functions and general Macdonald processes to certain probability measures that we
call \emph{Heckman-Opdam processes}. In particular, $\Pr^{\alpha,M,\theta}$ is a Heckman-Opdam
process. As all we shall need in the sequel is the above theorem, we moved the discussion of these
more general limiting relations to the appendix.

\begin{proof}[Proof of Theorem \ref{Theorem_convergence_to_Jacobi}]
 We need to prove that \eqref{eq_macdonald_single} converges to \eqref{eq_general_beta} and
\eqref{eq_macdonald_transition} converges to \eqref{eq_transitional_probabilities},
\eqref{eq_transitional_probabilities_big_N}. Let us start from the former.

For any $\lambda\in\GTP_N$  and $M\ge N$ we have with the agreement that $\lambda_i=0$ for $i>N$
(see \cite[Chapter VI, (6.11)]{M})
\begin{multline*}
 P_\lambda(1,\dots,t^{M-1};q,t)=t^{\sum_{i=1}^N (i-1)\lambda_i} \prod_{i<j\le M} \frac{(q^{\lambda_i-\lambda_j}
 t^{j-i};q)_\infty} {(q^{\lambda_i-\lambda_j}
 t^{j-i+1};q)_\infty} \cdot \frac{(
 t^{j-i+1};q)_\infty} {(
 t^{j-i};q)_\infty}
 \\= t^{\sum_{i=1}^N (i-1)\lambda_i} \prod_{i<j\le N} \frac{(q^{\lambda_i-\lambda_j}
 t^{j-i};q)_\infty} {(q^{\lambda_i-\lambda_j}
 t^{j-i+1};q)_\infty} \prod_{i=1}^N \prod_{j=N+1}^M \frac{(q^{\lambda_i}
 t^{j-i};q)_\infty} {(q^{\lambda_i}
 t^{j-i+1};q)_\infty}
 \prod_{i<j}^{i\le N;\, j\le M} \frac{(
 t^{j-i+1};q)_\infty} {(
 t^{j-i};q)_\infty}.
\end{multline*}
In the limit regime
$$
 q=\exp(-\eps),\quad t=q^{\theta},\quad \lambda_i=-\eps^{-1} \log(\r_i),\quad \eps\to 0,
$$
using Lemma \ref{lemma_q_poch_conv} we have
$$
 t^{\sum_{i=1}^N \lambda_i(i-1)} \to \prod_{i=1}^N (\r_i)^{\theta(i-1)}, \quad
\frac{(q^{\lambda_i-\lambda_j}
 t^{j-i};q)_\infty} {(q^{\lambda_i-\lambda_j}
 t^{j-i+1};q)_\infty}\to \Big(1-\r_i/\r_j\Big)^{\theta},
$$
$$
 \prod_{j=N+1}^M \frac{(q^{\lambda_i}
 t^{j-i};q)_\infty} {(q^{\lambda_i}
 t^{j-i+1};q)_\infty}= \frac{(q^{\lambda_i}
 t^{N+1-i};q)_\infty} {(q^{\lambda_i}
 t^{M+1-i};q)_\infty} \to \Big(1-\r_i\Big)^{\theta(M-N)},
$$
and using Lemma \ref{lemma_q_Gamma_conv} we get
$$
  \frac{(
 t^{j-i+1};q)_\infty} {(
 t^{j-i};q)_\infty}\sim
 \frac{\Gamma(\theta(j-i))}{\Gamma(\theta(j-i+1))} \eps^{-\theta}.
$$
We also need (see \cite[Chapter VI, (6.19)]{M})
\begin{equation}
\label{eq_blambda}
 \frac{Q_\lambda(\cdot;q,t)}{P_\lambda(\cdot;q,t)}=b_\lambda= \prod_{1\le i\le j \le \ell(\lambda)}
 \frac{f(q^{\lambda_i-\lambda_j} t^{j-i})}{f(q^{\lambda_i-\lambda_{j+1}} t^{j-i})},\quad
 f(u)=\frac{(tu;q)_\infty}{(qu; q)_\infty}.
\end{equation}
 Thus, canceling asymptotically equal factors, we get
$$
 b_\lambda\sim \prod_{i=1}^N \frac{f(1)}{f(q^{\lambda_i}t^{N-i})}\sim
 \frac{\eps^{N(1-\theta)}}{\Gamma(\theta)^N} \prod_{i=1}^N(1-\r_i)^{\theta-1}.
$$
Also
\begin{multline*}
\frac{(ta_i b_j;q)_\infty}{(a_i b_j;q)_\infty}=\frac{(t \cdot t^{i-1}\cdot t^{\alpha+j-1}
;q)_{\infty}}{(t^{i-1}\cdot t^{\alpha+j-1};q)_{\infty}}=
 f(t^{i-1} \cdot t^{\alpha+j-1})  \sim
\left[\frac{\Gamma(\theta(i-1+j-1+\alpha)}{\Gamma(\theta(i-1+j-1+\alpha+1))}
\eps^{-\theta}\right].
\end{multline*}

We conclude that for $N\le M$ as $\eps\to 0$
\begin{multline*}
 \prod_{i=1}^N \prod_{j=1}^M \frac{(t^{i-1}\cdot t^{\alpha+j-1};q)_{\infty}} {(t \cdot t^{i-1}\cdot t^{\alpha+j-1} ;q)_{\infty}}
 P_\lambda(1,\dots,t^{N-1};q,t)Q_\lambda(t^{\alpha},\dots,t^{\alpha+M-1};q,t)
\\ \sim {\rm const}\cdot \eps^{N} \prod_{1\le i<j\le N} (\r_i-\r_j)^{2\theta} \prod_{i=1}^N
(\r_i)^{\alpha\theta}(1-\r_i)^{\theta(M-N)+\theta-1}
\end{multline*}
where ${\rm const}$ is a certain (explicit) constant (depending on $N$, $M$, $\theta$ and
$\alpha$). Taking into the account that $d\r_i\sim -\eps \r_i d \lambda_i$ , that the convergence
in all the above formulas is uniform over compact subsets of the set $0<\r_1<\dots<\r_N<1$, and
that both prelimit and limit measures have mass one, we conclude that \eqref{eq_macdonald_single}
weakly converges to \eqref{eq_general_beta}. For $N>M$ the argument is similar.

\medskip
It remains to prove that \eqref{eq_macdonald_transition} weakly converges to
\eqref{eq_transitional_probabilities}. Using \cite[Chapter VI, (7.13)]{M} we have
\begin{equation}
\label{eq_psi}
  P_{\lambda/\mu}(t^{N-1};q,t)=\psi_{\lambda/\mu}(t^{N-1}),
\end{equation}
where ($f(\cdot)$ was defined above, see \eqref{eq_blambda})
$$
 \psi_{\lambda/\mu}(x)=x^{|\lambda|-|\mu|}
f(1)^{N-1} \prod_{i<j<N} f(q^{\mu_i-\mu_j}
 t^{j-i})
  \prod_{i\le j<N}
\frac{f(q^{\lambda_i-\lambda_{j+1}}t^{j-i})}{f(q^{\mu_i-\lambda_{j+1}}t^{j-i})
{f(q^{\lambda_i-\mu_j}t^{j-i})}}.
$$
In the limit regime
$$
 q=\exp(-\eps),\quad t=q^{\theta},\quad \lambda_i=-\eps^{-1} \log(\r_i),\quad \mu_i=-\eps^{-1}
 \log(\r_i'),\quad \eps\to 0,
$$
$$
 f(1)\sim \frac{\varepsilon^{1-\theta}}{\Gamma(\theta)}, \quad
 f(q^{\mu_i-\mu_j}
 t^{j-i})\sim (1-\r'_i/\r'_j)^{1-\theta},
$$
$$
\frac{f(q^{\lambda_i-\lambda_{j+1}}t^{j-i})}{f(q^{\mu_i-\lambda_{j+1}}t^{j-i}){f(q^{\lambda_i-\mu_j}t^{j-i})}}\sim
\left[\frac{1-\r_i/\r_{j+1}}{\big(1-\r'_i/\r_{j+1}\big)\big(1-\r_i/\r'_j\big)}\right]^ {1-\theta},
$$
$$
 t^{(N-1)(|\lambda|-|\mu|)}\sim \left(\frac{\prod_i \r_i}{\prod_i \r'_i}\right)^{(N-1)\theta} .
$$
Therefore,
\begin{multline}
  \psi_{\lambda/\mu}(t^{N-1})\sim \frac{\varepsilon^{(N-1)(1-\theta)}}{\Gamma(\theta)^{N-1}}
 \prod_i \r_i^{(N-1)\theta} \prod_i (\r'_i)^{1-N\theta} \prod_{i<j<N}
 (\r'_j-\r'_i)^{1-\theta} \\ \times \prod_{i\le j<N}
\left[\frac{\r_{j+1}-\r_i}{\big(\r_{j+1}-\r'_i\big)\big(\r'_j-\r_i\big)}\right]^ {1-\theta}.
\end{multline}
Taking into the account that $d\r'_i\sim - \eps \r'_i d \lambda_i$ and the above formulas for the
asymptotic behavior of $P_\lambda(1,\dots,t^{N-1};q,t)$, the uniformity of convergence on compact
subsets of the set defined by interlacing condition $\r'\prec \r$, and the fact that we started
with a probability measure and obtained a probability density,  we conclude that
\eqref{eq_macdonald_transition} weakly converges to \eqref{eq_transitional_probabilities}. For
$N>M$ the argument is similar.
\end{proof}

Let $\Lambda$ denote the algebra of symmetric functions, which can be viewed as the algebra of
symmetric polynomials of bounded degree in infinitely many variables $x_1,x_2,\dots$, see e.g.\
\cite[Chapter I, Section 2]{M}. One way to view $\Lambda$ is as an algebra of polynomials in
Newton power sums
$$
 p_k=\sum_{i} (x_i)^k,\quad k=1,2,\dots.
$$

\begin{definition}
\label{definition_observables}
 For a symmetric function $f\in\Lambda$, let
 $f(N;\cdot)$ denote the function on $\St^M$ given by
 $$
  f(N;\r)=\begin{cases}
  f(\r^N_1,\r^N_2,\dots,\r^N_N,0,0,\dots), & N\le M,\\
  f(\r^N_1,\r^N_2,\dots,\r^N_M,\underbrace{1,\dots,1}_{N-M},0,0,\dots), & N>M.
  \end{cases}
 $$
\end{definition}
For example,
 $$
   p_k(N;\r) = \begin{cases} \sum\limits_{i=1}^N (\r^N_i)^k, & N\le M,\\
                            \sum\limits_{i=1}^M (\r^N_i)^k + N-M, & N>M.
                            \end{cases}
 $$
For every $M,N\ge 1$ and $\alpha>0$ define functions $H(y;\alpha,M)$ and $H_N(\cdot; \alpha,M)$ in
variables $y_1,\dots,y_N$ through
\begin{equation}
\label{eq_definition_of_H}
  H_N(y_1,\dots,y_N;\alpha,M)=\prod_{i=1}^N H(y_i;\alpha,M) =\prod_{i=1}^N \frac{\Gamma(-y+\theta\alpha)}{\Gamma(-y+\theta\alpha+M\theta)}.
\end{equation}
For any subset $I\subset\{1,\dots,N\}$ define
$$
 B_I(y_1,\dots,y_N;\theta) = \prod_{i\in I} \prod_{j\not \in I} \frac{y_i -y_j-\theta}{y_i-y_j}.
$$
Define the shift operator $T_i$ through
$$
 [T_i f](y_1,\dots,y_N) = f(y_1,\dots,y_{i-1}, y_i-1,y_{i+1},\dots,y_N).
$$
For any $k\le N$ define the $k$th order difference operator $\D^k_N$ acting on functions in
variables $y_1,\dots,y_N$ through
\begin{equation}
\label{eq_definition_of_dif_operators} \D^k_N = \sum_{|I|=k} B_I(y_1,\dots,y_N;\theta) \prod_{i\in I}
T_i.
\end{equation}

{\bf Remark.} The operators $\D^k_N$ appeared in \cite[Appendix]{Cherednik-Quantum}, their
relation to the Heckman--Opdam hypergeometric functions was studied in \cite{Cherednik-Harish},
see also Proposition \ref{prop_HO_propeties} (V) below.

\smallskip

The following statement is parallel to Theorem \ref{theorem_Macdonald_expectations}.
\begin{theorem}
\label{Theorem_HO_expectations}  Fix any integers $m\ge 1$,  $N_1\ge N_2 \ge\dots\ge N_m\ge 1$ and
$1\le k_i\le N_i$, $i=1,\dots,m$.  With $\E$ taken with respect to $\Pr^{\alpha,M,\theta}$ of
Definition \ref{Def_distrib} we have
\begin{equation}
\label{eq_HO_expect}
 \E \left(\prod_{i=1}^m e_{k_i}(N_i,\r) \right) = \dfrac{ \D^{k_m}_{N_m}\cdots \D^{k_2}_{N_2} \D^{k_1}_{N_1}
 \left[\prod\limits_{i=1}^{N_1} H(y_i;\alpha,M)\right]}{ \prod_{i=1}^{N_1} H(y_i;\alpha,M)} \rule[-5mm]{0.9pt}{17mm}_{\,
 y_i=\theta(1-i)},
\end{equation}
where $e_k(N,\r)$ is given by Definition \ref{definition_observables} with $e_k$ being the $k$th degree elementary
symmetric polynomial.
\end{theorem}
\begin{proof} We start from Theorem \ref{theorem_Macdonald_expectations} and perform the limit transition of Theorem
\ref{Theorem_convergence_to_Jacobi}.

Note that $q^{\lambda_i}t^{N-i}<1$, therefore, $e_{k_i}(\lambda^{N_i})<N_i^{k_i}$ and Theorem
\ref{Theorem_convergence_to_Jacobi} implies that left side of \eqref{eq_Macd_expect} converges to
the left side of \eqref{eq_HO_expect}. Turning to the right sides, observe that for
$b_i=t^{\alpha+i-1}$ we have
$$
 \prod_{i=1}^M
 \frac{(tzb_i;q)_\infty}{(zb_i;q)_\infty}=\frac{(zt^{\alpha+M};q)_{\infty}}{(zt^{\alpha};q)_\infty}.
$$
Set
$$
q=\exp(-\eps),\quad t=q^\theta,\quad z_i=\exp(\eps y_i).
$$
Note that for any function $g(z_1,\dots,z_N)$ we have as $\eps\to 0$
$$
 \Mac^k_N g(z_1,\dots,z_n) \sim \D^k_N g(\exp(\eps y_1),\dots,\exp(\eps y_N)).
$$
Further, Lemma \ref{lemma_q_Gamma_conv} implies that as $\eps\to 0$
$$
  \frac{(z_it^{\alpha+M};q)_{\infty}}{(z_i t^{\alpha};q)_\infty}\sim \eps^{-\theta M}
 \frac{\Gamma(-y_i+\theta\alpha)}{\Gamma(-y_i+\theta\alpha+M\theta)} = \eps^{-\theta M} H(y_i).
$$
It follows that
\begin{multline*}
\lim_{\eps\to 0} \dfrac{ \Mac^{k_m}_{N_m}\cdots \Mac^{k_2}_{N_2} \Mac^{k_1}_{N_1}
 \left[\prod\limits_{i=1}^{N_1} \frac{(z_it^{\alpha+M};q)_{\infty}}{(z_it^{\alpha};q)_\infty} \right]}{ \prod_{i=1}^{N_1} \frac{(z_it^{\alpha+M};q)_{\infty}}{(z_it^{\alpha};q)_\infty}} \rule[-5mm]{0.9pt}{17mm}_{\,
 z_i=t^{i-1}}\\=\dfrac{ \D^{k_m}_{N_m}\cdots \D^{k_2}_{N_2}  \D^{k_1}_{N_1}
 \left[\prod\limits_{i=1}^{N_1} H(y_i;\alpha,M)\right]}{ \prod_{i=1}^{N_1} H(y_i;\alpha,M)} \rule[-5mm]{0.9pt}{17mm}_{\,
 y_i=\theta(1-i)}. \qedhere
\end{multline*}
\end{proof}

Next, we aim to define operators $\P^k_N$ which will help us in studying the limiting behavior of
observables $p_k(N,\r)$, cf.\ Definition \ref{definition_observables}.

Recall that partition $\lambda$ of number $n\ge 0$ is a sequence of integers
$\lambda_1\ge\lambda_2\ge\dots\ge 0$ such that $\sum_{i=1}^{\infty} \lambda_i=n$. The number of
non-zero parts $\lambda_i$ is denoted $\ell(\lambda)$ and called its length. The number $n$ is
called the size of $\lambda$ and denoted $|\lambda|$. For a partition $\mu=(\mu_1,\mu_2,\dots)$
let
$$
 e_\mu=\prod_{i=1}^{\ell(\mu)} e_{\mu_i}.
$$
Elements $e_\mu$ with $\mu$ running over the set $\Y$ of all partitions, form a linear basis of
$\Lambda$, cf.\ \cite[Chapter I, Section 2]{M}.

Let $PE(k,\mu)$ denote the transitional coefficients between $e$ and $p$ bases in the algebra of
symmetric functions, cf.\ \cite[Chapter I, Section 6]{M}:
$$
 p_k= \sum_{\mu\in\Y:\, |\mu|=k} PE(k,\mu) e_\mu.
$$
Define
\begin{equation}
\label{eq_P_through_Mac}
 \P^k_N=\sum_{\mu\in\Y:\, |\mu|=k} PE(k,\mu) \frac{1}{\ell(\mu)!} \sum_{\sigma\in\mathfrak S_{\ell(\mu)}} \prod_{i=1}^{\ell(\mu)}
\D^{\mu_{\sigma(i)}}_N,
\end{equation}
where $\mathfrak S_m$ is the symmetric group of rank $m$.

{\bf Remark.} Recall that operators $\D^k_N$, $k=1,\dots,N$ \emph{commute} (because they are
limits of $\Mac^k_N$ that are all diagonalized by the Macdonald polynomials, cf.\ \cite[Chapter
VI, (4.15)-(4.16)]{M}). However, it is convenient for us to think that the products of operators
in \eqref{eq_P_through_Mac} are ordered (and we sum over all orderings).

\smallskip

Theorem \ref{Theorem_HO_expectations} immediately implies the following statement.
\begin{theorem}
\label{Theorem_HO_expectations_p} Fix any integers $m\ge 1$,  $N_1\ge N_2 \ge\dots\ge N_m\ge 1$
and $k_i\ge 1$, $i=1,\dots,m$.  With $\E$ taken with respect to $\Pr^{\alpha,M,\theta}$ of
Definition \ref{Def_distrib} we have
\begin{equation}
\label{eq_HO_expect_2}
 \E \left(\prod_{i=1}^m p_{k_i}(N_i,\r) \right) = \dfrac{ \P^{k_m}_{N_m}\cdots \P^{k_2}_{N_2} \P^{k_1}_{N_1}
 \left[\prod\limits_{i=1}^{N_1} H(y_i;\alpha,M)\right]}{ \prod_{i=1}^{N_1} H(y_i;\alpha,M)} \rule[-5mm]{0.9pt}{17mm}_{\,
y_i=\theta(1-i)}.
\end{equation}
\end{theorem}

\section{Integral operators}
\label{Section_integral}

The aim of this section is to express expectations of certain observables with respect to measure
$\Pr^{\alpha,M,\theta}$ as contour integrals. This was done in \cite{BigMac} for certain
expectations of Macdonald processes; however, the approach of \cite{BigMac} fails in our case (the
contours of integration required in that paper do not exist) and we have to proceed differently.

Since in \eqref{eq_HO_expect}, \eqref{eq_HO_expect_2} the expectations of observables are
expressed in terms of the action of difference operators on products of univariate functions, we
will produce integral formulas for the latter.

\bigskip

Let $\S_n$ denote the set of all \emph{set partitions} of $\{1,\dots,n\}$. An element $s\in\S_N$
is a collection $S_1,\dots, S_k$ of disjoint subsets of $\{1,\dots,n\}$ such that
$$
\bigcup_{m=1}^k S_m=\{1,\dots,n\}.
$$
The number of non-empty sets in $s\in\S_n$ will be called the \emph{length} of $s$ and denoted as
$\ell(s)$. The parameter $n$ itself will be called the \emph{size} of $s$ and denoted as $|s|$. We
will also denote by $[n]$ the set partition of $\{1,\dots,n\}$ consisting of the single set
$\{1,\dots,n\}$.

Let $g(z)$ be a meromorphic function of a complex variable $z$, let $y=(y_1,\dots,y_N)\in\mathbb
C^N$, and let $d>0$ be a parameter. The system of closed positively oriented contours $\mathcal
C_{1}(y;g),\dots,\mathcal C_{k}(y;g)$ in the complex plane is called $(g,y,d)$-admissible ($d$
will be called \emph{distance parameter}), if
\begin{enumerate}
\item For each $i$, $\mathcal C_{i+1}(y;g)$ is inside the inner boundary of the $d$--neighborhood of $\mathcal
C_i(y;g)$. (Hence, $\mathcal C_k(y;g)$ is the smallest contour.)
\item All points $y_m$ are inside the smallest contour $\mathcal C_k(y;g)$ (hence, inside all contours) and
 $\frac{g(z-1)}{g(z)}$ is analytic inside the
 largest contour $\mathcal C_1(y;g)$. (Thus, potential singularities of $\frac{g(z-1)}{g(z)}$ have
to be outside all the contours.)
\end{enumerate}
From now on we assume that such contours do exist for every $k$ (and this will be indeed true for
our choices of $g$ and $y_m$'s).

Let $G^{(k)}$ be the following formal expression (which can be viewed as a $k$--dimensional
differential form)
$$
 G^{(k)}(v_1,\dots,v_k; y_1,\dots,y_N;g)= \prod_{i<j} \frac{(v_i-v_j)^2 }{ (v_i-v_j)^2-\theta^2}
 \prod_{i=1}^k\left(\prod_{m=1}^N \frac{v_i-y_m-\theta}{v_i-y_m}\cdot  \frac{g(v_i-1)}{g(v_i)}
  dv_i\right)
$$
(the dependence on $\theta$ is omitted from the notations). Note that $G^{(k)}$ is
\emph{symmetric} in $v_i$; this will be important in what follows.

Take any set partition $s=(S_1,\dots,S_{\ell(s)})\in\S_k$ and let $G_s$ denote the expression in
$\ell(s)$ variables $w_1,\dots, w_{\ell(s)}$ obtained by taking for each $h=1,\dots,\ell(s)$
 the \emph{residue} of $G^{(k)}$ at
\begin{equation}
\label{eq_x25}
 v_{i_1}=v_{i_2}-\theta=\dots=v_{i_r}-\theta(r-1),\quad S_h=\{i_1,\dots,i_r\},
\end{equation}
in variables $v_{i_2}$,\dots, $v_{i_r}$ (more precisely, we first take the residue in variable
$v_{i_r}$ at the pole $v_{i_r}=v_{i_{r-1}}+\theta$, then the residue in variable $v_{i_{r-1}}$ at
the pole $v_{i_{r-1}}=v_{i_{r-2}}+\theta$, etc) and renaming the remaining variable $v_{i_1}$ by
$w_h$. Here $i_1,\dots, i_r$ are all elements of the set $S_h$. Note that the symmetry of
$G^{(k)}$ implies that the order of elements in $S_h$ as well as the ordering of the resulting
variables $w_h$ are irrelevant. However, we need to specify \emph{some} ordering of $w_h$; let us
assume that the ordering is by the smallest elements of sets of the partitions, i.e.\ if $S_i$
corresponds to $w_i$ and $S_j$ corresponds to $w_j$ then the order in pair $(i,j)$ is the same as
that of the minimal elements in $S_i$ and $S_j$.

Observe that, in particular, $G_{1^k}=G^{(k)}$ with $v_i=w_i$, $i=1,\dots,k$.

\bigskip

\begin{definition}
\label{definition_Integral_operator}
 An admissible integral operator $\mathcal I_N$ is an operator which acts on the functions
 of the form $\prod_{i=1}^N g(y_i)$ via
 a $p$--dimensional integral
 $$
  \mathcal I_N \left[\prod_{i=1}^N g(y_i)\right] = \prod_{i=1}^N g(y_i) \cdot c\cdot \oint
  \prod_{i<j} \frac{\Cr_{i,j}(w_i-w_j)}{(w_i-w_j)^{d_{ij}}}
   \prod_{j=1}^{p} G_{[k_j]}(w_j;y_1,\dots,y_N;g),
 $$
 where $p$ is the dimension of integral, $\{k_j\}_{j=1}^p$ is a sequence of positive integral parameters, $d_{ij}$ are non-negative integral parameters,
  $\Cr_{i,j}(z)$ are analytic functions of $z$ which have limits as $z\to\infty$, and $c$ is a constant. The integration goes over nested
 admissible contours with large enough distance parameter, and $g(z)$ is assumed to be such that the admissible contours exist. We call $\sum_{i,j} d_{ij}$ the degree
 of operator $\mathcal I_N$.
\end{definition}
{\bf Remark.} When we have a series of integral operators $\mathcal I_N$, $N=1,2,\dots$, we
additionally assume that all the above data ($p$, $\{k_j\}$, $d_{ij}$, $\Cr_{i,j}(z)$) does not
depend on $N$. When $N$ is irrelevant, we sometime write simply $\mathcal I$.

\smallskip

A subclass of admissible integral operators is given by the following definition.

\begin{definition}
\label{def_Integral_operator}
 For a set partition $s\in\S_n$ define the integral operator $\DI^{s}_N$ acting on the product functions
 $\left(\prod_{i=1}^N g(y_i)\right)$ via
 \begin{equation}
 \label{eq_DI_operator}
  \DI^{s}_N \left(\prod_{i=1}^N g(y_i)\right)= \left(\prod_{i=1}^N g(y_i)\right)
  \frac{1}{(2\pi\i)^{\ell(s)}} \oint G_s(w_1,\dots, w_{\ell(s)};y_1,\dots,y_N;g),
 \end{equation}
 where
 each variable $w_i$, $1\le i\le \ell(s)$ is integrated over the (positively oriented, encircling
 $\{y_m\}_{m=1}^N$) contour
 $\mathcal C_i(y;g)$, and the system of contours $\mathcal C_i(y;g)$ is $(g,y,\theta
 |s|)$--admissible, as defined above.
\end{definition}
{\bf Remark.} The dimension of $\DI^s_N$ is $\ell(s)$ and its degree is $0$.

\smallskip

Now we are ready to state the main theorem of this section.

\begin{theorem}
\label{theorem_decomposition_of_P} The action of the difference operator $\P^k_N$ (defined by
\eqref{eq_P_through_Mac}) on any product function $\prod_{i=1}^N g(y_i)$ can be written as a
(finite) sum of admissible integral operators
$$
 \P^k_N=(-\theta)^{-k} \DI^{[k]}_N + \sum_{G\in\mathcal G} \mathcal I_N
(G),
$$
where the summation goes over a certain finite set $\mathcal G$ (independent of $N$). All integral
operators $\mathcal I_N(G)$, $G\in \mathcal G$ are such that the differences of their dimensions
and degrees are non-positive.
\end{theorem}
{\bf Remark. } Theorem \ref{theorem_decomposition_of_P} implies that as $N\to\infty$, the leading
term of $\P^k_N$ is given by $(-\theta)^{-k} \DI^{[k]}_N$ as can be seen by dilating the
integration contours by a large parameter. Moreover, as we will see in the next section the same
is true for the compositions of $\P^k_N$ (as in Theorem \ref{Theorem_HO_expectations_p}), their
leading term will be given by the compositions of $(-\theta)^{-k} \DI^{[k]}_N$. This property is
crucial for the proof of the Central Limit Theorem that we present in the next section.

\medskip

The rest of this section is devoted to the proof of Theorem \ref{theorem_decomposition_of_P} and
is subdivided into three steps. In Step 1 we find the decomposition of operators $\D^k_N$ (defined
by \eqref{eq_definition_of_dif_operators}) into the linear combination of admissible integral
operators $\DI^{s}_N$. In Step 2 we substitute the result of Step 1 into the definition of
operators $\P^k_N$ \eqref{eq_P_through_Mac} and obtain the expansion of $\P^k_N$ as a big sum of
admissible integral operators. We also encode each term in this expansion by a certain
\emph{labeled graph}. Finally, in Step 3 we observe massive cancelations in sums of Step 2,
showing that all integral operators in the decomposition of $\P^k_N$, whose difference of the
dimension and degree is positive (except for $(-\theta)^{-k} \DI^{[k]}_N$) vanish.

\subsection{Step 1: Integral form of operators $\D^k_N$.}

\begin{proposition} Fix integers $N\ge k \ge 1$. When applied to product functions of the form $\left(\prod_{i=1}^N
g(y_i)\right)$, the following operator identity holds:
$$
 \D^k_N=\frac{(-\theta)^{-k}}{k!} \sum_{s=(S_1,S_2,\dots)\in \S_k} \left[\prod_{h=1}^{\ell(s)} (|S_h|-1)!\right]
 \DI^s_N,
$$
where $\D^k_N$ is given by \eqref{eq_definition_of_dif_operators} and $\DI^s_N$ is given by
\eqref{eq_DI_operator}. \label{proposition_e_operator}
\end{proposition}
\begin{proof}
 Let us evaluate $\DI^s_N$ as a sum of residues. First, assume that $s=1^k$, i.e.\
 this is the partition of $\{1,\dots,k\}$ into singletons.
 We will first integrate over the smallest contour, then the second smallest one, etc.
 When we integrate over the smallest contour we get the sum of residues at points $y_1,\dots, y_N$.
 If we pick the $y_m$--term at this first step, then at the second one we could either take the residue at
 $y_j$ with $j\ne m$ or at $y_m-\theta$; there is no residue at $y_m$ thanks to $(v_i-v_j)^2$ in
 the definition of $G^{(k)}$, and there is no residue at $y_m+\theta$ thanks to the factor
 $v_i-y_m-\theta$ in $G^{(k)}$. When we continue, we see the formation of strings of residue locations of
 the form
 \begin{equation}
 \label{eq_x24}
  y_m,\, y_m -\theta,\, y_m - 2\theta,\dots.
 \end{equation}
 Observe that the sum of the residues in the decomposition of $\DI^{1^k}_N$ can be mimicked by the decomposition of the product
 $$
  \prod_{i=1}^k(x^i_1 +x^i_2+\dots +x^i_N)
 $$
 into the sum of monomials in variables $x^i_j$ (here "$i$" is the upper index, not an exponent).
 A general monomial
 $$
  x^1_{i_1} x^2_{i_2} \cdots x^k_{i_k}
 $$
 is identified with the residue of $G^{(k)}$ at
 $$
  y_1,\, y_1 -\theta,\,\dots,\, y_1-(m_1-1)\theta,\quad  y_2,\, y_2 -\theta,\,\dots,\, y_2-(m_2-1)\theta,\, \dots,
 $$
 where $m_c$ is the multiplicity of $c$'s in $(i_1,\dots, i_k)$.

 More generally, when we evaluate $\DI^s_N$ for a general $s\in\S_k$ we obtain sums of
 similar residues, and now the decomposition is mimicked by the product
 $$
  \prod_{h=1}^{\ell(s)} \left(x_1^{i_1(h)} x_1^{i_2(h)} \cdots x_1^{i_r(h)}+ x_2^{i_1(h)} x_2^{i_2(h)} \dots
  x_2^{i_r(h)} + \dots + x_N^{i_1(h)} x_N^{i_2(h)} \cdots x_N^{i_r(h)}\right),
 $$
 where $i_1(h),\dots, i_r(h)$ are all elements of the set $S_h$ in $s$.

 Now we will use the following combinatorial lemma which will be proved a bit later.
 \begin{lemma} We have
 \label{lemma_summation_mim_residues}
 \begin{multline}
 \label{eq_summation_mim_res}
  \sum_{s=(S_1,S_2,\dots)\in \S_k} (-1)^{k-\ell(s)}  \prod_{h=1}^{\ell(s)} (|S_h|-1)! \left(x_1^{i_1(h)} \cdots x_1^{i_r(h)}+
  \dots + x_N^{i_1(h)} \cdots x_N^{i_r(h)}\right)
  \\=
  \sum_\sigma x^1_{\sigma(1)} x^2_{\sigma(2)}\cdots x^k_{\sigma(k)},
 \end{multline}
 where $\sigma$ runs over all injective maps from $\{1,\dots,k\}$ to $\{1,\dots,N\}$.
 \end{lemma}

 Applying Lemma \ref{lemma_summation_mim_residues} and noting that the factor $(-1)^{k-\ell(s)}$ appears because the strings
 in \eqref{eq_x25} and \eqref{eq_x24} have different directions (which results in different signs when taking the
 residues),
   we conclude that
 $$
  \sum_{s=(S_1,S_2,\dots)\in \S_k} \left[ \prod_j (|S_j|-1)!\right] \DI^s_N
 $$
 is the sum of residues of $G^{(k)}$ at collections of \emph{distinct} points $y_{\sigma(1)},\dots,
 y_{\sigma(k)}$ as in \eqref{eq_summation_mim_res}. Computing the residues explicitly, comparing with the definition of
 $\D^k_N$,
 and noting that the factor $k!$ appears
 because of the ordering of $\sigma(1),\dots, \sigma(k)$ (we need to sum over $k$--point subsets, not over ordered $k$--tuples), we are done.
\end{proof}

We now prove Lemma \ref{lemma_summation_mim_residues}.
\begin{proof}[Proof of Lemma \ref{lemma_summation_mim_residues}]
 Pick $m_1,\dots, m_k$ and compare the coefficient of $x^1_{m_1}\cdots
 x^k_{m_k}$ in both sides of \eqref{eq_summation_mim_res}. Clearly, if all $m_i$ are distinct, then
 the coefficient is $1$ in the right side. It is also $1$ in left side, because such monomial
 appears only in the decomposition of $\DI^{1^k}_N$ and with coefficient $1$. If some of $m_i$'s
 coincide, then the corresponding coefficient in the right side of \eqref{eq_summation_mim_res} is
 zero. Let us prove that it also vanishes in the left side.

 Let $\hat s=(\hat S_1,\hat S_2 \dots)\in \S_k$ denote the partition of $\{1,\dots
k\}$ into sets formed by equal $m_i$'s (i.e.\ $a$ and $b$ belong to the same set of partition iff
$m_a=m_b$). Then the coefficient of $x^1_{m_1}\cdots
 x^k_{m_k}$ in the left side of \eqref{eq_summation_mim_res} is
\begin{equation}
\label{eq_x5} \sum_{s=(S_1,S_2\dots)\in \S_k} (-1)^{k-\ell(s)}  \prod_{h=1}^{\ell(s)} (|S_h|-1)!,
\end{equation}
where the summation goes over $s$ such that $s$ is a refinement of $\hat s$, i.e.\ sets of $\hat
s$ are unions of the sets of $s$. Clearly, \eqref{eq_x5} is equal to
$$
 \prod_{i=1}^{\ell(\hat s)}  \sum_{\substack{s \text{ runs over set partitions of } \hat S_i\\s=(S_1,S_2,\dots)}} (-1)^{|\hat S_i|-\ell(s)}  \prod_{h=1}^{\ell(s)}
 (|S_h|-1)!.
$$
Now it remains to prove that for any $n\ge 1$
\begin{equation}
\label{eq_x6} \sum_{s=(S_1,S_2,\dots)\in\S_n} (-1)^{n-\ell(s)} \prod_{h=1}^{\ell(s)}
 (|S_h|-1)! =0.
\end{equation} For that consider the well-known summation over symmetric group $\mathfrak S(n)$
\begin{equation}
\label{eq_x7}
 \sum_{\sigma\in\mathfrak S(n)} (-1)^{{\rm sgn}(\sigma)} =0.
\end{equation}
Under the map $\phi: \mathfrak S(n)\to \S_n$ mapping a permutation $\sigma$ into the set partition
that corresponds to the cyclic structure of $\sigma$, \eqref{eq_x7} turns into \eqref{eq_x6}.
\end{proof}

\subsection{Step 2: Operators $\P^k_N$ as sums over labeled graphs}
\label{Section_step2}

Let us substitute the statement of Proposition \ref{proposition_e_operator} into
\eqref{eq_P_through_Mac}. We obtain
\begin{multline}
\label{eq_P_through_Integral}
 \P^k_N=(-\theta)^{-k} \sum_{|\mu|=k} \frac{PE(k,\mu)}{\prod_i \mu_i!}\\ \times \frac{1}{\ell(\mu)!} \sum_{\sigma\in\mathfrak S_{\ell(\mu)}} \prod_{i=1}^{\ell(\mu)}
\left( \sum_{s=(S_1,S_2,\dots)\in \S_{\mu_{\sigma(i)}}}\left[  \prod_j (|S_j|-1)!\right] \DI^s_N
\right)
\end{multline}

The aim of this section is to understand the combinatorics of the resulting expression.

We start with the following proposition.

\begin{proposition}\label{proposition_iteration_of_DI} Take any $p\ge 1$ set partitions
$s^1,\dots, s^p$, and let $s^i=(S_1^i,S_2^i,\dots)$, $i=1,\dots,p$. Then
\begin{multline}
\label{eq_product_of_integral_operators}
  \DI^{s^p}_N \cdots \DI^{s^1}_N \left(\prod_{m=1}^N g(y_i)\right)=
  \left(\prod_{m=1}^N g(y_i)\right)
  \frac{1}{(2\pi\i)^{\sum_{j=1}^p \ell(s^j)}}\\\times \oint \prod_{1\le i<j \le p} \Cr(i,j) \cdot
   \prod_{j=1}^{p} G_{s^j}(w_1^j,\dots, w_{\ell(s^j)}^j;y_1,\dots,y_N;g),
\end{multline}
 where  each variable $w_i^j$ is integrated over the (positively oriented, encircling
  $\{y_m\}$)
contour
 $\mathcal C^j_i(y;g)$. The contours are nested by lexicographical order on $(j,i)$ (the $(1,1)$
 contour is the largest one) and are admissible with large enough distance parameter; the function $g(\cdot)$ is assumed
 to be such that the contours exist. Furthermore,
 $$
  \Cr(i,j)=
\prod_{a=1}^{\ell(s^i)} \prod_{b=1}^{\ell(s^j)} \prod_{c=1}^{|S^j_b|} \frac
 {w_{a}^i-w^j_b-\theta c+\theta |S_a^i|}{w_{a}^i-w^j_b-\theta c} \cdot
 \frac{w_{a}^i-w^j_b-\theta c+1}
 {w_{a}^i-w^j_b-\theta c+\theta |S_a^i|+1}.
 $$
\end{proposition}
\begin{proof}
The formula is obtained by iterating the application of operators $\DI^s_N$. First, we apply
$\DI^{s^1}_N$ using Definition \ref{def_Integral_operator} and renaming the variable $w_a$,
$a=1,\dots,\ell(s^1)$ into $w_a^1$. The $y$--dependent part in the right-hand side of
\eqref{eq_DI_operator} is $\prod_{m=1}^N g'(y_m)$ with
\begin{multline*}
g'(y_m):= g(y_m) \prod_{a=1}^{\ell(s^1)} \frac{w_{a}^1-y_m-\theta}
 {w_{a}^1-y_m}\cdot
 \frac{w_{a}^1-y_m} {w_{a}-y_m+\theta} \cdots
\frac{w_{a}^1-y_m+\theta(|S_a^1|-2)} {w_{a}^1-y_m+\theta(|S_a^1|-1)} \\=
 g(y_m) \prod_{a=1}^{\ell(s^1)} \frac{w_{a}^1-y_m-\theta}
 {w_{a}^1-y_m+\theta(|S_a^1|-1)},
\end{multline*}
where $s^1=(S_1^1,S_2^1,\dots)$.  We have
\begin{equation}
\label{eq_x22}
 \frac{g'(v-1)}{g'(v)}= \frac{g(v-1)}{g(v)}
 \prod_{a=1}^{\ell(s^1)}  \frac
 {w_{a}^1-v+\theta(|S_a^1|-1)}{w_{a}^1-v-\theta} \cdot \frac{w_{a}^1-v+1-\theta}
 {w_{a}^1-v+1+\theta(|S_a^1|-1)}.
\end{equation}
 In order to iterate, this function  must not have
poles inside the contour of integration. To achieve this it suffices to choose on the next step
the contours which are much closer to $y_m$ (i.e.\ at the first step the contours are large, on
the second step they are smaller, etc).

When we further apply $\DI^{s^2}_N$, the product
\begin{equation}
\label{eq_x21}
 \prod_{b=1}^{\ell(s^2)} \frac{g'(w^2_b-1)}{g'(w^2_b)} \cdots \frac{g'(w^2_b-1+\theta(|S^2_b|-1))}{g'(w^2_b+\theta(|S^2_b|-1) )}
\end{equation}
appears, where $s^2=(S^2_1,S^2_2,\dots)$. Substituting the $g$--independent part of the right side
of \eqref{eq_x22} into \eqref{eq_x21} we get $\Cr(1,2)$. Further applying operators
$\DI^{s^3}$,\dots, $\DI^{s^p}$ we arrive at \eqref{eq_product_of_integral_operators}.
\end{proof}

Our next step is to expand the products in \eqref{eq_P_through_Integral} using Proposition
\ref{proposition_iteration_of_DI} to arrive at a big sum. Each term in this sum involves an
integral encoded by a collection of set partitions $s^1,\dots, s^p$. The dimension of integral
equals $\sum_{j=1}^p \ell(s^j)$ and each integration variable corresponds to one of the sets in
one of partitions $s^j$. Let us enumerate the nested contours of integration by numbers from $1$
to $\sum_j \ell(s^j)$ (contour number $1$ is the largest one as above) and rename the variable on
contour number $i$ by $z_i$. Then the integrand has the following product structure:

\begin{enumerate}
\item For each $i$ there is a factor, which is a function of $z_i$.
The exact form of this function depends only on the size of the corresponding set in one of $s^j$.
We denote this ``multiplicative factor'' by $M_{r}(z_i)$ (where $r$ is the size of the
corresponding set).
\item For each pair of indices $i<j$ corresponding to two sets from
 the same set partition there is a factor, which is a function of $z_i$ and $z_j$. The exact
form of this factor depends on the sizes of the corresponding sets. We call is ''cross-factor of
type I'' and denote $Cr_{r,r'}^I(z_i,z_j)$ (where $r$ and $r'$ are the sizes of the corresponding
sets).
\item For each pair of indices $i<j$ corresponding to two sets from \emph{distinct} set
partitions there is a factor, which is a function of $z_i$ and $z_j$. The exact form of this
factor depends on the sizes of the corresponding sets. We say that this is the ''cross-factor of
type II'' and denote in by $Cr_{r,r'}^{II}(z_i,z_j)$ (where $r$ and $r'$ are the sizes of the
corresponding sets).
\end{enumerate}
Altogether there are $\sum_{j=1}^p\ell(s^j)$ ``multiplicative factors'' and
${\sum_{j=1}^p\ell(s^j)}\choose 2$ ``cross factors''.

Now we expand each cross-factor in power series in $(z_i-z_j)^{-1}$. We have
\begin{multline}
\label{eq_cross_I}
 Cr_{r,r'}^I(z_i,z_j)=\prod_{a=1}^r\prod_{b=1}^{r'} \frac{
 \big(z_i-z_j+\theta(a-b)\big)^2}{\big(z_i-z_j+\theta(a-b)\big)^2-\theta^2}=
  1+\sum_{m=2}^{k-1} \frac{a^I_{r,r'}(m)}{(z_i-z_j)^m} +  \frac{A^I_{r,r'}(z_i-z_j)}{(z_i-z_j)^k},
\end{multline}
where $A^I_{r,r'}(z_i-z_j)$ tends to a finite limit as $(z_i-z_j)\to\infty$.
Also
\begin{multline}
\label{eq_cross_II}
 Cr_{r,r'}^{II}(z_i,z_j)=
\prod_{c=1}^{r'}\left( \frac
 {z_i-z_j-\theta c+\theta r }{z_i-z_j-\theta c} \cdot
 \frac{z_i-z_j-\theta c+1}
 {z_i-z_j-\theta c+\theta r+1}\right)\\=
 1+\sum_{m=2}^{k-1} \frac{a^{II}_{r,r'}(m)}{(z_i-z_j)^m} +
 \frac{A^{II}_{r,r'}(z_i-z_j)}{(z_i-z_j)^k},
\end{multline}
where $A^{II}_{r,r'}(z_i-z_j)$ tends to a finite limit as $(z_i-z_j)\to\infty$. It is crucial for
us that both series have no first order term.

We again substitute the above expansions into \eqref{eq_P_through_Integral} and expand everything
into even bigger sum of integrals. Our next aim is to provide a description for a general term in
this sum. The general term is related to several choices that we can make:
\begin{equation}
\begin{array}{ll}
\label{List_of_choices}
\text{1.}& \text{Young diagram } \mu;\\
\text{2.}& \text{Permutation }\sigma\text{ of the set }\{1,\dots,\ell(\mu)\};\\
\text{3.}& \text{For each }1\le i\le\ell(\mu)\text{, the set partition }s^i\in \S_{\mu_{\sigma(i)}};\\
\text{4.}& \text{For each pair of sets of (the same or different) set partitions, one of }\\
&\text{the terms in expansions \eqref{eq_cross_I}, \eqref{eq_cross_II}.}
\end{array}
\end{equation}

It is convenient to illustrate all these choices graphically. For that purpose we take $k$
vertices enumerated by numbers $1,\dots,k$. They are united into groups (``clusters'') of lengths
$\mu_{\sigma(i)}$, i.e.\
 the first group has the vertices $1,\dots,\mu_{\sigma(1)}$, the second one has
 $\mu_{\sigma(1)}+1,\dots,\mu_{\sigma(1)}+\mu_{\sigma(2)}$, etc. This symbolizes the choices of
 $\mu$ and $\sigma$. Some of the vertices inside clusters are united into multivertices --- this
 symbolizes the set partitions. Finally, each pair of multivertices can be either not joined by an
 edge, which means the choice of term $1$ in \eqref{eq_cross_I} or \eqref{eq_cross_II}, or joined
 by a red edge with label $m$, which means the choice of term with $(z_i-z_j)^m$ in
 \eqref{eq_cross_II}, or joined by a black edge with label $m$, which means the choice of term with
 $(z_i-z_j)^m$ in \eqref{eq_cross_I}. An example of such a graphical illustration is shown at
 Figure \ref{Fig_graph_full}.
\begin{figure}[h]
\begin{center}
{\scalebox{1.0}{\includegraphics{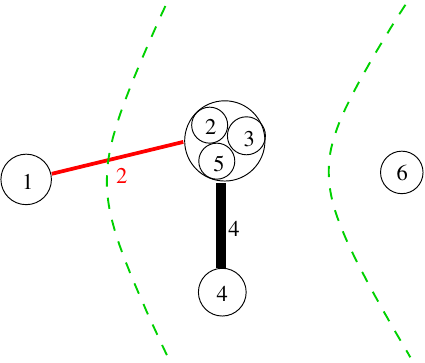}}} \caption{Graphical illustration for a
general integral term. Here $\mu=(4,1,1)$, so we have $3$ clusters (separated by green dashed
contours), the set partition from $\S_4$ has one set of size $3$ and one set of size $1$. The edges
are shown in thin red and thick black. $b_1$, $b_2$, $b_3$ label three black connected components
of the underlying graph. \label{Fig_graph_full} }
\end{center}
\end{figure}
 An integral is reconstructed by the picture via the following
 procedure: Each integration variable $z_i$ corresponds to one of the multivertices; the variables are
 ordered by the minimal numeric labels of vertices they contain and are integrated over admissible nested
 contours with respect to this order. As described above, for each variable $z_i$ we have a multiplicative factor
$M_r(z_i)$, where $r$ is the number of vertices in the corresponding multivertex (this number $r$
will be called the \emph{rank} of multivertex). For each pair of variables $z_i$, $z_j$, if
corresponding vertices are joined by an edge, then we also have a cross--term depending on the
color and label of the edge:
\begin{enumerate}
\item The edge is black if the term comes from $Cr_{r,r'}^I(z_i,z_j)$ and red if the
term comes from $Cr_{r,r'}^{II}(z_i,z_j)$.
\item The number $m\ge 2$ indicates the power of $(z_i-z_j)^{-1}$ in
\eqref{eq_cross_I} or \eqref{eq_cross_II}.
\end{enumerate}

Next, we note that many features of the picture are irrelevant for the resulting integral (in
other words, different pictures might give the same integrals). These features are:
\begin{enumerate}
\item Decomposition into clusters;
\item All numbers on each multivertex except for the minimal one;
\item The order (i.e.\ nesting of integration contours) between different connected components; i.e.\ only the order inside each component
matters.
\end{enumerate}

So let us remove all of the above irrelevant features of the picture. After that we end up with
the following object that we denote by $G$: We have a collection of multivertices; each
multivertex should be viewed as a set of $r$ ordinary vertices (we call $r$ the \emph{rank}). Some
of the multivertices are joined by red or black edges with labels. Each connected component of the
resulting graph has a linear order on its multivertices. In other words, there is a partial order
on multivertices of $G$ such that only vertices in the same connected component can be compared.
We call the resulting object \emph{labeled graph}. The fact that a labeled graph appeared from the
above graphical illustration of the integral implies the following properties:
\begin{enumerate}
\item Sum of the ranks of all multivertices is $k$;
\item Multivertices of one black connected component can not be joined by a red edge (because they
came from the same cluster);
\item Suppose that $A,B,C$ are three multivertices of one (uncolored) connected component and, moreover, $A$ and $B$
belong to the same black connected component. In this case $A<C$ if and only if $B<C$ (thus, also
$A>C$ iff $B>C$).
\end{enumerate}
The labeled graph obtained after removing all irrelevant features of the picture of Figure
\ref{Fig_graph_full} is shown in Figure \ref{Fig_graph}.

\begin{figure}[h]
\begin{center}
{\scalebox{1.0}{\includegraphics{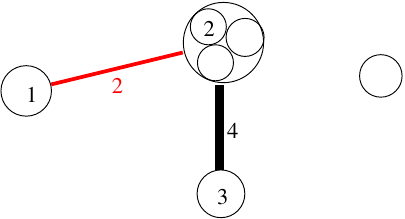}}} \caption{Graph with one multivertex of rank
$3$ and three multivertices of rank $1$ \label{Fig_graph} }
\end{center}
\end{figure}

Given a labeled graph $G$ we can reconstruct the integral by the same procedure as before, we
denote the resulting integral by $\mathcal I (G).$ In particular, graph $G$ with a single
multivertex of rank $k$ corresponds to $\DI^{[k]}_N$ appearing in Theorem
\ref{theorem_decomposition_of_P}.

Note that in our sum each integral corresponding to a given graph $G$ comes with a prefactor
$$
  \prod_{v \text{ is a multivertex of }G} (r(v)-1)!.
$$
It is important for us that this prefactor depends only on the labeled graph $G$ (and not on the
data we removed to obtain $G$). Because of that property we can forget about the prefactor when
analyzing the sum of the integrals corresponding to a given graph.

\subsection{Step 3: Cancelations in terms corresponding to a given labeled graph}

Our aim now is to compute the total coefficient of $ \mathcal I (G)$ for a given graph $G$, i.e.\
we want to compute the (weighted) sum of all integrals corresponding to $G$. As we will see, for
many graphs $G$ this sum vanishes.

\medskip

First, fix some $\mu$ with $\ell(\mu)=n$ and a permutation $\sigma\in \mathfrak S(n)$
(equivalently fix two out of four choices in \eqref{List_of_choices}). Let $W(G,\mu,\sigma)$
denote the number of integral terms corresponding to a graph $G$ and these two choices; in other
words, this is the number of ways to make the remaining two choices in \eqref{List_of_choices} in
such a way as to get the integral term of the type $\mathcal I (G)$.

Let $\widehat G$ denote the subgraph of $G$ whose multivertices either have rank at least $2$ or
has at least one edge attached to it (i.e.\ we exclude
 multivertices of rank $1$ that have no adjacent red or black edges).
Let $B$ denote the set of all black connected components of $\widehat G$\footnote{Up to now we
were considering graphs $G$ and $\widehat G$ up to isomorphism. But here, in order to define $B$
and then $\phi$, we \emph{fix} some representative of the isomorphism class of $\widehat G$. This
choice is also responsible for the factor $|Aut(\widehat G)|^{-1}$ in \eqref{eq_sum_over_phi}.}.
Note that each black component arises from one of the clusters (which correspond to the
coordinates of $\mu$) according to our definitions. Thus, each element of $B$ corresponds to one
of the coordinates of $\mu$, i.e.\ we have a map $\phi$ from $B$ into $\{1,\dots,n\}$ (each $i$
will further correspond to $\mu_{\sigma(i)}$). This map must satisfy the following property: if
the partial order on $G$ is such that members of a black component $b_1$ precede members of a
black component $b_2$, then $\phi(b_1)<\phi(b_2)$. In other words, $B$ is equipped with a partial
order (which is a projection of the partial order on $G$) and $\phi$ is an order-preserving map.
Different $\phi$'s correspond to different cluster configurations that $G$ may have originated
from. Let $\Phi=\Phi(G,n)$ denote the set of such maps $\phi$.

{\bf Example:} Consider the graph of Figure \ref{Fig_graph}. It has three connected black
components $b_1,b_2,b_3$ shown in Figure \ref{Fig_graph_full}: $b_1$ has 2 multivertices of ranks
$3$ and $1$, $b_2$ has one multivertex of rank $1$ (joined by a red edge), $b_3$ has one isolated
multivertex of rank $1$. Therefore, $B$ consists of two elements $B=\{b_1,b_2\}$ (note that the
isolated multivertex $b_3$ or rank $1$ was excluded). The partial order has the only inequality
$b_2 < b_1$. Also take $\mu$ to be a partition with two nonzero parts. Now $\phi$ should be a map
from $\{b_1,b_2\}$ to $\{1,2\}$ such that $\phi(b_2)<\phi(b_1)$. This means that $\phi(b_2)=1$ and
$\phi(b_1)=2$. Therefore, there is a single such $\phi$, see Figure \ref{Fig_graph_clusters}.

\begin{figure}[h]
\begin{center}
{\scalebox{1.0}{\includegraphics{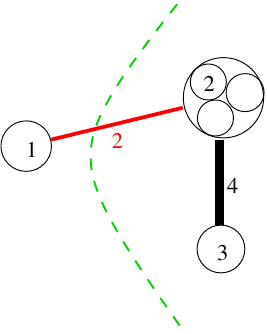}}} \caption{Graphical illustration for
the unique  possible $\phi$ when  $\mu$ has two parts. Equivalently, this is a unique
 decompositions of graph of Figure \ref{Fig_graph} (without single isolated vertices) into two clusters.
 \label{Fig_graph_clusters} }
\end{center}
\end{figure}

Now suppose that $\phi$ is fixed, and let $W(G,\mu,\sigma,\phi)$ denote the number of integrals
corresponding to it. We have
\begin{equation}
\label{eq_sum_over_phi} W(G,\mu,\sigma)=\frac{1}{|Aut(\widehat G)|}\sum_{\phi\in\Phi(G,n)}
W(G,\mu,\sigma,\phi),\quad \quad n=\ell(\mu),
\end{equation}
where $Aut(\widehat G)$ is the group of all automorphisms of graph $\widehat G$.

 Let us call a
vertex $v$ of $G$ \emph{simple} if $v$ is isolated and the corresponding multivertex has rank $1$.
\begin{lemma}
\label{lemma_non_sym_poly_phi}
 $W(G,\mu,\sigma,\phi)$ is a polynomial in $n$ variables $\mu_{\sigma(i)}$ (with coefficients depending on $G$ and $\phi$ only)
 of degree equal to the number of
 non--simple vertices in $G$, i.e.\ the number of vertices in $\widehat G$.
\end{lemma}
\begin{proof}

 For
each coordinate $\mu_{\sigma(i)}$ of $\mu$ we have chosen (via $\phi$) which black components of
$B$ belong to it. After that, we claim that the total number number of ways to choose a set
partition $s\in\S_{\mu_{\sigma(i)}}$ and factors from the integral corresponding to it in such a
way as to get factors corresponding to these black components, is equal to a polynomial in
$\mu_{\sigma(i)}$ depending on the set $\phi^{-1}(i)$; we use the notation
$P_{\phi^{-1}(i)}(\mu_{\sigma(i)})$ for it. Indeed, if the black components of $\phi^{-1}(i)$ have
$d$ vertices altogether (recall that all these vertices are not simple), then we  choose $d$
(unordered) elements out of the set with $\mu_{\sigma(i)}$ elements; there are $\mu_{\sigma(i)}
\choose d$ ways to do this. After that there is a fixed (depending solely on the set $\phi^{-1}(i)$
and $G$) number of ways to do all the other choices, i.e.\ to choose a set partition of these $d$
elements which would agree with $G$ on $\phi^{-1}(i)$. Therefore,
\begin{equation}
\label{eq_x8}
 P_{\phi^{-1}(i)}(\mu_{\sigma(i)})=  c(\phi^{-1}(i)) {\mu_{\sigma(i)} \choose d}.
\end{equation}
Note that this polynomial automatically vanishes if $\mu_{\sigma(i)}<d$. It is also convenient to
set $P$ to be $1$ if $\phi^{-1}(i)=\emptyset$. Since all the choices for different $i$ are
independent, and there is always a unique way to add red edges required by $G$ to the picture, we
conclude that the total number of integrals is
$$
 \prod_i P_{\phi^{-1}(i)} (\mu_{\sigma(i)}).
$$
Note that this is a polynomial of $\mu_i$ of degree equal to the total number of non-simple
vertices in $G$. \end{proof}

{\bf Example:} Continuing the above example with the graph of Figure \ref{Fig_graph}, for the
unique
 $\phi$, multivertertices of $b_1$ are in one cluster (corresponding to
$\mu_{2}$) and multivertex of $b_2$ is in another cluster (corresponding to $\mu_1$). In the first
cluster of size $\mu_2$ we choose 4 elements (corresponding to 4 vertices of $b_1$); there are
$\mu_2(\mu_2-1)(\mu_2-2)(\mu_2-3)/24$ ways to do this. Having chosen these 4 elements we should
subdivide them into those $3$ corresponding to the rank 3 multivertex and $1$ corresponding to the
rank $1$ multivertex in $b_1$. When doing this we have a restriction: rank $1$ multivertex should
have a greater number, than the minimal number among the vertices of the rank $3$ multivertex.
This simply means that the rank $1$ multivertex is either number $2$, number $3$ or number $4$ of
our set with $4$ elements. Thus,
$$
 P_{\phi^{-1}(2)}(\mu_2)=3{\mu_2\choose 4}.
$$
In the second cluster of size $\mu_1$ we choose one element corresponding to $b_2$. There are
$\mu_1$ ways to do this. Note that we ignore all simple vertices. Indeed, there is no need to
specify what's happening with simple vertices: the parts of set partitions corresponding to them
are just a decomposition of a set into singletons, which is automatically uniquely defined as soon
as we do all the choices for non-simple vertices. We conclude that $\phi$ yields the following
polynomial of degree $5$
$$
 P_{\phi^{-1}(1)}(\mu_1)P_{\phi^{-1}(2)}(\mu_2)=\frac{3}{24} \mu_1 \mu_2(\mu_2-1)(\mu_2-2)(\mu_2-3).
$$

We proceed summing over all $\phi\in\Phi(G,n)$, $n=\ell(\mu)$.
\begin{lemma}
\label{lemma_non_sym_poly}
 $W(G,\mu,\sigma)$ is a polynomial in $n$ variables $\mu_{\sigma(i)}$ (with coefficients depending on $G$)
 of degree equal to the number of
 non--simple vertices in $G$.

 Moreover, if $G$ has an isolated multivertex of degree $r>1$, then the highest
 order component of $W(G,\mu,\sigma)$ is divisible by $\sum_{i=1}^n (\mu_i)^r$.
\end{lemma}
\begin{proof}
 The first part is an immediate consequence of \eqref{eq_sum_over_phi} and Lemma \ref{lemma_non_sym_poly_phi}.

As for the second part, fix an isolated multivertex $v$ of $G$ of degree $r$.
 Consider the process of constructing the polynomial for $W(G\setminus
v;\mu,\sigma)$ through \eqref{eq_sum_over_phi} and Lemma \ref{lemma_non_sym_poly_phi}. Note that
any order-preserving $\phi_{G\setminus v}\in\Phi(G\setminus v,n)$ for the graph $G\setminus v$
corresponds to exactly $n=\ell(\mu)$ order-preserving $\phi_G$'s for $G$: they differ by the image
of $v$, while images of all other black components are the same as in $\phi_{G\setminus v}$.
Moreover, if $\phi_G(v)=t$, then the polynomial corresponding to $\phi_G$ is the one for
$\phi_{G\setminus v}$
 times
 $$
  {{\mu_{\sigma(t)}-d} \choose r },
$$
 where $d$ is
 the total number of non-simple vertices in the black components of
$\phi_{G\setminus v}^{-1}(t)$, as in Lemma \ref{lemma_non_sym_poly_phi}. We conclude that the
highest degree term of the sum of all polynomials corresponding to fixed $\phi_{G\setminus v}$ and
various choices of $t=1,\dots,n$ is divisible by $\sum \mu_i^r$. Clearly, this property survives
when we further sum over all $\phi_{G\setminus v}$.
\end{proof}

Now let us also sum over $\sigma$. For a graph $G$, let $U(G,\mu)$ denote $\frac{1}{\ell(\mu)!}$
times the total number of the integrals given by graph $G$ in the decomposition of
$$\sum_{\sigma\in\mathfrak S_{\ell(\mu)}}
\prod_{i=1}^{\ell(\mu)} \D^{\mu_{\sigma(i)}}_N.$$

\begin{lemma}
\label{Lemma_order_of_polynomial} For any labeled graph $G$ there exists a symmetric function
$f_G(x_1,x_2,\dots)\in\Lambda$ of degree equal to the number of non--simple vertices in $G$, such
that
$$
 U(G,\mu)=f_G(\mu_1,\dots,\mu_{\ell(\mu)},0,0,\dots).
$$
Moreover, if $G$ has an isolated multivertex of rank $r$ then the highest degree component of
$f_G$ is divisible by $p_r=\sum_i x_i^r$.
\end{lemma}
\begin{proof} Combining Lemma \ref{lemma_non_sym_poly} and identity
$$
\sum_{\sigma\in \mathfrak S(n)} \frac{1}{n!} W(G,\mu,\sigma) = U(G,\mu), \quad\quad n=\ell(\mu)
$$
we get a symmetric polynomial $f^n_G(x_1,\dots,x_N)$ (of desired degree and with the desired
divisibility of highest degree component) defined by
$$
 f_G^n(\mu_1,\dots,\mu_n)=U(G,\mu).
$$
Note that if we now add a zero coordinate to $\mu$ (thus, extending its length by $1$), then the
polynomial does not change, i.e.
\begin{equation}
\label{eq_stability}
 f_G^{n+1}(x_1,\dots,x_n,0)=f_G^n(x_1,\dots,x_n).
\end{equation}
Indeed, when $\mu_{n+1}=0$, \eqref{eq_x8} vanishes unless $\phi^{-1}(\sigma^{-1}(n+1))=\emptyset$,
therefore, the summation over $\Phi(G,n+1)$ is in reality over $\phi$ such that
$\phi^{-1}(\sigma^{-1}(n+1))=\emptyset$. These are effectively $\phi$'s from $\Phi(G,n)$  and the
sum remains stable.

The property \eqref{eq_stability} implies that the sequence $f_G^n$, $n=1,2,\dots$ defines an
element $f_G\in\Lambda$, cf.\ \cite[Chapter I, Section 2]{M}.
\end{proof}

\bigskip

The next crucial cancelation step is given in the following statement.

\begin{proposition}
\label{proposition_sum_is_zero}
 Take $k>1$ and let $f\in\Lambda$ be a symmetric function of degree at most $k$
 such that $f\in\mathbb C [p_1,\dots,p_{k-1}]$. We have
 $$
   \sum_{\mu\in\Y_k} \frac{PE(k,\mu)f(\mu_1,\mu_2,\dots) }{\prod_i \mu_i!}=0,
 $$
 where $\Y_k$ is the set of all partitions $\mu$ with $|\mu|=k$.
\end{proposition}
\begin{proof}
 Let $\Psi:\Lambda\to \mathbb C$
 be an (algebra-) homomorphism sending $p_1\mapsto 1$ and $p_k\mapsto 0$, $k\ge
2$.
 With the notation
$$
 e_\mu=\prod_{i=1}^{\ell(\mu)} e_{\mu_i},
$$
where $e_m,$ is the elementary symmetric function of degree $m$, we have
 $$
 \Psi(e_\mu)= \frac{1}{\prod_i \mu_i!},
 $$
as follows from the identity of generating series (see e.g.\ \cite[Chapter I, Section 2]{M})
$$
 \sum_{k=0}^\infty e_k z^k =\exp\left(-\sum_{k=1}^{\infty} \frac{p_k(-z)^k}{k}\right).
$$
 Let $d[i_1,\dots,i_m]$ denote the differential operator on $\Lambda$ (which is viewed as an algebra of
polynomials in $p_k$ here) of the form
 $$
  d[i_1,\dots,i_m]=\prod_{j=1}^m (-1)^{i_j} i_j \frac{\partial}{\partial p_{i_j}}, \quad
  i_1,\dots,i_j\ge 1.
 $$
 Then for any $n\ge 1$ (see \cite[Chapter I, Section 5, Example 3]{M})
$$
 d[n]e_r=(-1)^{n} n \frac{\partial}{\partial p_{n}} (e_r)=\begin{cases} e_{r-n},& r\ge n,\\ 0,&\text{
 otherwise.} \end{cases}
$$
Let $E_\mu=\prod_{i=1}^{\ell(\mu)} \left({e_{\mu_i}}{\mu_i!}\right)$ (note that $\Psi(E_\mu)=1$).
Then
$$
 d[n]E_\mu=\sum_{j\ge 1} E_{\mu-n v_j} \mu_j(\mu_j-1)\cdots (\mu_j-n+1),
$$
where $v_i$ is the vector with the $i$th coordinate equal to $1$ and all other coordinates equal
to $0$, and we view $\mu$ as vector $(\mu_1,\mu_2,\dots)$. More generally,
$$
 d[i_1,\dots,i_m] E_\mu= \sum_{j_1,\dots,j_m\ge 1} E_{\mu-i_1 v_{j_1}-i_2 v_{j_2}-\dots} \prod_{q=1}^m
 \mu^{(q)}_{j_q}(\mu_{j_q}^{(q)}-1)\cdots (\mu^{(q)}_{j_q}-i_q+1),
$$
where
$$
\mu^{(q)}= \mu-i_{q+1} v_{j_{q+1}}-i_{q+2} v_{j_{q+2}}-\dots,\quad 1\le q\le m,
$$
and, in particular, $\mu^{(m)}=\mu$. Observe that $\Psi (d[i_1,\dots,i_m] E_\mu)$ is a symmetric
polynomial in $\mu_i$ with highest degree part being $p_{i_1}\cdots p_{i_m}$. Therefore, any
symmetric function $f\in\Lambda$ in variables $\mu_i$ of degree at most $k$ and such that
$f\in\mathbb C[p_1,\dots,p_{k-1}]$ can be obtained as a linear combination of $\Psi
(d[i_1,\dots,i_m] E_\mu)$ with $i_j<k$. (Indeed, after we subtract a linear combination of $\Psi
(d[i_1,\dots,i_m] E_\mu)$ that agrees with the highest degree part of $f$, we get a polynomial $f'$
of degree at most $k-1$, which is automatically in $\mathbb C[p_1,\dots,p_{k-1}]$. After that we
repeat for $f'$, etc.). It remains to apply this linear combination of $\Psi \circ
d[i_1,\dots,i_m]$ to the identity
$$
p_k=\sum_{\mu\in\Y_k} PE(k,\mu) \frac{E_\mu}{\prod_i \mu_i!}.\qedhere
$$
\end{proof}

Now we are ready to prove the Theorem \ref{theorem_decomposition_of_P}.

\begin{proof}[Proof of Theorem \ref{theorem_decomposition_of_P}]
As we explained above, labeled graph $G$ defines an admissible integral operator $\mathcal I (G)$.
The dimension of this operator is equal to the number of multivertices in $G$ and the degree
equals the sum of labels of all edges of $G$. We will show that the sum in Theorem
\ref{theorem_decomposition_of_P} is over the set $\mathcal G$ of all labeled graphs with $k$
vertices and such that the dimension minus the degree of the corresponding integral operator is a
non-positive number.

\smallskip

We start from the decomposition \eqref{eq_P_through_Integral} of $\P^k_N$ into the sum of
integrals. Note that the coefficient of $\DI^{[k]}_N$  is $(-\theta)^{-k}\cdot PE(k,k)/k$. The
second factor is $PE(k,k)=(-1)^{k-1}k$, therefore, the coefficient of $\DI^{[k]}_N$ is
$-(\theta)^{-k}$.

We further use Proposition \ref{proposition_iteration_of_DI} and then expand crossterms in the
resulting integrals as in \eqref{eq_cross_I}, \eqref{eq_cross_II}. Note that $A^I_{r,r'}(z_i-z_j)$
and $A^{II}_{r,r'}(z_i-z_j)$ in these expansions by the very definition have limits as
$z_i-z_j\to\infty$. Since all the integrals in our expansion are at most $k$--dimensional, all
terms where $A^I_{r,r'}(z_i-z_j)$ or $A^{II}_{r,r'}(z_i-z_j)$ are present satisfy the assumption
that dimension minus degree is non-positive.

As is explained above, all other terms in the expansion are enumerated by certain labeled graphs.
Our aim is to show that if the contribution of a given graph is non-zero, then the dimension minus
degree of the corresponding integral operator is non-positive. If a graph $G$ has less than $k$
non-simple vertices, then combining Lemma \ref{Lemma_order_of_polynomial} with Proposition
\ref{proposition_sum_is_zero} we conclude that the contribution of this graph vanishes. On the
other hand, if a graph $G$ has $k$ non-simple vertices which form $M\le k$ multivertices, then the
corresponding integral is $M$--dimensional. If $G$ has no isolated multivertices, then it has at
least $\lceil M/2 \rceil$ edges. Since each edge increases the degree at least by $2$, we conclude
that the dimension minus degree of the corresponding integral operator is non-positive. Finally, if
$G$ has $k$ vertices and an isolated multivertex of degree less than $k$ (isolated multivertex of
degree $k$ corresponds precisely to $\DI^{[k]}_N$), then we can again use Lemma
\ref{Lemma_order_of_polynomial} and Proposition \ref{proposition_sum_is_zero} concluding that the
contribution of this graph vanishes.
\end{proof}

\section{Central Limit Theorem}

\label{Section_CLT_general}

\subsection{Formulation of GFF-type asymptotics}

The main goal of this section is to prove the following statement.

\begin{theorem}
\label{theorem_joint_CLT}
 Suppose that we have a large parameter $L$, and parameters $M\ge 1$, $\alpha>0$,
  $N_1\ge N_2\ge \dots \ge N_h\ge 1$
  grow linearly in it, i.e.
 $$
  M \sim L \hat M,\quad \alpha \sim L \hat\alpha ,\quad N_i\sim L\hat N_i,\quad
  L\to\infty;\quad\quad \hat M\ge 0,\, \hat\alpha \ge 0,\, N_h>0.
 $$
 Let $\r\in\St^M$ be distributed according to $\Pr^{\alpha,M,\theta}$ of Definition \ref{Def_distrib}.
 Then for any integers $k_i\ge 1$, $i=1,\dots,h$, the random vector
 $\left(p_{k_i}(N_i;\r)-\E p_{k_i}(N_i;\r)\right)_{i=1}^h$ converges, in the sense of joint
 moments, thus weakly, to the Gaussian vector with mean $0$ and covariance given by
\begin{multline}
\label{eq_limit_covariance} \lim_{L\to\infty} \E\bigg(\left[ p_{k_1}(N_1;\r)-\E
p_{k_1}(N_1;\r)\right]\left[p_{k_2}(N_2;\r)-\E p_{k_2}(N_2;\r)\right]\bigg)\\=
 \frac{\theta^{-1}}{(2\pi \i)^2}
 \cdot \oint\oint \frac{du_1 du_2}{(u_1-u_2)^2}
 \prod_{r=1}^2 \left( \frac{u_r}{(u_r+\hat N_r)} \cdot \frac{u_r-{\hat\alpha}}{(u_r-\hat \alpha-\hat
 M)}
 \right)^{k_r},
\end{multline}
where both integration contours are closed and positively oriented, they enclose the poles of the
integrand at $u_1=-\hat N_1$, $u_2=-\hat N_2$ ($\hat N_1\ge \hat N_2,$ as above), but not at
$u_r=\hat\alpha+\hat M$, $r=1,2$, and $u_2$--contour is contained in the $u_1$ contour.
\end{theorem}
{\bf Remark 1.} We can change $p_k(N;\r)$ in the statement of theorem by removing $N-M$ ones from
its definition (cf.\ Definition \ref{definition_observables}) when $N>M$.

{\bf Remark 2.} The limit covariance depends on $\theta$ only via the prefactor $\theta^{-1}$.

{\bf Remark 3.} Our methods also give the asymptotics of $\E p_{k_i}(N_i;\r)/L$, which provides a
limit shape theorem (or law of large numbers) for $\beta$--Jacobi ensemble. We do not pursue this
further as this was already done in \cite{DP}, \cite{Killip}, \cite{Jiang}.

\medskip

In Sections \ref{section_warm_up}, \ref{Section_Gauss_lemma}, \ref{section_proof_of_CLT} we prove
Theorem \ref{theorem_joint_CLT}, and in Section \ref{section_identification} we identify the limit
covariance with that of a pullback of the Gaussian Free Field (whose basic properties are
discussed in Section \ref{Section_GFF}) and also give an alternative expression for the covariance
in terms of Chebyshev polynomials.

\subsection{A warm up: Moments of $p_1$.}
\label{section_warm_up}

In order to see what kind of objects we are working with take $N_1=N=\hat N L$ as in Theorem
\ref{theorem_joint_CLT} and consider the limit distribution of $p_1(N ;\r)$.

For $p_1$ the situation is simplified by the fact that
$$
 \P^1_N=\D^1_N=\DI^{[1]}_N.
$$

We study $\E[p_1(N;\r)]^m$ using Theorem \ref{Theorem_HO_expectations_p}. Applying $m\ge 1$ times
the operator $\P^1_N$, using Proposition \ref{proposition_iteration_of_DI} and changing the
variables $w^i_1=L u_i$ we arrive at the following formula
\begin{equation}
\label{eq_expectation_of_power}
 \E[p_1(N;\r)]^m = \frac{L^m(-\theta)^{-m}}{(2\pi\i)^m} \oint \dots\oint \prod_{i<j} Cr_L(u_i,u_j)
 \prod_{i=1}^m F_L(u_i) d u_i,
\end{equation}
where
$$
F_L(u)=\frac{H(Lu-1;\alpha,M)}{H(Lu;\alpha,M)}  \prod_{m=1}^N
\frac{Lu+(m-2)\theta}{Lu+(m-1)\theta},
$$
with $H$ from \eqref{eq_definition_of_H} given by
$$
 H(y;\alpha,M)= \prod_{i=1}^N \frac{\Gamma(-y+\theta\alpha)}{\Gamma(-y+\theta\alpha+M\theta)},
$$
so that
\begin{multline*}
F_L(u)=
  \frac{\Gamma(-Lu+\theta\alpha+M\theta)}{\Gamma(-Lu+1+\theta\alpha+M\theta)}
 \frac{\Gamma(-Lu+1+\theta\alpha)}{\Gamma(-Lu+\theta\alpha)}
 \prod_{m=1}^N \frac{Lu+(m-2)\theta}{Lu+(m-1)\theta}\\=
 \frac{Lu-{\theta}}{Lu+{(N-1)\theta}}\cdot
\frac{Lu- {\theta\alpha}}{Lu-{\theta\alpha-M\theta} } =
 \frac{u-\frac{\theta}{L}}{u+\frac{(N-1)\theta}{L}}\cdot
\frac{u-\frac{\theta\alpha}{L}}{u-\frac{\theta\alpha+M\theta}{L}}.
\end{multline*}
Also we have
$$
 Cr_L(u_1,u_2)=\frac{(u_1-u_2+\frac{1-\theta}{L})(u_1-u_2)}{(u_1-u_2+\frac{1}{L})(u_1-u_2-\frac{\theta}{L})},
$$
and the integration in \eqref{eq_expectation_of_power} is performed over nested contours (the
smallest index corresponds to the largest contour) enclosing the singularity of $F_L$ in
$-\frac{(N-1)\theta}{L}$.

Note that as $L\to\infty$, $F_L$ converges to an analytic limit $F$ given by
$$
 F(u)=\frac{u}{u+\theta\hat N}\cdot
\frac{u-\theta\hat \alpha}{u-\theta\hat \alpha-\theta\hat M}.
$$

Define the function $V_L(u_1,u_2)$ through
$$
 V_L(u_1,u_2)=Cr_L(u_1,u_2)-1.
$$
Note that as $L\to\infty$
$$
 V_L(u_1,u_2)\sim \frac{1}{L^2}\cdot \frac{\theta}{(u_1-u_2)^2}.
$$
Therefore, as $L\to\infty$,
\begin{multline}
\label{eq_limit_variance}
 \E[p_1(N;\r)]^2-[\E(p_1(N;\r))]^2= \frac{L^2(-\theta)^{-2}}{(2\pi\i)^2} \oint\oint V_L(u_1,u_2)
  F_L(u_1) F_L(u_2) d u_1 d u_2 \\ \to \frac{ \theta^{-1}}{(2\pi\i)^2} \oint\oint \frac{F(u_1) F(u_2)}{(u_1-u_2)^2} d
u_1 d
  u_2.
\end{multline}
Changing the variables $u_i=\theta u'_i$, $i=1,2$, we arrive at the limit covariance formula
\eqref{eq_limit_covariance} for $p_1$.

The proof of the fact that $p_1(N;\r)-\E p_1(N;\r)$ is asymptotically Gaussian follows from a
general lemma that we present in the next section.

\subsection{Gaussianity lemma}
\label{Section_Gauss_lemma}

Let us explain the features of the formula \eqref{eq_expectation_of_power} that are important for
us. The integration in \eqref{eq_expectation_of_power} goes over $m$ contours
$\boldsymbol{\gamma}=(\gamma_1,\dots,\gamma_m)$, such that $\boldsymbol{\gamma}$ belongs to a
certain class $\aleph_m$. As long as $\boldsymbol{\gamma}\in\aleph_m$, the actual choice of
$\boldsymbol{\gamma}$ is irrelevant. The crucial property of classes $\aleph_m$ is that if
$(\gamma_1,\dots,\gamma_m)\in\aleph_m$ and $1\le i_1<\dots<i_l\le m$, then
$(\gamma_{i_1},\gamma_{i_2},\dots,\gamma_{i_l})\in\aleph_l$. Further, if
$(\gamma_1,\dots,\gamma_m)\in\aleph_m$, then $F_L(u)$ converges to a limit function $F(u)$
uniformly over $u\in\gamma_i$, $i=1,\dots,m$, as $L\to\infty$, and also $V_L(u_1,u_2)\sim
\frac{1}{L^2}\cdot \frac{\theta}{(u_1-u_2)^2}$ uniformly over $(u_1,u_2)\in \gamma_i \times
\gamma_j$, $1\le i<j \le m$.

Let us now generalize the above properties. We will need this generalization when dealing with
$p_k(N;\r)$, $k\ge 2$, see Section \ref{section_proof_of_CLT}. Fix an integral parameter $q>0$ (in
the above example with $p_1$, $q=1$) and take $q$ ``random variables''\footnote{Throughout this
section random variable is just a name for the collection of moments, in other words, it is not
important whether moments that we specify indeed define a conventional random variable. To avoid
the confusion we write ``random variables'' and ``moments'' with the quotation marks when speaking
about such \emph{virtual} random variables.} $\xi_1(L),\dots,\xi_q(L)$ depending on an auxiliary
parameter $L$. Suppose that the following data is given. (In what follows we use the term
``multicontour'' for a finite collection of closed positively oriented contours in $\mathbb C$,
and we call the number of elements of a multicontour its dimension.)

\begin{enumerate}
\item For each $k=1,\dots,q$, we have an integer $l(k)>0$. In the above
example with $p_1$, $l(1)=1$

\item For any $n\ge 1$ and any $n$--tuple of integers $K=(1\le k_1\le k_2\le\dots\le k_n \le q)$,
we have a class of multicontours $\aleph_K$, such that $\boldsymbol{\gamma}\in\aleph_K$ is a
family of $n$ multicontours $\boldsymbol{\gamma}=(\gamma_1,\dots,\gamma_n)$ and $\gamma_i$ is a
$l(k_i)$--dimensional multicontour.

\item If $\boldsymbol{\gamma}=(\gamma_1,\dots,\gamma_n)\in\aleph_K$ and $1\le i_1<\dots<i_t\le n$, then
$(\gamma_{i_1},\gamma_{i_2},\dots,\gamma_{i_t})\in\aleph_{K'}$, where
$$
 K'=(k_{i_1}\le k_{i_2} \le \dots\le k_{i_t}).
$$

\item For each $k=1,\dots,q$, and each value of $L$ we have a continuous function of $l(k)$
variables: ${\mathfrak F}_L^k(u_1,\dots,u_{l(k)})$. If
$\boldsymbol{\gamma}=(\gamma_1,\dots,\gamma_n)\in\aleph_K$, $1\le i\le n$ and $k=k_{i}$, then
${\mathfrak F}_L^k(u_1,\dots,u_{l(k)})$ converges as $L\to\infty$ to a (continuous)
 function ${\mathfrak F}^k(u_1,\dots,u_{l(k)})$
uniformly over $(u_1,\dots,u_{l(k)})\in\gamma_i$.

\item For each pair $1\le k,r \le q$ and each value of $L$ we have a continuous function
 ${\mathfrak Cr}_L^{k,r}(u_1,\dots,u^i_{l(k)}; v_1,\dots,v_{l(r)})$.
 If $\boldsymbol{\gamma}=(\gamma_1,\dots,\gamma_n)\in\aleph_K$, $1\le i<j\le n$, and $a=k_i$, $b=k_j$, then
${\mathfrak Cr}_L^{a,b}(u_1,\dots,u_{l(a)}; v_1,\dots,v_{l(b)})$ converges as $L\to\infty$ to a
(continuous) function ${\mathfrak Cr}^{a,b}(u_1,\dots,u^i_{l(a)}; v_1,\dots,v_{l(b)})$ uniformly
over $(u_1,\dots,u_{l(a)})\in\gamma_i$, $(v_1,\dots,v_{l(b)})\in\gamma_j$.

\item For each $k=1,\dots,q$, we have certain ($L$--dependent) constants $c_L(k)$.

\item An additional real parameter $\gamma>0$ is fixed.
\end{enumerate}
Suppose now that for any $n\ge 1$ and any $n$--tuple of integers $K=(1\le k_1\le k_2\le\dots\le
k_n \le q)$, there exists $L(K)>0$ such that for $L>L(K)$ the joint ``moments'' of $\xi_i(L)$
corresponding to $K$ have the form
$$
 \E(\xi_{k_1}(L)\cdots\xi_{k_n}(L)) =\prod_{i=1}^n c_L(k_i)  \oint\cdots \oint \prod_{i<j} Cr_L(k_i,i;k_j,j) \prod_{i=1}^n F_L(k_i,i),
$$
where
$$
 F_L(k,i)={\mathfrak F}_L^k(u^i_1,\dots,u^i_{l(k)}) du^i_1\cdots du^i_{l(k)},
$$
$$
 Cr_L(k,i;r,j)=1+L^{-\gamma} V_L(k,i;r,j)=1+L^{-\gamma} {\mathfrak Cr}_L^{k,r}(u^i_1,\dots,u^i_{l(k)};
 u^j_1,\dots,u^j_{l(r)}),
$$
and the integration goes over \emph{any} set of contours $\boldsymbol{\gamma}\in \aleph_K$.

\begin{lemma}
\label{lemma_limit_gaussianity} In the above settings, as $L\to\infty$, the ``random vector''
$$\frac{\xi_1(L)-\E\xi_1(L)}{c_L(1)L^{-\gamma/2}},\dots,\frac{\xi_q(L)-\E\xi_q(L)}{c_L(k_i)L^{-\gamma/2}}
$$ converges
(in the sense of ``moments'') to the Gaussian random vector $\zeta_1,\dots,\zeta_q$ with mean $0$
and covariance (here $k\le r$)
\begin{multline*}
 \E\zeta_k\zeta_r= \oint\cdots\oint {\mathfrak Cr}^{k,r}(u^1_1,\dots,u^1_{l(k)};
 u^2_1,\dots,u^2_{l(r)})\\ \times {\mathfrak F}^k(u^1_1,\dots,u^1_{l(k)}) {\mathfrak F}^r(u^2_1,\dots,u^2_{l(r)})
 du^1_1\cdots du^1_{l(k)} du^2_1\cdots du^2_{l(r)},
\end{multline*}
where the integration goes over $\boldsymbol{\gamma}\in\aleph_{(k,r)}$. The answer does not depend
on the choice of $\boldsymbol{\gamma}\in\aleph_{(k,r)}$.
\end{lemma}
{\bf Remark.} Note that Lemma \ref{lemma_limit_gaussianity} is merely a manipulation with
integrals and their asymptotics, and we use ``moments'' just as names for these integrals.
\begin{proof}[Proof of Lemma \ref{lemma_limit_gaussianity}]
Take any $K=1\le k_1\le k_2\le\dots\le k_n \le q$ and let us compute the corresponding
\emph{centered} ``moment''. We have
\begin{multline}
\label{eq_x1} \E
\prod_{i=1}^n\frac{\xi_{k_i}-\E\xi_{k_i}}{c_L(k_i)L^{-\gamma/2}}\\=L^{n\gamma/2}\oint\cdots\oint\prod_{i=1}^n
F_L(k_i,i) \sum_{A\subset\{1,\dots,n\}} (-1)^{n-|A|} \prod_{i,j\in A,\,i<j}(1+L^{-\gamma}\cdot
V_L(k_i,i;k_j,j)),
\end{multline}
with integration over $\boldsymbol{\gamma}\in\aleph_K$ (here we use the hypothesis that classes
$\aleph$ are closed under the operation of taking subsets). For a set $A\subset\{1,\dots,n\}$ let
$A^{(2)}$ denote the set of all pairs $i<j$ with $i,j\in A$,
$$
 A^{(2)}=\{(i,j)\mid i,j\in A,\, i<j\}.
$$
Also for any set $B$ of \emph{pairs} of numbers, let $S(B)$ denote the set of all first and second
coordinates of elements from $B$ (support of $B$). With this notation \eqref{eq_x1} transforms
into
\begin{multline}
\label{eq_x2} L^{n\gamma/2} \oint \dots\oint
 \left(\prod_{i=1}^n F_L(k_i,i)\right) \sum_{A\subset\{1,\dots,n\}}(-1)^{n-|A|}
  \sum_{B\subset A^{(2)}} L^{-\gamma|B|}\prod_{(i,j)\in B} V_L(k_i,i;k_j,j)
\\= L^{n\gamma/2} \oint \dots\oint
 \left(\prod_{i=1}^n F_L(k_i,i)\right) \sum_{B\subset\{1,\dots,n\}^{(2)}} \prod_{(i,j)\in B} L^{-\gamma|B|} V_L(k_i,i;k_j,j)
 \sum_{A\mid\, S(B)\subset A} (-1)^{n-|A|}.
\end{multline}
Note that for any two finite sets $I_1\subset I_2$ we have
$$
 \sum_{A: \, I_1\subset A\subset I_2} (-1)^{|I_2|-|A|}= \begin{cases} 1, & I_1=I_2,\\
 0,&\text{otherwise.} \end{cases}
$$
Hence, \eqref{eq_x2} is
\begin{equation}
\label{eq_x3} L^{n\gamma/2} \oint \dots\oint
 \left(\prod_{i=1}^n F_L(k_i,i)\right) \sum_{B\subset\{1,\dots,n\}^{(2)} \mid\, S(B)=\{1,\dots,n\}}
 L^{-\gamma|B|}
 \prod_{(i,j)\in B} V_L(k_i,i;k_j,j).
\end{equation}
Note that the set of pairs $B$ such that $S(B)=\{1,\dots,n\}$ must have at least $\lceil
\frac{n}{2}\rceil$ elements. Therefore, if $n$ is odd, then the factor $L^{n\gamma/2}
L^{-\gamma|B|}$ in \eqref{eq_x3} converges to $0$. If $n$ is even, then similarly, the products
with $|B|>n/2$ are negligible. If $|B|=n/2$ then $B$ is just a perfect matching of the set
$\{1,\dots,n\}$. We conclude that for even $n$, \eqref{eq_x3} converges as $L\to\infty$ to
\begin{multline*}
 \sum_{B\in \text{perfect matchings of }\{1,\dots,n\}}  \prod_{(i,j)\in B}
 \oint\cdots\oint {\mathfrak Cr}^{k_i,k_j}(u^1_1,\dots,u^1_{l(k_i)};
 u^2_1,\dots,u^2_{l(r_i)})\\ \times {\mathfrak F}^{k_i}(u^1_1,\dots,u^1_{l(k_i)}) {\mathfrak F}^r(u^2_1,\dots,u^2_{l(r_i)}) du^1_1\cdots du^1_{l(k_i)} du^2_1\cdots
 du^2_{l(r_i)}.
\end{multline*}
This is precisely Wick's formula (known also as Isserlis's theorem, see \cite{Is}) for the joint
moments of Gaussian random variables $\zeta_1,\dots,\zeta_q$.
\end{proof}

\subsection{Proof of Theorem \ref{theorem_joint_CLT}}
\label{section_proof_of_CLT}

Throughout this section we fix  $k_1,\dots,k_h\ge 1$ and $N_1\ge N_2\ge\dots\ge N_h\ge 1$. Our aim
is to prove that the moments of the vector
$$
 (p_{k_1}(N_1;\r)-\E p_{k_1}(N_1;\r),\,  p_{k_2}(N_2;\r)-\E p_{k_2}(N_2;\r),\dots, p_{k_h}(N_h;\r)-\E p_{k_h}(N_h;\r) )
$$
converge to those of the Gaussian vector with covariance given by \eqref{eq_limit_covariance}.
Clearly, this would imply Theorem \ref{theorem_joint_CLT}.

 The proof is a combination of Theorem \ref{Theorem_HO_expectations_p}, Theorem
\ref{theorem_decomposition_of_P} and Lemma \ref{lemma_limit_gaussianity}.

Theorem \ref{theorem_decomposition_of_P} yields that operator $\P^k_N$ is a sum of $R=R(k)$ terms
with leading term being $(-\theta)^{-k}\DI^{[k]}_N$. Let us denote through $\P^k_N\{j\}$,
$j=1,\dots,R(k)$, all the terms in $\P^k_N$, with $j=1$ corresponding to
$(-\theta)^{-k}\DI^{[k]}_N$.

Now the joint moments of random variables $p_k(N;\r)$ (with varying $k$ and $N$) can be written as
(cf.\ Theorem \ref{Theorem_HO_expectations_p})
\begin{equation}
\label{eq_HO_expect_expanded}
 \E \left(\prod_{i=1}^m p_{k_i}(N_i;\r) \right) =  \dfrac{ \prod\limits_{i=1}^m \left(\sum_{j=1}^{R(k_j)} \P^{k_i}_{N_i}\{j\} \right)
 \left[\prod\limits_{i=1}^{N_m} H(y_i;\alpha,M)\right]}{\prod_{i=1}^{N_m} H(y_i;\alpha,M)} \rule[-5mm]{0.9pt}{17mm}_{\,
y_i=\theta(1-i)}.
\end{equation}

Introduce \emph{formal} ``random variables'' $p_{k}(N)\{j\}$, $j=1,\dots,R(k)$, such that
\begin{equation}
\label{eq_auxiliary_moments}
 \E \left(\prod_{i=1}^m p_{k_i}(N_i)\{j_i\} \right) =  \dfrac{ \prod\limits_{i=1}^m  \P^{k_i}_{N_i}\{j_i\}
 \left[\prod\limits_{i=1}^{N_m} H(y_i;\alpha,M)\right]}{\prod_{i=1}^{N_m} H(y_i;\alpha,M)} \rule[-5mm]{0.9pt}{17mm}_{\,
y_i=\theta(1-i)}.
\end{equation}
The word formal here means that at this time a set of ``random variables'' for us is just a
collection of numbers --- their joint ``moments''.

 Clearly, we have (formally, in the sense of ``moments'')
$$p_k(N;\r)=\sum_{j=1}^{R(k)} p_{k}(N)\{j\}.$$

\begin{lemma}
\label{lemma_vars_satisfy}
 ``Random variables''
 $p_{k}(N)\{j\}$ satisfy the assumptions of Lemma \ref{lemma_limit_gaussianity} with $\gamma=2$ and
 coefficients $c_L(k;N;j)$ corresponding to $p_{k}(N)\{j\}$ being of order $L^{d(k,N,j)}$ as $L\to\infty$, where
 $d(k,N,1)=1$ and $d(k,N,j)$ is a non-positive integer for $j>1$.
\end{lemma}
\begin{proof}
 We want to compute the ``moments'' of $p_{k}(N)\{j\}$. For that we
 use Theorem \ref{theorem_decomposition_of_P} and Definition \ref{definition_Integral_operator}
 to write $p_{k}(N)\{j\}$ as an integral operator, then apply the formula \eqref{eq_auxiliary_moments}
 and Proposition \ref{proposition_iteration_of_DI}. Finally, we  change the variables $w^j=L u^j$. Let us specialize all the data required for the application of
Lemma \ref{lemma_limit_gaussianity}.

\begin{enumerate}
\item The dimension $l$ corresponding to $p_{k}(N)\{j\}$ is the dimension of $j$th integral
operator in the decomposition of $\P^k_N$ (see Theorem \ref{theorem_decomposition_of_P}).

\item The contours of integration are (scaled by $L$) nested admissible contours arising in Proposition
\ref{proposition_iteration_of_DI}. We further assume that the distance parameter grows linearly in
$L$, thus, after rescaling by $L$, the admissibility condition does not depend on $L$.

\item The definition of admissible contours (see beginning of Section \ref{Section_integral})
readily implies this property.

\item The functions $\mathcal F^k_L$ are integrands in Definition
\ref{definition_Integral_operator} after the change of variables $w^j=L u^j$. If we extract the
prefactor $L^{dim-deg}$ (where $dim$ is the dimension of the integral and $deg$ is its
 degree), which will be absorbed by contants $C_L(k)$, then the functions $\mathcal F^k_L$ clearly
 converge to analytic limits.
\item The cross-terms in the formulas for joint ``moments'' were explicitly computed in Proposition
 \ref{proposition_iteration_of_DI} and one can easily see that after change of variables $w^j=Lu^j$
 and expansion in power series in $L^{-1}$, the first order terms cancel out and we get the
 required expansion with $\gamma=2$.
\item Each constant $C_L(k)$ is the product of $c$ from Definition \ref{def_Integral_operator} and
$L^{dim-deg}$ from property 4. Note that Theorem \ref{theorem_decomposition_of_P} yields that this
$dim-deg$ corresponding to $p_{k}(N)\{j\}$ is $1$ for $j=1$ and is
 less than $1$ for $j>1$.
\item $\gamma=2$.\qedhere
\end{enumerate}
\end{proof}

Now we apply Lemma \ref{lemma_limit_gaussianity} to the random variables $p_{k_t}(N_t)\{j\}$,
$t=1,\dots, h$, $j=1,\dots R(k_t)$, and conclude that their moments converge (after rescaling) to
those of Gaussian random variables. Since
\begin{equation}
 p_{k_t}(N_t;\r)-\E p_{k_t}(N_t;\r)=\sum_{j=1}^{R(k_t)} (p_{k_t}(N_t)\{j\}-\E p_{k_t}(N_t)\{j\}),
\end{equation}
and by Lemma \ref{lemma_vars_satisfy} in the last sum the $j$th term is of order
$L^{d(k_t,N_t,j)-1}$ (where $-1=-\gamma/2$), we conclude that as $L\to\infty$ all the terms except
for $j=1$ vanish. Therefore, the moments of random vector $(p_{k_t}(N_t)-\E p_{k_t}(N_t))_{t=1}^h$
converge to those of the Gaussian random vector with mean $0$ and the same covariance as the limit
(centered) covariance of variables $p_{k_t}(N_t)\{1\}$.

In the rest of the proof we compute this limit covariance. By the definition, $\P_{k}(N)\{1\}$ is
$-(\theta)^{-k}\DI^{[k]}_N$, and the operator $\DI^{[k]}_N$ acts on product functions via (we are
using Definition \ref{def_Integral_operator})
\begin{multline}
 \DI^{[k]}_N \prod_{i=1}^N g(y_i) = \left(\prod_{i=1}^N g(y_i)\right) \frac{\prod_{1\le i<j\le k}
\theta^2(j-i)^2}{\prod_{1\le i<j\le k} \theta(j-i+1) \prod_{1\le i<j-1\le k} \theta(j-i-1)} \\
\times \frac{1}{2\pi \i} \oint
 \prod_{m=1}^N \frac{v-y_m-\theta}{v+(k-1)\theta-y_m} \prod_{i=0}^{k-1} \frac{ g(v+\theta
 i-1)}{g(v+\theta i)} dv
\\=
 \left(\prod_{i=1}^N g(y_i)\right) \frac{\theta^{k-1}}{2 k\pi \i } \oint
 \prod_{m=1}^N \frac{v-y_m-\theta}{v+(k-1)\theta-y_m} \prod_{i=0}^{k-1} \frac{ g(v+\theta
 i-1)}{g(v+\theta i)} dv,
\label{eq_P_operator_refined}
\end{multline}
where the contour encloses $\{y_m-(k-1)\theta\}_{m=1}^N$. Therefore, as in Proposition
\ref{proposition_iteration_of_DI},

\begin{flalign}
 \notag \E \bigg( p_{k_1}(N_1)\{1\}  p_{k_2}(N_2)\{1\} \bigg) - \E (
 p_{k_1}(N_1)\{1\})\E(p_{k_2}(N_2)\{1\})\hskip 3.5cm\\ =
 \frac{\theta^{-2}}{k_1 k_2} \cdot \frac{1}{(2\pi \i)^2}
 \oint\oint dv_1 dv_2   \prod_{r=1}^2 \left(\prod_{m=1}^{N_r} \frac{v_r-y_m-\theta}{v_r+(k_r-1)\theta-y_m} \prod_{i=0}^{k_r-1} \frac{ g(v_r+\theta
 i-1)}{g(v_r+\theta i)}\right)
 \label{eq_x23}\\
 \times \left(\prod_{i=0}^{k_2-1} \frac{v_1-v_2-\theta i-\theta+1}{v_1-v_2+{(k_1-1)\theta-\theta
i+1}} \cdot \frac{v_1-v_2+{(k_1-1)\theta-\theta i}}{v_1-v_2-\theta i-\theta} -1 \right),
 \label{eq_x4}
\end{flalign}

where $y_m=\theta(1-m)$, contours are nested ($v_1$ is larger) and enclose the poles at
$y_m-(k_r-1)\theta$. Using $g(z)=H(z;\alpha,M)$ from \eqref{eq_definition_of_H} we have
$$
 \frac{g(z-1)}{g(z)}= \frac{z-{\theta\alpha}}{z{-\theta\alpha-M\theta}}.
$$
The part \eqref{eq_x23} of the integrand   simplifies  to
$$
 \prod_{r=1}^2 \left( \frac{(v_r-\theta)v_r(v_r+\theta)\cdots
(v_r+(k_r-2)\theta)}{(v_r+\theta(N_r-2+1))\cdots(v_r+\theta(N_r-2+k_r))} \prod_{i=0}^{k_r-1}
\frac{v_r+\theta i-{\theta\alpha}}{v_r+\theta i{-\theta\alpha-M\theta}} \right)
$$
Elementary computations reveal that the part \eqref{eq_x4} of the integrand as a power series in
$(v_1-v_2)^{-1}$ is
\begin{multline}
 -1 +\prod_{i=0}^{k_2-1}\left(1-\frac{-\theta i -\theta}{v_1-v_2}+\frac{\theta^2
 (i+1)^2}{(v_1-v_2)^2}+O\left((v_1-v_2)^{-3}\right)\right)\left(1+\frac{-\theta i -\theta +1}{v_1-v_2}\right)\\ \times
 \left(1-\frac{(k_1-1-i)\theta+1}{v_1-v_2}+\frac{((k_1-1-i)\theta+1)^2}{(v_1-v_2)^2}+O\left((v_1-v_2)^{-3}\right)\right)\\
 \times \left(1+\frac{(k_1-1-i)\theta}{v_1-v_2}\right)
 = \frac{\theta
 k_1 k_2 }{(v_1-v_2)^2} +O\left((v_1-v_2)^{-3}\right).
\end{multline}
Changing the variables $v_i=\theta L u_i$  transforms \eqref{eq_x23}, \eqref{eq_x4} into
\begin{equation}
  \frac{\theta^{-1}}{(2\pi \i)^2}
 \oint\oint \frac{du_1 du_2}{(u_1-u_2)^2}
 \prod_{r=1}^2 \left( \frac{u_r}{(u_r+\hat N_r)} \cdot \frac{u_r-{\hat\alpha}}{(u_r-\hat \alpha-\hat M)}
 \right)^{k_r} + O(L^{-1}),
\end{equation}
where contours are nested ($u_2$ is smaller) and enclose the singularities at $-\hat N_1$, $-\hat
N_2$ (but not at  $\hat \alpha+\hat M$) . Sending $L$ to infinity completes the proof.

\subsection{Preliminaries on the two--dimensional Gaussian Free Field}
\label{Section_GFF} In this section we briefly recall what is the 2d Gaussian Free Field. An
interested reader is referred to \cite{Sh}, \cite[Section 4]{Dubedat}, \cite[Section 2]{HMP} and
references therein for a more detailed discussion.

\begin{definition} The Gaussian Free Field with Dirichlet boundary conditions in the upper half--plane $\mathbb H$
is a (generalized) centered Gaussian random field $\mathcal F$  on $\mathbb H$ with covariance
given by
\begin{equation}
\label{eq_GFF_cov}
 \mathbb E (\mathcal F(z) \mathcal F(w)) = -\frac{1}{2\pi} \ln\left|\frac{z-w}{z-\overline
 w}\right|, \quad z,w\in\mathbb H.
\end{equation}
\end{definition}
We note that although $\mathcal F$ can be viewed as a random element of a certain functional
space,
 there is no such thing as a value of $\mathcal F$ at a given point $z$ (this is related to the
singularity of \eqref{eq_GFF_cov} at $z=w$).

Nevertheless, $\mathcal F$ inherits an important property of conventional functions: it can be
integrated with respect to (smooth enough) measures. Omitting the description of the required
smoothness of measures, we record this property in the following two special cases that we present
without proofs.

\begin{lemma}
\label{Lemma_GFF_by_smooth_measure}
 Let $\mu$ be an absolutely continuous (with respect to the Lebesgue measure) measure on $\mathbb H$ whose density is a smooth function $g(z)$ with
compact support. Then
$$\int_{\mathbb H} \mathcal F d\mu=\int_{\mathbb H} \mathcal F(u) g(u) du $$ is a well-defined centered Gaussian random variable.
Moreover, if we have two such measures $\mu_1$, $\mu_2$ (with densities $g_1$, $g_2$), then
\begin{multline*}
 \mathbb E \left[ \Big(\int_{\mathbb H}    \mathcal F(u) g_1(u) du\Big) \cdot \Big(\int_{\mathbb H}   \mathcal
   F(u) g_2(u)
 du\Big) \right]\\
  =\int_{\mathbb H^2} g_1(z) \left(-\frac{1}{2\pi} \ln\left|\frac{z-w}{z-\overline
 w}\right|\right) g_2(w) dz dw = \int_{\mathbb H} g_1(u) \Delta^{-1} g_2(u) du,
\end{multline*}
where $\Delta^{-1}$ is the inverse of the Laplace operator with Dirichlet boundary conditions in
$\mathbb H$.
\end{lemma}
\begin{lemma}
\label{Lemma_GFF_by_curve_measure}
 Let $\mu$ be a measure on $\mathbb H$ whose support is a smooth curve $\gamma$
 and whose density with respect to the natural (arc-length) measure on $\gamma$ is
 given by a smooth function $g(z)$ such that
\begin{equation}
\label{eq_x13} \int\int_{\gamma\times\gamma} g(z) \left(-\frac{1}{2\pi}
\ln\left|\frac{z-w}{z-\overline
 w}\right|\right) g(w) dz dw <\infty.
\end{equation}
Then
$$\int_{\mathbb H} \mathcal F d\mu=\int_{\gamma} \mathcal F(u) g(u) du $$ is a well-defined Gaussian centered random
variable of variance given by \eqref{eq_x13}. Moreover, if we have two such measures $\mu_1$,
$\mu_2$ (with two curves $\gamma_1$, $\gamma_2$ and two densities $g_1$, $g_2$), then
\begin{multline*}
 \mathbb E \left[ \Big(\int_{\gamma_1}    \mathcal F(u) g_1(u) du\Big) \cdot \Big(\int_{\gamma_2}   \mathcal
   F(u) g_2(u)
 du\Big) \right]
  =\int\int_{\gamma_1\times \gamma_2} g_1(z) \left(-\frac{1}{2\pi} \ln\left|\frac{z-w}{z-\overline
 w}\right|\right) g_2(w) dz dw.
\end{multline*}
\end{lemma}
 In principle, the above two lemmas can be taken as an alternative
definition of the Gaussian Free Field as a random functional on (smooth enough) measures.

Another property of functions that $\mathcal F$ inherits is the notion of pullback.
\begin{definition}
\label{Definition_GFF_pullback} Given a domain $D$ and a map $\Omega: D\to \mathbb
H$, the pullback $\mathcal F \circ \Omega$ is a generalized centered Gaussian Field
on $D$ with covariance
$$
 \mathbb E (\mathcal F(\Omega(z)) \mathcal F(\Omega(w))) = -\frac{1}{2\pi} \ln\left|\frac{\Omega(z)-\Omega(w)}{\Omega(z)-\overline
 \Omega(w)}\right|, \quad z,w\in D.
$$
Integrals of $\mathcal F \circ \Omega$ with respect to measures can be computed through
$$
 \int_{D} \left(\mathcal F \circ \Omega\right) d\mu = \int_{\mathbb H} \mathcal F d\Omega(\mu),
$$
where $d\Omega(\mu)$ stands for the pushforward of measure $\mu$.
\end{definition}
The above definition immediately implies the following analogue of Lemma
\ref{Lemma_GFF_by_curve_measure} (there is a similar analogue of Lemma
\ref{Lemma_GFF_by_smooth_measure}).
\begin{lemma}
\label{Lemma_GFF_pullback_by_curve_measure}
 In notation of Definition \ref{Definition_GFF_pullback}, let $\mu$ be a measure on $D$ whose support is a smooth curve $\gamma$ and whose density with respect to the natural (length) measure on $\gamma$ is
 given by a smooth function $g(z)$ such that
\begin{equation}
\label{eq_x11} \int\int_{\gamma\times\gamma} g(z) \left(-\frac{1}{2\pi}
\ln\left|\frac{\Omega(z)-\Omega(w)}{\Omega(z)-\overline
 \Omega(w)}\right|\right) g(w) dz dw <\infty.
\end{equation}
  Then
$$\int_{D} \left(\mathcal F\circ\Omega\right) d\mu=\int_{\gamma} \mathcal F(\Omega(u)) g(u) du $$ is a well-defined Gaussian centered random variable of variance given by \eqref{eq_x11}.
Moreover, if we have two such measures $\mu_1$, $\mu_2$ (with two curves $\gamma_1$, $\gamma_2$
and two functions $g_1$, $g_2$), then
\begin{multline*}
 \mathbb E \left[ \Big(\int_{\gamma_1}    \mathcal F(\Omega(u)) g_1(u) du\Big) \cdot \Big(\int_{\gamma_2}   \mathcal
   F(\Omega(u)) g_2(u)
 du\Big) \right]\\
  =\int\int_{\gamma_1\times \gamma_2} g_1(z) \left(-\frac{1}{2\pi} \ln\left|\frac{\Omega(z)-\Omega(w)}{\Omega(z)-\overline
 \Omega(w)}\right|\right) g_2(w) dz dw.
\end{multline*}
\end{lemma}

As a final remark of this section, we note that the Gaussian Free Field is a conformally invariant
object: if $\phi$ is a conformal automorphism of $\mathbb H$ (i.e.\ a real Moebius
transformation), then the distributions of $\mathcal F$ and $\mathcal F \circ \phi$ are the same.

\subsection{Identification of the limit object}
\label{section_identification}

The aim of this section is to interpret Theorem \ref{theorem_joint_CLT} as convergence of a random
height function to a pullback of the Gaussian Free Field.

\begin{lemma} \label{lemma_Maple}
We have
$$
 \frac{u}{u+a}\cdot \frac{u-b}{u-b-c} = q_1 + q_2 \frac{u-b-c}{u+a} + q_3 \frac{u+a}{u-b-c},
$$
where
$$
 q_1=\frac{2ac+ba+b^2+bc}{(a+b+c)^2},\quad q_2 =\frac{(b+a)a}{(a+b+c)^2},\quad q_3=\frac{c(b+c)}{(a+b+c)^2}.
$$
\end{lemma}
\begin{proof} Straightforward computations. \end{proof}

\begin{lemma}
\label{lemma_circles_of_real} For $a,b,c>0$ the functon
\begin{equation}
\label{eq_bijection}
 u\mapsto\frac{u}{u+a}\cdot \frac{u-b}{u-b-c}
\end{equation} is real on the circle with center $\frac{a(b+c)}{a-c}$ and radius
$\frac{\sqrt{ac(a+b)(c+b)}}{|a-c|}$. When $a=c$, this circle becomes a vertical line $\Re u=b/2$.
\end{lemma}
\begin{proof}
Lemma \ref{lemma_Maple} yields that expression \eqref{eq_bijection} is real when
$\frac{u-b-c}{u+a}$ belongs to the circle of radius
$$
 \sqrt{\frac{q_3}{q_2}}=\sqrt{\frac{c(b+c)}{a(b+a)}}.
$$
with center at $0$. Reinterpreting this as a condition on $u$ completes the proof.
\end{proof}

We use Lemma \ref{lemma_circles_of_real} to give the following definition.

\begin{definition} \label{def_Omega} Define the map $\Omega$ from the subset $D$ of $[0,1]\times
[0,+\infty]$ in $(x,\hat N)$ plane defined by the inequalities
$$
 C_1(\hat N,\hat \alpha, \hat M)-2\sqrt{C_2(\hat N,\hat \alpha, \hat M)}\le x \le C_1(\hat N,\hat \alpha, \hat M)+2\sqrt{C_2(\hat N,\hat \alpha, \hat M)},
$$
$$
C_1(\hat N,\hat \alpha, \hat M)=\frac{\hat  N \hat M + (\hat N+ \hat \alpha)(\hat M+\hat
\alpha)}{(\hat N+\hat \alpha+\hat M)^2},
$$
$$
C_2(\hat N,\hat \alpha, \hat M)=\frac{\hat M( \hat M+\hat \alpha) \hat N (\hat
N+\hat \alpha)}{(\hat N+\hat \alpha+\hat M)^4},
$$
to the upper half-plane $\mathbb H$ through the requirement that the horizontal section at height
$\hat N$ is mapped to the upper half-plane part of the circle with center
$$
 \frac{ \hat N (\hat \alpha +\hat M)}{\hat N - \hat M}
$$ and radius $$ \frac{\sqrt{\hat M(\hat M+\hat \alpha) \hat N (\hat N+\hat \alpha)}}{|\hat N-\hat M|} $$
(when $\hat N=\hat M$ the circle is replaced by the vertical line $\Re(u)=\hat
\alpha/2$)
 in
such a way that point $u\in\mathbb H$ is the image of point
$$
 \left(\frac{u}{u+\hat N}\cdot \frac{u- \hat \alpha }{u- \hat\alpha-\hat M};\hat
 N\right)\in D.
$$
\end{definition}

Note that for small $\hat N$, the radius of the circle is very small, and its center
approaches $0$. As $\hat N$ grows, so is the radius, while the center escapes to
$-\infty$.  The radius becomes infinite when $\hat N=\hat M$ and then the circle is
replaced by the vertical line $\Re(u)=\hat \alpha/2$. If we further grow $\hat N$,
then the circle reappears, with position of its center now starting at $+\infty$ and
decreasing as $\hat N$ grows; simultaneously the radius decreases starting at
$\infty$. As $\hat N\to\infty$, the radius becomes $\sqrt{\hat M (\hat M+\hat
\alpha)}$ and the center approaches $\hat \alpha +\hat M$. One then shows that the
image of $\Omega$ is $\mathbb H$ without (the upper half of) the ball with radius
$\sqrt{\hat M (\hat M+\hat \alpha)}$ and centered at $\hat \alpha +\hat M$. It is
also straightforward to show that $\Omega$ is injective.

The boundary of the set $D$ for some values of parameters is shown in Figure
\ref{Fig_frozen}.

\begin{figure}[h]
\begin{center}
{\scalebox{0.24}{\includegraphics{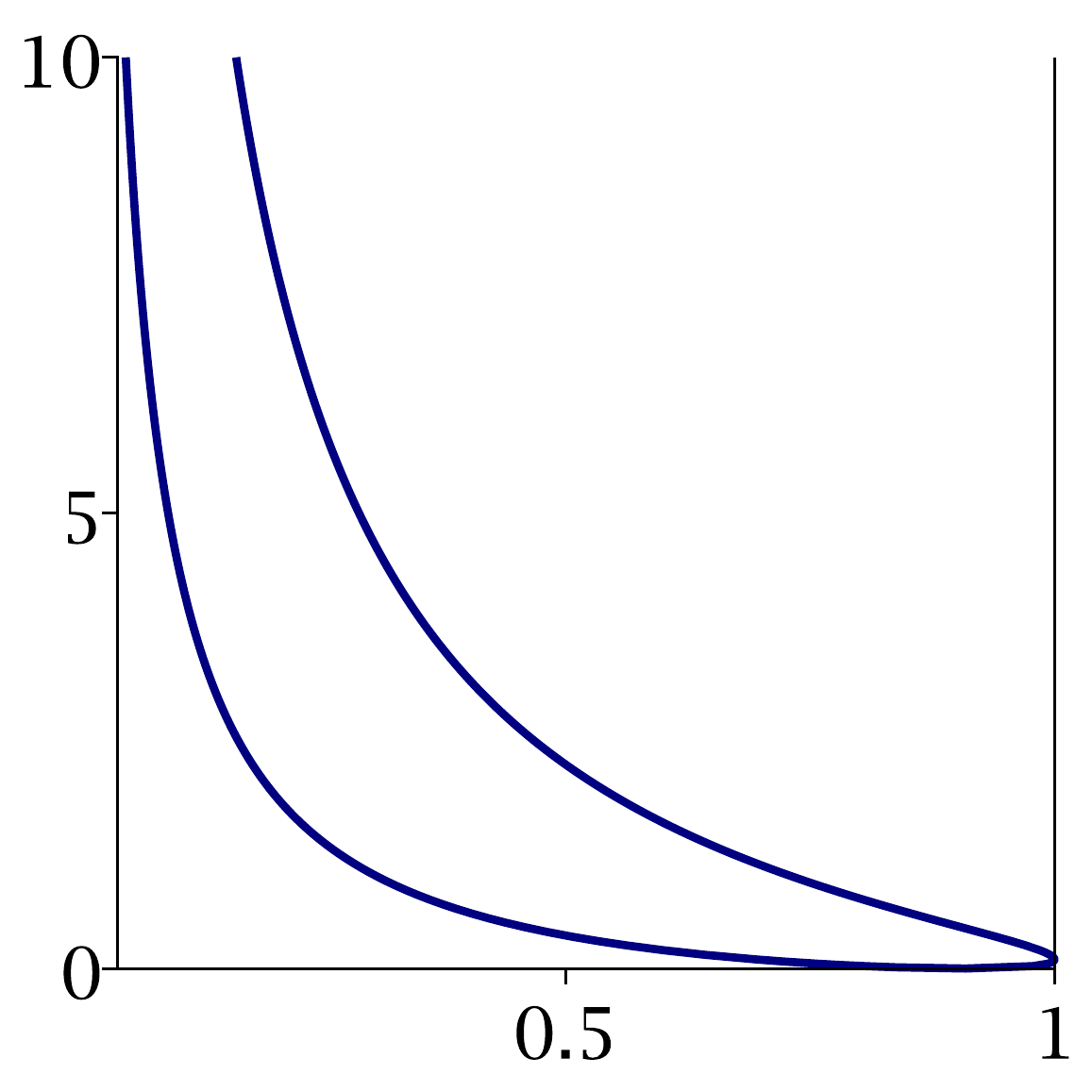}}}
{\scalebox{0.24}{\includegraphics{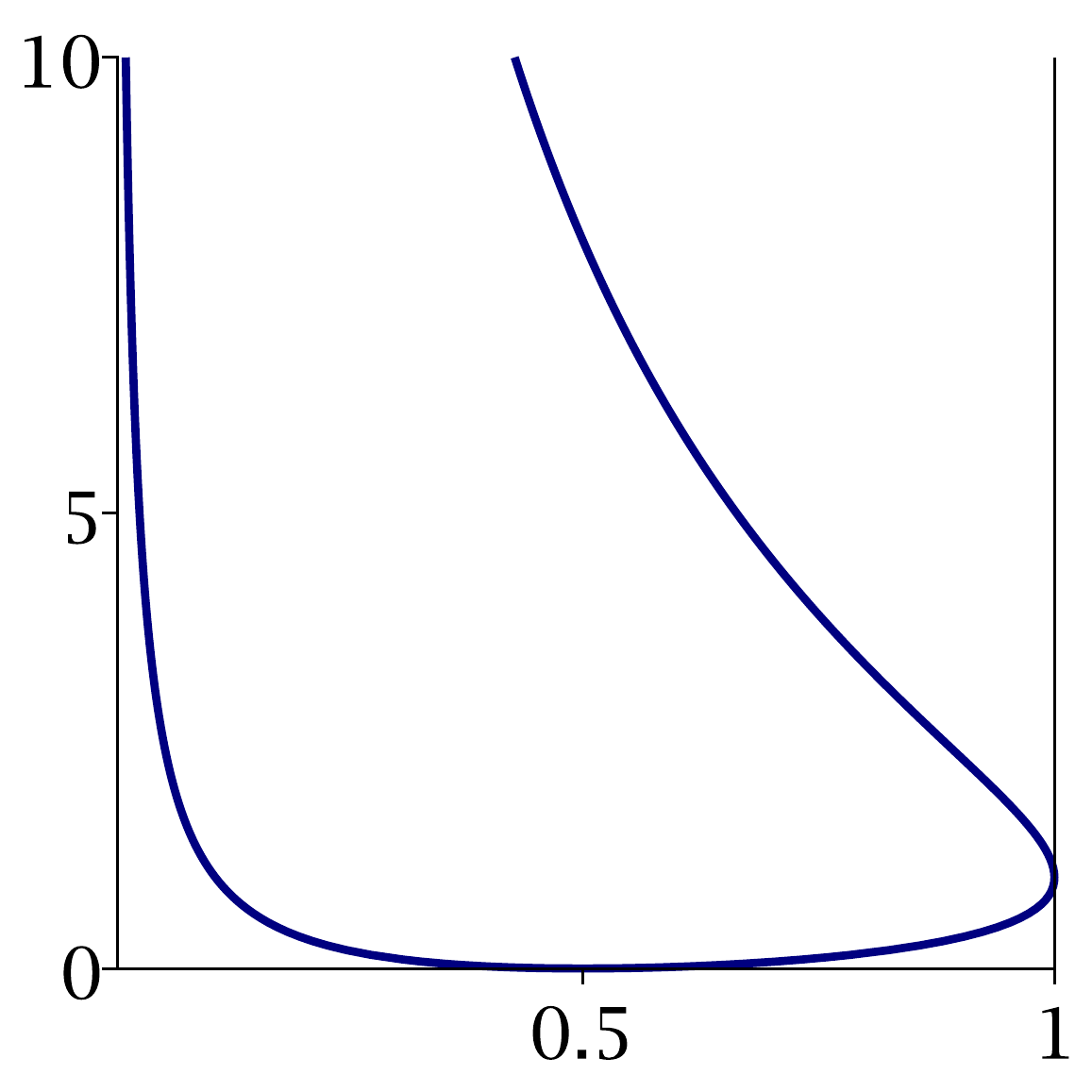}}}
{\scalebox{0.24}{\includegraphics{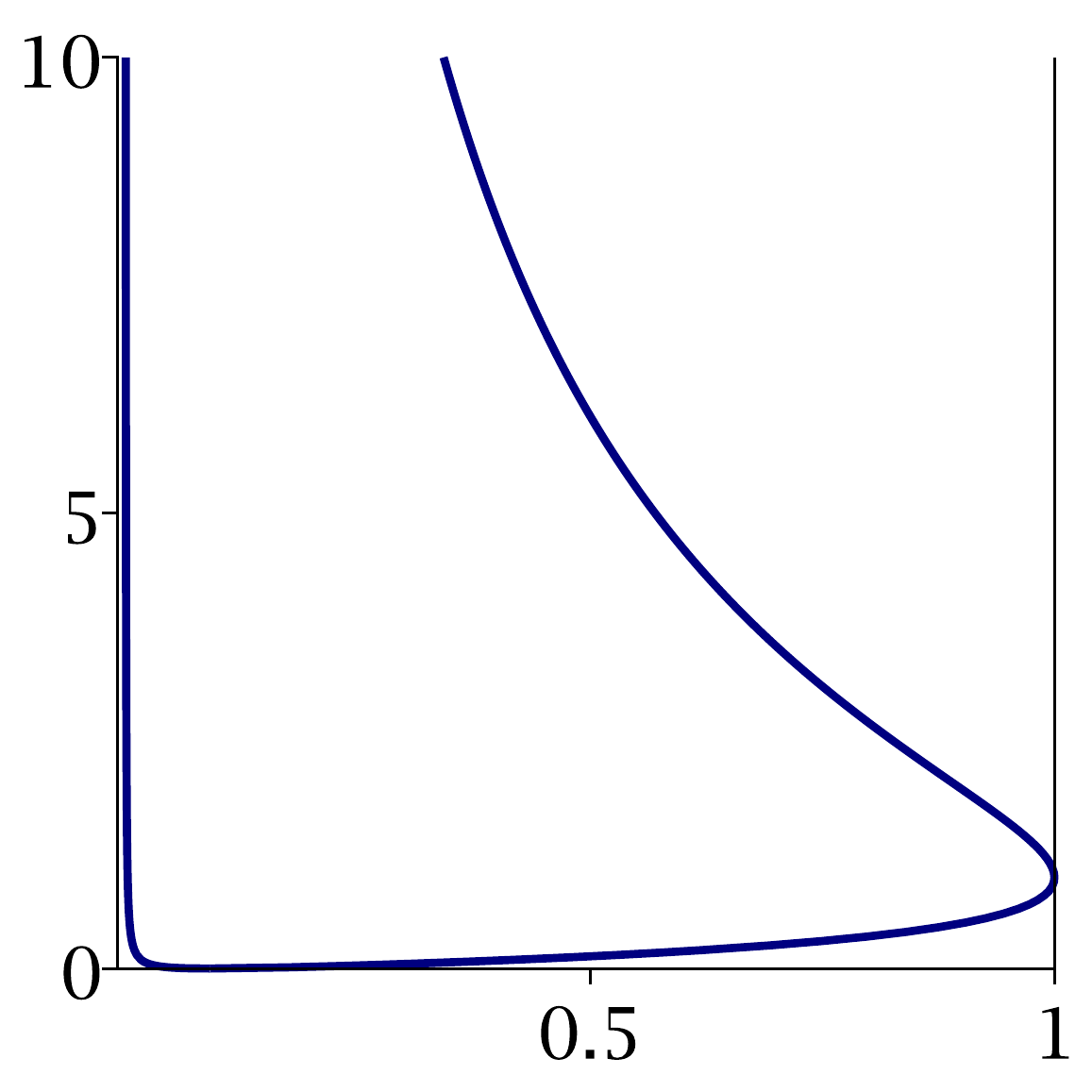}}} \caption{The boundary of domain $D$ (i.e.\
\emph{frozen} boundary) for parameters (from left to right) $\hat M=0.1$, $\hat \alpha= 1$; $\hat
M=1$, $\hat \alpha= 1$; $\hat M=1$, $\hat \alpha= 0.1$ \label{Fig_frozen}  }
\end{center}
\end{figure}

\begin{definition} Suppose that $\r\in\St^M$ is distributed according to $\Pr^{\alpha,M,\theta}$.
 For a pair of numbers $(x,N)\in[0,1]\times{\mathbb R}_{>0}$ define the height function ${\mathcal H}(x,N)$ as the (random) number of points $\r^{\lfloor N\rfloor}_i$ which are
 less than $x$.
\end{definition}

\begin{theorem}
\label{theorem_GFF}
 Suppose that as our large parameter $L\to\infty$, parameters $M$, $\alpha$ grow linearly in $L$:
 $$
  M \sim L \hat M,\quad \alpha \sim L \hat\alpha;\quad \hat M>0, \quad \hat \alpha>0.
 $$
 Then the centered random (with respect to measure $\Pr^{\alpha,M,\theta}$) height function
 $$
  \sqrt{\theta\pi}({\mathcal H}(x,L\hat N)-\E {\mathcal H}(x,L\hat N))
 $$
 converges to the pullback of the Gaussian Free Field on $\mathbb H$ with respect to map $\Omega$ of Definition \ref{def_Omega}
 in
 the following sense:
 For any set of polynomials $R_1,\dots,R_k\in \mathbb C[x]$ and positive numbers $\hat N_1,\dots,\hat N_k$,
 the joint distribution of
 $$
  \int_0^1 ({\mathcal H}(x,L\hat N_i)-\E {\mathcal H}(x,L\hat N_i)) R_i(x) dx,\quad i=1,\dots, k,
 $$
 converges to the joint distribution of the similar averages of the pullback of GFF given by
 $$
  \int_0^1 {\mathcal F}(\Omega(x,\hat N_i)) R_i(x) dx,\quad i=1,\dots, k.
 $$
\end{theorem}
{\bf Remark.} One might wonder why the Gaussian Free Field should appear in
such a context. We were originally guided by two observations: First, the
limiting covariance after a simple renormalization turns out to independent of
$\theta$. Second, at $\theta=1$ the random matrix corners processes are known
to yield GFF fluctuations, either through direct moment computations
\cite{B-CLT}, or through a suitable degenerations of $2d$ models of random
surfaces for which the appearance of the Gaussian Free Field is widely
anticipated, cf.\ the discussion after Theorem \ref{Theorem_main_intro} above.
We were thus led to believe that it should be possible to identify our
covariance with that of a pullback of the Gaussian Free Field. The
exact form of the desired map $\Omega$
was prompted by the integral formula \eqref{eq_limit_covariance}. To our
best knowledge, more conceptual reasons for the appearance of the GFF in
general beta ensembles are yet to be discovered.

\begin{proof}[Proof of Theorem \ref{theorem_GFF}]
 We assume that $\hat N_1\ge \hat N_2\ge \dots \ge \hat N_k>0$ and set $N_i=\lfloor L \hat N_i\rfloor$. Clearly, it
 suffices to consider monomials $R_i(x)=x^{m_i}$. For those we have
$$
  \int_0^1 ({\mathcal H}(x,L\hat N_i)-\E {\mathcal H}(x,L\hat N_i)) x^{m_i} dx=  \int_0^1 {\mathcal H}(x,L\hat N_i) x^{m_i} dx - \E \int_0^1 {\mathcal H}(x,L\hat N_i)
  x^{m_i}
  dx.
$$
Integrating by parts, we get
\begin{multline*}
 \int_0^1 {\mathcal H}(x,L\hat N_i) x^{m_i} dx= {\mathcal H}(1,L\hat N_i) \frac{1^{m_i+1}}{m_i+1}-\int_0^1 {\mathcal H}'(x, L\hat N_i)
 \frac{x^{m_i+1}}{m_i+1} dx\\= \frac{\min(N_i,M)}{m_i+1} -\sum_{j=1}^{ \min (N_i,M)} \dfrac{(\r^{N_i}_j)^{m_i+1}}{m_i+1}.
\end{multline*}
Therefore,
$$
\int_0^1 ({\mathcal H}(x,L\hat N_i)-\E {\mathcal H}(x,L\hat N_i)) x^{m_i} dx = -\frac{1}{m_i+1}
\bigg( p_{m_i+1}(N_i;\r) -\E p_{m_i+1}(N_i;\r)\bigg).
$$
Applying Theorem \ref{theorem_joint_CLT}, we conclude, that random variables $$
\sqrt{\theta\pi}\int_0^1 ({\mathcal H}(x,L\hat N_i)-\E {\mathcal H}(x,L\hat N_i)) x^{m_i} dx$$ are
asymptotically Gaussian with the covariance of $i$th and $j$th ($i<j$) given by
\begin{multline}
\label{eq_covcov}
 \frac{\pi}{(2\pi \i)^2 (m_i+1)(m_j+1)}
 \oint\oint \frac{du_1 du_2}{(u_1-u_2)^2} \\ \times
  \left( \frac{u_1}{(u_1+\hat N_i)} \cdot \frac{u_1-{\hat\alpha}}{(u_1-\hat \alpha-\hat
 M)}
 \right)^{m_i+1} \left( \frac{u_2}{(u_2+\hat N_j)} \cdot \frac{u_2-{\hat\alpha}}{(u_2-\hat \alpha-\hat
 M)}
 \right)^{m_j+1}
\end{multline}
where contours are nested ($u_1$ is larger) and enclose the singularities at $-\hat N_i$, $-\hat
N_j$. We claim that we can deform the contours, so that $u_1$ is integrated over the circle with
center $
 \frac{ \hat N_i (\hat \alpha +\hat M)}{\hat N_i - \hat M}
$ and radius $ \frac{\sqrt{\hat M(\hat M+\alpha) \hat N_i (\hat N_i+\alpha)}}{|\hat N_i-\hat M|}
$, while $u_2$ is integrated over the circle with center $
 \frac{ \hat N_j (\hat \alpha +\hat M)}{\hat N_j - \hat M}
$ and radius $ \frac{\sqrt{\hat M(\hat M+\alpha) \hat N_j (\hat N_j+\alpha)}}{|\hat N_j-\hat M|}$.
Indeed, if $\hat N_i,\hat N_j<M$, then the first contour is larger and both contours are to the
left from the vertical line $\Re u=\hat\alpha/2$ and, thus, do not enclose the singularity at
$\hat\alpha+\hat M$ that could have potentially been an obstacle for the deformation. Cases $\hat
N_i > M >\hat N_j$ and $\hat N_i,\hat N_j>M$ are similar and when $\hat N_i=\hat M$ or $\hat
N_j=\hat M$, then the circles are replaced by vertical lines $\Re(u)=\alpha/2$ as in Definition
\ref{def_Omega}.

The top halves of the above two circles can be parameterized via images of the horizontal segments
with respect to the map $\Omega$; to parameterize the whole circles we also need the conjugates
$\bar\Omega$. Hence, \eqref{eq_covcov} transforms into
\begin{multline}
\label{eq_covcov2}
 \frac{\pi}{(2\pi \i)^2 (m_i+1)(m_j+1)}
 \int\int \frac{x_1^{m_i+1} x_2^{m_j+1} dx_1 dx_2
 }{(\Omega(x_1,\hat N_i)-\Omega(x_2,\hat N_j))^2} \frac{\partial
 \Omega}{\partial x_1} (x_1,\hat N_j) \frac{\partial
 \Omega}{\partial x_2} (x_2,\hat N_j)
 \\
 -\frac{\pi}{(2\pi \i)^2 (m_i+1)(m_j+1)}
 \int\int \frac{x_1^{m_i+1} x_2
 ^{m_j+1} dx_1 dx_2 }{(\overline \Omega(x_1,\hat N_i)-\Omega(x_2,\hat N_j))^2} \frac{\partial
 \overline \Omega}{\partial x_1} (x_1,\hat N_j) \frac{\partial
  \Omega}{\partial x_2} (x_2,\hat N_j)
 \\
 -\frac{\pi}{(2\pi \i)^2 (m_i+1)(m_j+1)}
 \int\int \frac{x_1^{m_i+1} x_2
 ^{m_j+1}dx_1 dx_2 }{(\Omega(x_1,\hat N_i)-\overline\Omega(x_2,\hat N_j))^2} \frac{\partial
 \Omega}{\partial x_1} (x_1,\hat N_j) \frac{\partial
  \overline\Omega}{\partial x_2} (x_2,\hat N_j)
 \\
 +\frac{\pi}{(2\pi \i)^2 (m_i+1)(m_j+1)}
 \int\int \frac{x_1^{m_i+1} x_2
 ^{m_j+1}dx_1 dx_2 }{(\overline\Omega(x_1,\hat N_i)-\overline\Omega(x_2,\hat N_j))^2} \frac{\partial
 \overline\Omega}{\partial x_1} (x_1,\hat N_j) \frac{\partial
 \overline\Omega}{\partial x_2} (x_2,\hat N_j)
\end{multline}
with $x_1$, $x_2$ integrated over the horizontal slices of domain $D$ at heights $\hat N_i$ and
$\hat N_j$, respectively. Integrating twice by parts in \eqref{eq_covcov2} and noting that
boundary terms cancel out (since $\Omega$ is real at the ends of the integration interval) we
arrive at
\begin{multline}
\label{eq_cov_final}
 -\frac{1}{4\pi}
 \int\int x_1^{m_i} x_2^{m_j} dx_1 dx_2 \bigg(\ln(\Omega(x_1,\hat N_i)-\Omega(x_2,\hat N_j))\\-\ln(\overline \Omega(x_1,\hat N_i)-\Omega(x_2,\hat N_j))
 -\ln(\Omega(x_1,\hat N_i)-\overline \Omega(x_2,\hat N_j))+\ln(\overline\Omega(x_1,\hat N_i)-\overline\Omega(x_2,\hat
 N_j))\bigg).
\end{multline}
Since we know that the expression \eqref{eq_cov_final} is real, the choice of branches of $\ln$ is
irrelevant here. Real parts in \eqref{eq_cov_final} give
\begin{equation}
\label{eq_vwery_final}
 -\frac{1}{2\pi}
 \int\int \ln\left|\frac{\Omega(x_1,\hat N_i)-\Omega(x_2,\hat N_j)}{\Omega(x_1,\hat N_i)-\overline \Omega(x_2,\hat N_j)}\right|x_1^{m_i} x_2^{m_j} dx_1
 dx_2,
 \end{equation}
which (by definition) is the desired covariance of the averages over the pullback of the Gaussian
Free Field with respect to $\Omega$.
\end{proof}

An alternative way to write down the limit covariance is through the classical Chebyshev
polynomials:

\begin{definition}
Define Chebyshev polynomials $T_n$, $n\ge 0$, of the first kind through
$$
 T_n\left(\frac{z+z^{-1}}{2}\right)=\frac{z^n+z^{-n}}2.
$$
Equivalently,
$$
 T_n(x)=\cos(n\arccos(x)).
$$
\end{definition}

\begin{proposition}
\label{prop_Chebyshev_ortho}  Let $\r\in\St^M$ be distributed according to $\Pr^{\alpha,M,\theta}$
and suppose that as the large parameter $L\to\infty$, parameters $M$, $\alpha$ grow linearly in
$L$:
 $$
  M \sim L \hat M,\quad \alpha \sim L \hat\alpha;\quad \quad \quad \hat M>0,\,\, \hat \alpha \ge 0.
 $$
If we set (with constant $C_1, C_2$ given in Definition \ref{def_Omega})
$$
\hat T^{\hat N,\hat \alpha,\hat M}_n(x) = T_n\left(\frac{x-C_1(\hat N,\hat \alpha, \hat
M)}{2\sqrt{C_2(\hat N,\hat \alpha, \hat M)}}\right),
$$
then the limiting Gaussian random variables
$$
 {\mathcal T}_{\hat \alpha,\hat M}(n,\hat N)=
 \lim_{L\to\infty}\left( \sum_{i=1}^{\min(\lfloor L \hat
 N\rfloor,M)}\hat T^{\hat N,\hat \alpha,\hat M}_n\left(\r^{\lfloor L \hat
 N\rfloor}_i\right) -\E  \sum_{i=1}^{\min(\lfloor L \hat
 N\rfloor,M)} \hat T^{\hat N,\hat \alpha,\hat M}_n\left(\r^{\lfloor L \hat
 N\rfloor}_i\right)\right)
$$
have the following covariance ($\hat N_1\ge \hat N_2$)
\begin{multline}
 \E\left( {\mathcal T}_{\hat \alpha,\hat M}(n_1,\hat N_1) {\mathcal T}_{\hat \alpha,\hat M}(n_2,\hat N_2)\right)=
  \frac{n_1!}{4\theta (n_2-1)!(n_1-n_2)!}
\\ \times \left(\frac{\hat N_2-\hat N_1}{\hat N_2+\hat \alpha+\hat M} \sqrt{\frac{\hat M(\hat\alpha+\hat M)}{\hat
N_1(\hat \alpha+\hat N_1)}}\right)^{n_1-n_2}
 \left(\frac{(\hat N_1+\hat \alpha+\hat M)\sqrt{\hat N_2(\alpha+\hat N_2)}}{(\hat N_2+\hat \alpha+\hat
M)\sqrt{\hat N_1(\alpha+\hat N_1)}}\right)^{n_2},
\end{multline}
in $n_1\ge n_2$, and the covariance is $0$ when $n_1<n_2$.
\end{proposition}
{\bf Remark 1. } When $\hat N_1=\hat N_2$, the formula above simplifies to $\frac{n_1}{4\theta}
\delta_{n_1,n_2}$, which agrees with  \cite[Theorem 1.2]{DP}.

{\bf Remark 2.} When both $\hat N_1$ and $\hat N_2$ are infinitesimally small, the formula matches
the one from  \cite[Proposition 3]{B-CLT} proved there for $\theta=1/2,1$ (the cases of GOE and
GUE).

\begin{proof}[Proof of Proposition \ref{prop_Chebyshev_ortho}] Theorem \ref{theorem_joint_CLT} yields
\begin{multline}
\label{eq_x9}
 \E\left( {\mathcal T}_{\hat \alpha,\hat M}(n_1,\hat N_1) {\mathcal T}_{\hat \alpha,\hat M}(n_2,\hat N_2)\right)\\=
 \frac{\theta^{-1}}{(2\pi \i)^2}
 \oint\oint \frac{du_1 du_2}{(u_1-u_2)^2}
 \prod_{r=1}^2 T^{\hat N_r,\hat \alpha,\hat M}_{n_r}\left( \frac{u_r}{(u_r+\hat N_r)} \cdot \frac{u_r-{\hat\alpha}}{(u_r-\hat \alpha-\hat
 M)}
 \right).
\end{multline}
Using Lemma \ref{lemma_Maple} and changing the variables
$$
 v_r= \sqrt{\frac{\hat M(\hat\alpha+\hat M)}{\hat N_r(\hat \alpha+\hat N_r)}} \cdot \frac{u_r+\hat N_r}{u_r-\hat \alpha-\hat
 M},
$$
we have
$$
 \frac{u_r}{(u_r+\hat N_r)} \cdot \frac{u_r-{\hat\alpha}}{(u_r-\hat \alpha-\hat
 M)} = C_1(\hat N_r,\hat \alpha, \hat M)+ 2\sqrt{C_2(\hat N_r,\hat \alpha, \hat M)} \frac{v_r+v_r^{-1}}{2}.
$$
Thus,
$$
 T^{\hat N_r,\hat \alpha,\hat M}_{n_r}\left( \frac{u_r}{(u_r+\hat N_r)} \cdot \frac{u_r-{\hat\alpha}}{(u_r-\hat \alpha-\hat
 M)}\right)=\frac{v_r^{n_r}+v_r^{-n_r}}{2}.
$$
Also with the notation $A_r=\sqrt{\frac{\hat M(\hat\alpha+\hat M)}{\hat N_r(\hat \alpha+\hat
N_r)}} $, $B_r=\hat N_r$, $C=\hat \alpha+\hat M$, we have
$$
v_r=A_r\frac{u_r+B_r}{u_r-C}=A_r\left(1+\frac{B_r+C}{u_r-C}\right),\quad
\frac{A_r(B_r+C)}{v_r-A_r}=u_r-C,\quad u_r=C+\frac{A_r(B_r+C)}{v_r-A_r},
$$
$$
d u_r= -\frac{A_r(B_r+C)}{(v_r-A_r)^2},
$$
$$
 u_1-u_2=\frac{ v_2 A_1 (B_1+C)-v_1 A_2(B_2+C)+A_1A_2(B_2-B_1)} {(v_1-A_1)(v_2-A_2)}.
$$
In particular, if $\hat N_1=\hat N_2$, then $A_1=A_2$, $B_1=B_2$ and
$$
 \frac{du_1 du_2}{(u_1-u_2)^2}=\frac{dv_1 dv_2}{(v_1-v_2)^2}.
$$
Therefore, for $\hat N_1=\hat N_2$ \eqref{eq_x9} transforms into
$$
 \frac{\theta^{-1}}{(2\pi \i)^2}
 \oint\oint \frac{dv_1 dv_2}{(v_1-v_2)^2} \frac{v_1^{n_1}+v_1^{-n_1}}{2} \cdot \frac{v_2^{n_2}+v_2^{-n_2}}{2}
$$
with nested contours surrounding zero ( $v_2$--contour is smaller ).  We get
\begin{multline*}
 \frac{\theta^{-1}}{(2\pi \i)^2} \oint\oint dv_1
 dv_2\frac{1}{v_1^2}\left(1+2\frac{v_2}{v_1}+3\left(\frac{v_2}{v_1}\right)^2+\dots\right) \frac{v_1^{n_1}+v_1^{-n_1}}{2} \cdot \frac{v_2^{n_2}+v_2^{-n_2}}{2}
 \\=\frac{\theta^{-1}}{2\pi \i} \oint \frac{dv_1}{v_1^2} \frac{v_1^{n_2}+v_1^{-n_2}}{2} \cdot
 \frac{n_1}{2} v_1^{1-n_1}=\frac{\theta^{-1} n_1}{4} \delta_{n_1,n_2}.
\end{multline*}
If $\hat N_1>\hat N_2$, then \eqref{eq_x9} transforms into
$$
 \frac{\theta^{-1}}{(2\pi \i)^2}
 \oint\oint \frac{A_1 A_2 (B_1+C) (B_2+C)dv_1 dv_2}{(v_2 A_1(B_1+C) - v_1 A_2 (B_2+C)+A_1A_2(B_2-B_1))^2} \frac{v_1^{n_1}+v_1^{-n_1}}{2} \cdot
 \frac{v_2^{n_2}+v_2^{-n_2}}{2},
$$
where $v_2$ is integrated over a small circle containing the origin and $v_1$ --- over a large
circle. Since
\begin{multline*}
\frac{1}{(v_2 A_1(B_1+C) - v_1 A_2 (B_2+C)+A_1A_2(B_2-B_1))^2}\\=\frac{1+2\frac{v_2 A_1(B_1+C)}{
v_1 A_2 (B_2+C)-A_1A_2(B_2-B_1)}+3 \left(\frac{v_2 A_1(B_1+C)}{ v_1 A_2
(B_2+C)-A_1A_2(B_2-B_1)}\right)^2+\dots}{( v_1 A_2 (B_2+C)-A_1A_2(B_2-B_1))^2},
\end{multline*}
the integral over $v_2$ gives
\begin{equation}
\label{eq_x10}
  \frac{\theta^{-1}n_2}{8\pi \i} \oint \frac{v_1^{n_1}+v_1^{-n_1}}{(v_1 A_2
  (B_2+C)-A_1A_2(B_1-B_2))^{n_2+1}} A_2  (B_2+C)  \left(A_1(B_1+C)\right)^{n_2} dv_1
\end{equation}
with $v_1$ to be integrated over a large circle. If $n_1<n_2$, then this is $0$ because of the
decay of the integrand at infinity; otherwise we can use
\begin{multline*}
 \frac{1}{(v_1 A_2
  (B_2+C)-A_1A_2(B_2-B_1))^{n_2+1}}=
\frac{(1-A_1\frac{B_2-B_1}{B_2+C} \frac{1}{v_1})^{-n_2-1}}{ v_1^{n_2+1} A_2^{n_2+1}
(B_2+C)^{n_2+1}}
\\=
\frac{1 + (n_2+1)A_1\frac{B_2-B_1}{B_2+C}\frac{1}{v_1}+ \frac{(n_2+1)(n_2+2)}{2!}
\left(A_1\frac{B_2-B_1}{B_2+C}\frac{1}{v_1}\right)^2+\dots  }{ v_1^{n_2+1} A_2^{n_2+1}
(B_2+C)^{n_2+1}}
\end{multline*}
and \eqref{eq_x10} gives
$$
 \frac{\theta^{-1}}{4} \frac{n_2 (n_2+1)\cdots n_1}{(n_1-n_2)!} \left(A_1\frac{B_2-B_1}{B_2+C}\right)^{n_1-n_2}
 \left(\frac{A_1(B_1+C)}{A_2(B_2+C)}\right)^{n_2}.\qedhere
$$
\end{proof}




\section{CLT as $\beta\to\infty$}

\label{section_CLT_beta}

Throughout this section parameters $\alpha$ and $M$ of measure $\Pr^{\alpha;M;\theta}$ are fixed,
while $\theta$ changes. We aim to study what happens when $\theta\to\infty$.

Let $\Jac_N$ be Jacobi orthogonal polynomial of degree
 $\min(N,M)$ corresponding to the weight function $x^{\alpha-1} (1-x)^{|M-N|}$, see e.g.\
 \cite{Szego}, \cite{KLS}  for the general information on Jacobi
 polynomials. $\Jac_N$ has $\min(N,M)$ real roots on the interval $(0,1)$ that we enumerate in the
 increasing order. Let $\jr^N_i$ denote the $i$th root of $\Jac_N$.

\begin{theorem} \label{Theorem_CLT_in_beta}
 Let $\r\in\St^M$ be distributed according to $\Pr^{\alpha;M;\theta}$ of Definition \ref{Def_distrib}.
 As $\theta\to\infty$, $\r^N_i$ converges (in probability) to $\jr^N_i$. Moreover,
  the random vector
 $$
 \sqrt{\theta}(\r^N_i-\jr^N_i),\quad i=1,\dots,\min(N,M),\quad N=1,2,\dots,
$$
converges (in finite--dimensional distributions) to a centered Gaussian random vector $\xi^N_i$,
$i=1,\dots,\min(N,M)$, $N=1,2,\dots$.
\end{theorem}
We do not have any simple formulas for the covariance of $\{\xi^N_i\}$. \emph{Some} formulas, in
principle, could be obtained from our argument below, see also \cite{DE-CLT} where a different
form of the covariance (for single--level distribution for the Hermite and Laguerre ensembles
which are degenerations of the Jacobi ensemble) is given.

\medskip

In the rest of this section we give a sketch of the proof of Theorem \ref{Theorem_CLT_in_beta}.

\smallskip

We start by proving that vector $\sqrt{\theta}\big(e_k(N;r)-\E e_k(N;r)\big)$,
$k=1,\dots,\min(N,M)$, $N=1,2,\dots$, is asymptotically Gaussian. The proof is another application
of Lemma \ref{lemma_limit_gaussianity} and is similar to that of Theorem \ref{theorem_joint_CLT}.
Our starting point is Theorem \ref{Theorem_HO_expectations} which expresses joint moments of
random variables $e_k(N,\r)$ in terms of applications of operators $\D^k_N$. Proposition
\ref{proposition_e_operator} gives an expansion of $\D^k_N$ in terms of integral operators. We
further define formal ``random variables'' $e_k(N)\{s\}$, $s\in\S_k$, through their joint
``moments'' given by (here $N_1\ge N_2\ge\dots\ge N_m\ge 1$)
\begin{equation}
\label{eq_auxiliary_moments_2}
 \E \left(\prod_{i=1}^m e_{k_i}(N_i)\{s_i\} \right) =  \dfrac{ \prod\limits_{i=1}^m  \theta^{-|s_i|}\DI^{s_i}_{N_i}
 \left[\prod\limits_{i=1}^{N_1} H(y_i;\alpha,M)\right]}{\prod_{i=1}^{N_1} H(y_i;\alpha,M)} \rule[-5mm]{0.9pt}{17mm}_{\,
y_i=\theta(1-i)}.
\end{equation}
Clearly, the joint distribution of $e_k(N,\r)$ (understood in the sense of moments) coincides with
that of the sums
\begin{equation}
\label{eq_x14}
 e_k(N,\r)= \frac{(-1)^{k}}{k!} \sum_{s=(S_1,\dots)\in \S_k} \left[(-1)^{k-\ell(s)} \prod_j (|S_j|-1)!\right]
e_k(N)\{s\}.
\end{equation}

\begin{lemma}
\label{lemma_vars_satisfy_2}
 ``Random variables''
 $e_{k}(N)\{s\}$ satisfy the assumptions of Lemma \ref{lemma_limit_gaussianity}
 with $\theta=\theta(L)=L$ (parameters $M$, $N$ and $\alpha$ do not depend on $L$ here),  with $\gamma=1$ and
 coefficients $c_L(k;N;j)$ corresponding to $e_{k}(N)$ being of order $1$ as $L\to\infty$.
\end{lemma}
\begin{proof}
 This follows from Proposition \ref{proposition_iteration_of_DI}. After change of variables $w_j=L
u_j$ the integrand in Definition \ref{definition_Integral_operator} converges to an analytic
limit, while the prefactor $L^{|s_i|}$ arising as we change variables, cancels with
$\theta^{-|s_i|}$ in the definition of variables $e_k(N)\{s\}$. As for the cross-terms, after our
change of variables they behave like $1+O(L^{-1})$.
\end{proof}

Now Lemma \ref{lemma_limit_gaussianity} implies that ``random vector''
$\sqrt{\theta}(e_{k}(N)\{s\}-\E e_{k}(N)\{s\})$ converges (in the sense of ``moments'') to the
centered Gaussian vector $\widehat e_{k}(N)\{s\}$ (in the sense that the limit moments satisfy the
Wick formula). Therefore, $\sqrt{\theta}(e_k(N;\r)-\E e_k(N;\r))$ converges (in the sense of
``moments'') to the limit Gaussian vector $\widehat e_k(N)$ such that
\begin{equation}
\label{eq_limit_decomp_es}
 \widehat e_k(N)=\frac{(-1)^{k}}{k!} \sum_{s=(S_1,\dots)\in \S_k} \left[(-1)^{k-\ell(s)} \prod_j (|S_j|-1)!\right]
\widehat e_k(N)\{s\}.
\end{equation}
One can also compute the covariance of $\widehat e_k(N)\{s\}$ (thus also of $\widehat e_k(N)$,
$k=1,\dots,N$) similarly to the covariance computation in the proof of Theorem
\ref{theorem_joint_CLT}.

Let us now explain that the Cental Limit Theorem for $e_k(N;\r)$ implies the Central Limit Theorem
for $\r^i_j$.

\begin{lemma} For any $N\ge 1$ and $k\le\min(N,M)$ the expectation $\E(e_k(N;\r))$ does not depend on
$\theta$.
\end{lemma}
\begin{proof}
 Observe that $\E(e_k(N;\r))$ can be computed
 using formulas \eqref{eq_x14} and \eqref{eq_auxiliary_moments_2}. When we change the variables
 $w_j=Lu_j$ in integral representations for the operators $\DI^s_N$ we find that the resulting
 expression does not depend on $\theta$.
\end{proof}

\begin{corollary}
\label{cor_LLN_beta}
 For any $N\ge 1$ and $1\le i\le \min(N,M)$, $\r^N_i$ converges (in probability) as $\theta\to\infty$ to a deterministic
 limit.
\end{corollary}
\begin{proof}
 Since moments of $\sqrt{\theta}(e_k(N;\r)-\E(e_k(N;\r))$ converge to those of a Gaussian random
 variable, and $\E(e_k(N;\r))$ does not depend on $\theta$, $e_k(N;\r)\to \E(e_k(N;\r))$.
 It remains to note that $\r^N_1,\dots,\r^N_{\min(N,M)}$ can be reconstructed as roots of polynomial
 $$
  \prod_{i=1}^{\min(N,M)} (x-\r^N_i)=\sum_{j=0}^{\min(N,M)} x^{\min(N,M)-j} (-1)^j e_j(N;\r),
 $$
 therefore, they converge to the roots of polynomial
 $$
  \sum_{j=0}^{\min(N,M)} x^{\min(N,M)-j} (-1)^j \E e_j(N;\r).\qedhere
 $$
\end{proof}
\begin{proof}[Another proof of Corollary \ref{cor_LLN_beta}]
There is another way to see that  as $\theta\to\infty$ the random vector $\r\in\St^M$ distributed
according to $\Pr^{\alpha;M;\theta}$ exhibits a law of large numbers. Indeed,
\eqref{eq_general_beta} implies
 that as $\theta\to\infty$ the numbers $\r^N_1,\dots,\r^N_{\min(N,M)}$ (for foxed $N$) concentrate
 near the vector $z_1<z_2<\dots<z_{\min(N,M)}$ which maximizes
 \begin{equation}
 \label{eq_density_maximization}
  \prod_{1\le i<j \le \min(N,M)} (z_j-z_i)
 \prod_{i=1}^{\min(N,M)} z_i^{\alpha/2} (1-z_i)^{|M-N|/2+1/2}
 \end{equation}
 The maximum of \eqref{eq_density_maximization} is known to be unique, and the minimizing
 configuration is precisely the set of \emph{roots} of Jacobi orthogonal polynomial of degree
 $\min(N,M)$ corresponding to the weight function $x^{\alpha-1} (1-x)^{|M-N|}$. This statement
 dates back to the work of T.~Stieltjes, cf.\ \cite[Section 6.7]{Szego}, \cite{K-LLN}.
\end{proof}

 Now we are ready to prove that $\sqrt{\theta}(\r^N_i-\jr^N_i)$ converges to a Gaussian vector. For
 any $K\ge 1$ consider the map
 $$
  \Phi^K: W_K \to \mathbb R^K:\quad \Phi^K(x_1,\dots,x_K)=(e_1(x_1,\dots,x_K),\dots,
  e_K(x_1,\dots,x_K)),
 $$
 where $W_K$ is the Weyl chamber of rank $K$
 $$
  W_K=\{x_1\le x_2\le\dots\le x_K\}.
 $$
 Observe that $\Phi^K$ is invertible ($x_1,\dots,x_K$ are reconstructed as roots of polynomial
 $\prod_{i=1}^K (x-x_i)$, whose coefficients are $(-1)^k e_k$), moreover, $\Phi^K$ is a (local)
 diffeomorphism in the neighborhood of every point in the interior of $W_K$. For $x_1<x_2<\dots<x_k$ let
 $\phi[x_1,\dots,x_K]$ denote the differential of $\Phi^{-1}$ near the point
 $(x_1,\dots,x_k)\rightleftarrows(e_1,\dots,e_k)$. $\phi[x_1,\dots,x_K]$ is a linear map which can
 be represented via $K\times K$ matrix, a straightforward computation gives an explicit formula for
this
 matrix.

 For small enough numbers $\Delta x_1$,\dots $\Delta x_k$ we have
 \begin{multline}
 \label{eq_linearization}
  (x_1+\Delta x_1,\dots, x_K+\Delta x_K)
  \\= (x_1,\dots,x_K) + \phi[x_1,\dots,x_K](\Delta
  e_1,\dots,\Delta e_K) + o\left(\sqrt{(\Delta e_1)^2+\dots+ (\Delta e_K)^2}\right),
 \end{multline}
 where $\Delta e_i$ are defined through
 $$
  (e_1(x_1,\dots,x_K)+\Delta e_1,\dots, e_k(x_1,\dots,x_K)+\Delta e_K)=\Phi^K(x_1+\Delta x_1,\dots, x_K+\Delta
  x_K).
 $$
 Now take in \eqref{eq_linearization} $K=\min(N,M)$, $x_i=\jr_i^N$, $\Delta x_i = \r^N_i -\jr_i^N$
  and send $\theta\to\infty$.  Corollary \ref{cor_LLN_beta} implies that the vectors $(\Delta x_1,\dots,\Delta
  x_K)$ and $(\Delta e_1,\dots,\Delta e_K)$ tend to zero vector in probability. Therefore, with
  probability tending to $1$ the estimate of the remainder in \eqref{eq_linearization} becomes uniform.
 Now Central Limit Theorem for $\sqrt{\theta}(\Delta e_1,\dots, \Delta e_K)$ implies through \eqref{eq_linearization} the Central
 Limit Theorem (in the sense of weak limits) for $\sqrt{\theta}(\Delta x_1,\dots, \Delta x_K)$.

\section{Heckman-Opdam hypergeometric functions.}

Let $\GTH_N$ denote the set of all decreasing $N$--tuples of non-negative reals. Our notation is
explained by the fact that the limiting transition below realizes this set as a scaling limit of
$\GTP_N$ from the previous sections.

\begin{definition}
 For any $r\in\GTH_N$ and $\theta>0$ the function $\HO_r(y_1,\dots,y_N; \theta)$ is defined through integral
representation:
\begin{multline}
\label{eq_integral_rep}
 \HO_r(y_1,\dots,y_N; \theta)=\frac{1}{\Gamma(\theta)^{N(N-1)/2}}\int_{\ell^1\prec \ell^2\prec\dots\prec \ell^N=r}
 \exp\left(\sum_{k=1}^N y_k \bigg(\sum_{i=1}^k \ell^k_i -\sum_{i=1}^{k-1} \ell^{k-1}_i\bigg) \right)\\ \times
\prod_{i<j} \left|\exp(-\ell_i^N)-\exp(-\ell_j^N)\right|^{1-\theta}
  \prod_{k=1}^{N-1} \Bigg(\prod_{1\le i<j\le k} \left|\exp(-\ell_i^k)-\exp(-\ell_j^k)\right|^{2-2\theta}\\ \times
 \prod_{a=1}^k \prod_{b=1}^{k+1}
\left|\exp(-\ell^k_a)-\exp(-\ell^{k+1}_b)\right|^{\theta-1} \prod_{i=1}^k \Big(\exp((\theta-1)
\ell^k_i)d\ell^k_i \Big) \Bigg),
\end{multline}
where the integration goes over the Lebesgue measure on $N(N-1)/2$ dimensional polytope
$\ell^1\prec \ell^2\prec\dots\prec \ell^N=r$,  each $\ell^k$ is $k$--dimensional vector
$\ell^{k}_1> \ell^k_2>\dots>\ell^k_k$, and all the coordinates interlace, i.e.\ $\ell^k_i >
\ell^{k-1}_i>\ell^k_{i+1}$ for all meaningful $k$ and $i$.
\end{definition}

We are interested in functions $\HO_r$ because they are limits of Macdonald polynomials
$P_\lambda(x_1,\dots,x_N; q,t)$ in  limit regime  of Theorem \ref{Theorem_convergence_to_Jacobi}.

\begin{proposition}
\label{prop_convergence_to_HO}
 Suppose that
\begin{equation}
\label{eq_limit_params}
 t=q^{\theta},\quad q=\exp(-\varepsilon),\quad \lambda = \lfloor \varepsilon^{-1}
 (r_1,\dots,r_N) \rfloor ,\quad
 x_i=\exp(\varepsilon y_i),
\end{equation}
where $r_1>r_2>\dots>r_N$. Then there exists a limit
$$
 \HO_r(y_1,\dots,y_N;\theta)= \lim_{\varepsilon\to 0} \varepsilon^{\theta N(N-1)/2}
 P_\lambda(x_1,\dots,x_N;q,t),
$$
and this limit is uniform over $r$ and $x$ belonging to compact sets.
\end{proposition}
\begin{proof}[Sketch of the proof]
 Induction in $N$. The combinatorial formula (i.e.\ branching rule, see \cite[Chapter VI]{M}) for Macdonald polynomials
yields
\begin{equation}
\label{eq_Mac_branching}
  P_\lambda(x_1,\dots,x_N;q,t)=\sum_{\mu\prec\lambda} \psi_{\lambda/\mu}(x_N)
  P_\mu(x_1,\dots,x_{N-1};q,t),
\end{equation}
where (with notation $f(u)=\frac{(tu;q)_\infty}{(qu; q)_\infty}$)
$$
 \psi_{\lambda/\mu}(x)=x^{|\lambda|-|\mu|}
f(1)^{N-1} \prod_{1\le i<j<N} f(q^{\mu_i-\mu_j}
 t^{j-i})
  \prod_{1\le i\le j<N}
\frac{f(q^{\lambda_i-\lambda_{j+1}}t^{j-i})}{f(q^{\mu_i-\lambda_{j+1}}t^{j-i}){f(q^{\lambda_i-\mu_j}t^{j-i})}}.
$$
Suppose that as $\eps\to 0$, $\mu = \varepsilon^{-1}
 (m_1,\dots,m_{N-1})$. Then by Lemma \ref{lemma_q_poch_conv} and \ref{lemma_q_Gamma_conv}
$$
 x^{|\lambda|-|\mu|}\sim \exp(y (|r|-|m|)),\quad f(1)\sim \frac{\varepsilon^{1-\theta}}{\Gamma(\theta)}, \quad
 f(q^{\mu_i-\mu_j}
 t^{j-i})\sim (1-\exp(-(m_i-m_j)))^{1-\theta},
$$
$$
\frac{f(q^{\lambda_i-\lambda_{j+1}}t^{j-i})}{f(q^{\mu_i-\lambda_{j+1}}t^{j-i}){f(q^{\lambda_i-\mu_j}t^{j-i})}}\sim
\left[\frac{1-\exp(-(r_i-r_{j+1}))}{\big(1-\exp(-(m_i-r_{j+1}))\big)\big(1-\exp(-(r_i-m_j))\big)}\right]^
{1-\theta}.
$$
Therefore,
$$
 \psi_{\lambda/\mu}(x)\sim g_{r/m}(y) \varepsilon^{(N-1)(1-\theta)},
$$
where
\begin{multline}
 g_{r/m}(y)= \frac{\exp(y (|r|-|m|))}{\Gamma(\theta)^{N-1}} \prod_{1\le i<j<N}
 (1-\exp(-(m_i-m_j)))^{1-\theta} \\ \times \prod_{1\le i\le j<N}
\left[\frac{1-\exp(-(r_i-r_{j+1}))}{\big(1-\exp(-(m_i-r_{j+1}))\big)\big(1-\exp(-(r_i-m_j))\big)}\right]^
{1-\theta}
\end{multline}
As $\eps\to 0$ the summation in \eqref{eq_Mac_branching} turns into a Riemannian sum (with step
$\eps$) for an integral. Therefore, omitting uniformity estimates,
\begin{equation*}
 \lim_{\varepsilon\to 0} \varepsilon^{\theta N(N-1)/2}
 P_\lambda(x_1,\dots,x_N;q,t)=\int_{m\prec r} g_{r/m}(y_N)
 \HO_{m}(y_1,\dots,y_{N-1};\theta). \qedhere
\end{equation*}
%
\end{proof}

There are several ways to link the functions $\HO_r$ to known objects. One way is through the fact
that these functions are analytic continuations (in index) of Jack polynomials
$J_\mu(\cdot;\theta)$ (see e.g.\ \cite[Chapter VI, Section 10]{M}) as seen from the formula
\begin{equation}
\label{eq_HO_and_Jack}
 \frac{\HO_r(-\mu_1-(n-1)\theta,\dots,-\mu_n;\theta)}{\HO_r(-(n-1)\theta,\dots,-\theta,0;\theta)}=
\frac{J_\mu(\exp(-r_1),\dots,\exp(-r_n);\theta)}{J_\mu(1,\dots,1;\theta)},
\end{equation}
which is a limit of the well-known argument-index symmetry relation, cf.\ \cite[Chapter VI,
(6.6)]{M}.
$$
 \frac{P_\lambda(q^{\mu_1}t^{n-1},\dots,q^{\mu_n}t^0;q,t)}{P_\lambda(t^{n-1},\dots,1;q,t)}=\frac{P_\mu(q^{\lambda_1}t^{
n-1},\dots,q^{\lambda_n}t^0;q,t)}{P_\mu(t^{n-1},\dots,1;q,t)}
$$
for Macdonald polynomials.

Another way is through the connection to the Calogero--Sutherland Hamiltonian:
\begin{proposition}
 For every $y_1,\dots,y_N\in\mathbb C$,  $\HO_r(y_1,\dots,y_N;\theta)$ is an eigenfunction of the differential operator
(acting in $r_i$'s)
\begin{equation}
\label{eq_Calogero}
 \sum_{i=1}^N \left(\frac{\partial}{\partial r_i}\right)^2 +
\sum_{i<j}\frac{\theta(1-\theta)}{2\sinh^2\left(\frac{r_i-r_j}{2}\right)}
\end{equation}
with eigenvalue $\sum_{i=1}^N (y_i)^2$.
\end{proposition}
\begin{proof}[Sketch of the proof]
 The Pieri formula for the Macdonald polynomials (see \cite[Chapter VI, Section 6]{M}) yields
\begin{multline}
\label{eq_Pierri}
 P_{\mu}(x_1,\dots,x_N;q,t) e_1(x_1,\dots,x_N)= P_{\mu}(x_1,\dots,x_N;q,t) (x_1+\dots+x_N)\\=\sum_{\lambda=\mu\cup(a,b)}
\prod_{i=1}^{a-1} \frac{(1-q^{\mu_i-\mu_a-1}
t^{a-i+1})(1-q^{\mu_i-\mu_a}t^{a-i-1})}{(1-q^{\mu_i-\mu_a} t^{a-i})(1-q^{\mu_i-\mu_a-1} t^{a-i})}
P_{\lambda}(x_1,\dots,x_N;q,t),
\end{multline}
where the summation goes over Young diagrams $\lambda$ which differ from $\mu$ by adding a single
box $(a,b)$ (in row $a$ and column $b$). In the limit regime of Proposition
\ref{Prop_convergence_to_HO_with_zeros}, $P_\mu$ turns into $\HO_m$, where $\mu_i=\lfloor m_i
\eps^{-1} \rfloor$, and when $\lambda=\mu\cup(a,b)$, $P_\lambda$ turns into
$$
\HO_m+\eps \frac{\partial}{\partial m_a} \HO_m+\frac{\eps^2}{2} \left(\frac{\partial}{\partial
m_a}\right)^2\HO_m+\dots.
$$
On the other hand,
$$
\frac{(1-q^{\mu_i-\mu_a} t^{a-i+1})(1-q^{\mu_i-\mu_a-1}t^{a-i+1})}{(1-q^{\mu_i-\mu_a}
t^{j-i})(1-q^{\mu_i-\mu_a-1} t^{a-i})}=1+ \eps^2 \theta(1-\theta)
\frac{\exp(m_a-m_i)}{(1-\exp(m_a-m_i))^2}+O(\eps^3)
$$
Therefore, up to terms of order $\eps^3$, the identity \eqref{eq_Pierri} gives (omitting
$(y_1,\dots,y_N)$ from the argument of $\HO_m$)
\begin{multline}
\label{eq_difference_relation}
 \left(N+\eps \sum_{i=1}^N y_i +\frac{\eps^2}{2} \sum_{i=1}^N (y_i)^2\right)\HO_m =\sum_{i=1}^N\left(1+\eps
\frac{\partial}{\partial m_i} \HO_m+\frac{\eps^2}{2} \left(\frac{\partial}{\partial
m_a}\right)^2\HO_m\right)\\+\eps^2\sum_{i<j}
\frac{\theta(1-\theta)\exp(m_j-m_i)}{(1-\exp(m_j-m_i))^2}\HO_m.
\end{multline}
Note that on the level of formal computation one could just compare the coefficient of $\eps^2$ in
both sides of \eqref{eq_difference_relation} and get the right answer. However, in order to give a
proof, one needs to subtract from \eqref{eq_difference_relation} certain operators which will
cancel low order terms. For that we also need another identity for Macdonald polynomials, which is
$$
 P_{\mu}(x_1,\dots,x_N;q,t) e_N(x_1,\dots,x_N)=P_\lambda(x_1,\dots,x_N;q,t) x_1\cdots x_N=P_{\lambda+1^N}(x_1,\dots,x_N;q,t).
$$
As $\eps\to 0$, up to the terms of order $\eps^3$ this gives
\begin{equation}
\label{eq_x15}
 \left(1+\eps \sum_{i=1}^N y_i + \frac{\eps^2}{2} \left(\sum_{i=1}^N
 y_i\right)^2\right)\HO_m =\HO_m+\eps \sum_{i=1}^N \frac{\partial}{\partial m_i} \HO_m +\frac{\eps^2}{2}
 \left(\sum_{i=1}^N \frac{\partial}{\partial m_i}\right)^2 \HO_m.
\end{equation}
Combining \eqref{eq_difference_relation} and \eqref{eq_x15} we conclude that the first order (in
$\eps$) of the operator of multiplication by
$$
 2\eps^{-2}\left(e_1-N -(e_N-1) + \frac{1}{2}(e_N-1)^2\right)
$$
gives
$$
 \left(\sum_{i=1}^N y_i^2 \right)\HO_m = \sum_{i=1}^N \left(\frac{\partial}{\partial m_i}\right)^2
 \HO_m +\sum_{i<j}
\frac{2\theta(1-\theta)\exp(m_j-m_i)}{(1-\exp(m_j-m_i))^2}\HO_m. \qedhere
$$
\end{proof}

In principle, neither \eqref{eq_HO_and_Jack}, nor \eqref{eq_Calogero} uniquely define the
functions $\HO_m$. Indeed, the analytic continuation in \eqref{eq_HO_and_Jack} might be not
unique, while some of the eigenvalues in \eqref{eq_Calogero} coincide. However, with additional
arguments one shows that $\HO_m(y_1,\dots,y_N)$ can be identified with \emph{Heckman--Opdam
hypergeometric functions} for root system of type $A$, see \cite{HO}, \cite{Op}, \cite{HS}.

Yet another link to well-understood objects is obtained through the limit $\theta\to\infty$.
Straightforward computations show that in this limit transition \eqref{eq_integral_rep} converges
to the Givental integral formula \cite{Gi} for $\mathfrak{gl}_N$--Whittaker functions. From a
different point of view the convergence of Heckman--Opdam hypergeometric functions to Whittaker
functions is established in \cite{Shi}.

\bigskip

If some of $r_i$'s coincide, then the integral in \eqref{eq_integral_rep} will be identically zero.
But if we do a rescaling then we still get a nontrivial object, which is also a limit of Macdonald
polynomials. One example is given by the following statement in which $M-N$ coordinates of $r$
become equal to zero.

\begin{proposition}
\label{Prop_convergence_to_HO_with_zeros}
 Let $M\ge N$ and suppose that
 \begin{equation}
\label{eq_limit_params_with_zeros}
 t=q^{\theta},\quad q=\exp(-\varepsilon),\quad \lambda = \lfloor h \varepsilon^{-1}
 (r_1,\dots,r_N,0,\dots,0)\rfloor,\quad
 x_i=\exp(\varepsilon y_i),
\end{equation}
where $r_1>r_2>\dots>r_N>0$ and $\lambda\in\GTP_M$. Then there exists a limit
$$
 \HO_r(y_1,\dots,y_M;\theta)= \lim_{\varepsilon\to 0} \varepsilon^{\theta (N(N-1)/2+N(M-N))}
 P_\lambda(x_1,\dots,x_M;q,t),
$$
and this limit is uniform over $r$ and $x$ belonging to compact sets.
\end{proposition}
\begin{proof}[Sketch of the proof]
 Induction in $M$. When $M=N$ the statement coincides with that of Proposition
 \ref{prop_convergence_to_HO}. For $M>N$ we again use the branching rule:
\begin{equation}
\label{eq_Mac_branching_2}
  P_\lambda(x_1,\dots,x_M;q,t)=\sum_{\mu\prec\lambda} \psi_{\lambda/\mu}(x_M)
  P_\mu(x_1,\dots,x_{M-1};q,t),
\end{equation}
where we note that since $\lambda$ has $N$ non-zero coordinates, $\mu$ also has $N$ non-zero
coordinates. In other words, the summation is $N$--dimensional. The asymptotics of
$\psi_{\lambda/\mu}(x)$ also changes because of zeros in $\lambda$ and $\mu$. We have
$$
 \psi_{\lambda/\mu}(x)=x^{|\lambda|-|\mu|}
f(1)^{N} \prod_{i<j<N+1} f(q^{\mu_i-\mu_j}
 t^{j-i})
  \prod_{i\le j<N+1}
\frac{f(q^{\lambda_i-\lambda_{j+1}}t^{j-i})}{f(q^{\mu_i-\lambda_{j+1}}t^{j-i}){f(q^{\lambda_i-\mu_j}t^{j-i})}}.
$$
Thus,
$$
 \psi_{\lambda/\mu}(x)\sim g_{r/m} \eps^{N(1-\theta)},
$$
where $g_{r/m}$ is as in Proposition \ref{prop_convergence_to_HO}, but with the agreement that $r$
is of length $N+1$ with last coordinate being zero. As $\eps\to 0$ the summation in
\eqref{eq_Mac_branching_2} turns into a Riemannian sum (with step $\eps$) for an integral.
Therefore,
\begin{equation}
\label{eq_HO_branching_2}
 \lim_{\varepsilon\to 0} \varepsilon^{\theta(N(N-1)/2+N(M-N))}
 P_\lambda(x_1,\dots,x_M;q,t)=\int_{m\prec r} g_{r/m}(y_M)
 \HO_{m}(y_1,\dots,y_{M-1};\theta,h),
\end{equation}
where $r$ is thought of as having length $N+1$ with last coordinate being zero and $m$ has length
$N$.
\end{proof}

Similarly one can pass to the limit in the Macdonald $Q$--polynomials (which differ from $P$ by
multiplication by an explicit constant see \cite[Chapter 6]{M}).

\begin{proposition}
\label{Prop_convergence_to_QHO_with_zeros}
 Let $M\ge N$ and suppose that
 \begin{equation}
\label{eq_limit_params_with_zerosQ}
 t=q^{\theta},\quad q=\exp(-\varepsilon),\quad \lambda = \lfloor \varepsilon^{-1}
 (r_1,\dots,r_N,0,\dots,0)\rfloor,\quad
 x_i=\exp(\varepsilon y_i),
\end{equation}
where $r_1>r_2>\dots>r_N>0$ and $\lambda$ is a signature of size $M$. Then there exists a limit
$$
 \QHO_r(y_1,\dots,y_M;\theta, h)= \lim_{\varepsilon\to 0} \varepsilon^{N(\theta-1)+\theta (N(N-1)/2+N(M-N))}
 Q_\lambda(x_1,\dots,x_M;q,t),
$$
this limit is uniform over $r$ and $x$ belonging to compact sets.
\end{proposition}

\bigskip

The functions $\HO_r$ inherit various properties of Macdonald polynomials. Let us summarize some
of those (we agree that $\HO$--function vanishes when some of the indexing coordinates coincide).

\begin{proposition} \label{prop_HO_propeties} Functions $\HO_r$ and $\QHO_r$ have the following properties:
\begin{enumerate}[I.]
\item $\HO_r$ and $\QHO_r$ are symmetric functions of its arguments.
\item Homogeneity:
 $$
 \HO_r(y_1+A,\dots,y_N+A;\theta)=\exp(A|r|) \HO_r(y_1,\dots,y_N;\theta),
 $$
 $$
 \QHO_r(y_1+A,\dots,y_N+A;\theta)=\exp(A|r|) \QHO_r(y_1,\dots,y_N;\theta).
 $$
\item Cauchy--type identity:
 Take $N$ parameters $a_1,\dots,a_N$ and $M$ parameters $b_1,\dots,b_M$ such that $a_i+b_j<0$ for
 all $i,j$. Then
 $$
  \int_{r\in \GTH_{\min(N,M)}} \QHO_r(a_1,\dots,a_N;\theta) \HO_r(b_1,\dots,b_M;\theta) \prod_{i=1}^{\min(N,M)} d r_i = \prod_{i,j}
  \frac{\Gamma(-a_i-b_j)}{\Gamma(\theta-a_i-b_j)}.
 $$
\item Principal specialization: Let $r\in\GTH_N$ and $M\ge N$, then
\begin{multline}
\label{eq_principal_HO}
 \HO_r(0,-\theta,\dots, (1-M)\theta;\theta)\\=\prod_{i<j}^{i\le N;\, j\le M}
\frac{\Gamma(\theta(j-i))}{\Gamma(\theta(j-i+1))} \prod_{i<j} (e^{-r_j}-e^{-r_i})^{\theta}
\prod_{i=1}^N (1-e^{-r_i})^{\theta(M-N)}
\end{multline}
and
\begin{multline}
\label{eq_principal_QHO}
 \QHO_r(0,-\theta,\dots, (1-M)\theta;\theta)\\=\frac{1}{(\Gamma(\theta)^N}\prod_{i<j}^{i\le N;\, j\le M}
\frac{\Gamma(\theta(j-i))}{\Gamma(\theta(j-i+1))} \prod_{i<j} (e^{-r_j}-e^{-r_i})^{\theta}
\prod_{i=1}^N (1-e^{-r_i})^{\theta(M-N)+(\theta-1)}.
\end{multline}
\item Difference operators:
For any subset $I\subset\{1,\dots,N\}$ define
$$
 B_I(y_1,\dots,y_N;\theta) = \prod_{i\in I} \prod_{j\not \in I} \frac{y_i -y_j-\theta}{y_i-y_j}.
$$
Define shift operator $T_i$ through
$$
 [T_i f](y_1,\dots,y_N) = f(y_1,\dots,y_{i-1}, y_i-1,y_{i+1},\dots,y_N).
$$
For any $k\le N$ define $k$th difference operator $\D^k_N$ acting on symmetric functions in
variables $y_1,\dots,y_N$ through
$$
\D^k_N = \sum_{|I|=k} B_I(y_1,\dots,y_N;\theta) \prod_{i\in I} T_i.
$$
Then for any $r\in\GTH_n$ with $n\le N$ we have
$$
 \D^k_N \HO_r(y_1,\dots,y_N;\theta) = e_k(\exp(-r_1),\dots,\exp(-r_n), \, \underbrace{1,\dots, 1}_{N-n}\,)
 \HO_r(y_1,\dots,y_N;\theta),
$$
where $e_k$ is $k$th elementary symmetric polynomial (in $N$ variables).
\end{enumerate}
\end{proposition}
{\bf Remark. } The difference operators $\D^k_N$ of property V provide another way to identify
functions $\HO_r$ with Heckman--Opdam hypergeometric functions as shown in
\cite{Cherednik-Harish}.
\begin{proof}
 All the properties are straightforward limits of similar statements for Macdonald polynomials. For
 instance, III is a limit of the Cauchy identity for Macdonald polynomials
\begin{equation}
\label{eq_Cauchy_Mac}
 \sum_{\ell(\lambda)\le N} P_\lambda(\alpha_1,\dots,\alpha_N;q,t) Q_\lambda(\beta_1,\dots,\beta_M;q,t) =
 \prod_{i,j} \frac{(t\alpha_i \beta_j;q)_{\infty}}{(\alpha_i\beta_j;q)_{\infty}},
\end{equation}
IV is a limit of the evaluation formula for Macdonald polynomials
\begin{multline*}
 P_\lambda(1,\dots,t^{M-1};q,t)=t^{\sum (i-1)\lambda_i} \prod_{i<j\le M} \frac{(q^{\lambda_i-\lambda_j}
 t^{j-i};q)_\infty} {(q^{\lambda_i-\lambda_j}
 t^{j-i+1};q)_\infty} \cdot \frac{(
 t^{j-i+1};q)_\infty} {(
 t^{j-i};q)_\infty}
 \\= t^{\sum (i-1)\lambda_i} \prod_{i<j\le N} \frac{(q^{\lambda_i-\lambda_j}
 t^{j-i};q)_\infty} {(q^{\lambda_i-\lambda_j}
 t^{j-i+1};q)_\infty} \prod_{i=1}^N \prod_{j=N+1}^M \frac{(q^{\lambda_i}
 t^{j-i};q)_\infty} {(q^{\lambda_i}
 t^{j-i+1};q)_\infty}
 \prod_{i<j}^{i\le N;\, j\le M} \frac{(
 t^{j-i+1};q)_\infty} {(
 t^{j-i};q)_\infty}
\end{multline*}
and formula
$$
 \frac{Q_\lambda}{P_\lambda}=b_\lambda= \prod_{1\le i\le j \le \ell(\lambda)}
 \frac{f(q^{\lambda_i-\lambda_j} t^{j-i})}{f(q^{\lambda_i-\lambda_{j+1}} t^{j-i})}.
$$
 Finally, V is a limit of the Macdonald difference operators.
\end{proof}





\section*{Acknowledgements}

The authors are very grateful to I.~Corwin, A.~Edelman and G.~Olshanski for many
invaluable discussions, to A.~Ahn for pointing out an inaccuracy in our description
of $\Omega$, and to anonymous referees for reading the manuscript and many useful
suggestions.

 A.~B.
was partially supported by the NSF grant DMS-1056390. V.~G. was partially supported by RFBR-CNRS
grant 11-01-93105.



\frenchspacing
\bibliographystyle{plain}

\end{document}